\documentclass[11pt,letterpaper,english]{smfart}
\usepackage[utf8]{inputenc}
\usepackage[english,french]{babel}
\usepackage{smfthm}
\usepackage{mathrsfs}
\usepackage{graphics}
\usepackage{hyperref} 
\usepackage{amssymb, amsmath, mathabx}
\usepackage{lmodern}
\usepackage[all]{xy}
\usepackage{multicol}
\usepackage{color}
\usepackage[shortlabels]{enumitem}
\usepackage{scalerel}
\usepackage{tikz-cd}
\makeindex
\DeclareUnicodeCharacter{00A0}{\relax}
\allowdisplaybreaks

\author{Antoine Ducros}
\address{Sorbonne Université, Université Paris-Diderot, CNRS, Institut de Mathématiques de Jussieu-Paris
Rive Gauche, Campus Pierre et Marie Curie, case 247, 4 place Jussieu, 75252 Paris cedex 5, France}
\email{antoine.ducros\at imj-prg.fr}
\urladdr{https://webusers.imj-prg.fr/$\sim$antoine.ducros/}
\author{Ehud Hrushovski}
\address{Mathematical Institute,
University of Oxford,
Andrew Wiles Building,
Radcliffe Observatory Quarter
Woodstock Road, 
Oxford
OX2 6GG, UK}
\email{Ehud.Hrushovski\at maths.ox.ac.uk}
\urladdr{https://www.maths.ox.ac.uk/people/ehud.hrushovski}
\author{François Loeser}
\address{Institut universitaire de France, Sorbonne Université, Institut de Mathématiques de Jussieu-Paris
Rive Gauche CNRS, Campus Pierre et Marie Curie, case 247, 4 place Jussieu, 75252 Paris cedex 5, France
}
\email{francois.loeser@imj-prg.fr}
\urladdr{https://webusers.imj-prg.fr/$\sim$francois.loeser/}

\title{Non-archimedean integrals as limits of
complex integrals}
\NumberTheoremsAs{subsection}
\SwapTheoremNumbers

\newcommand{\eg}{e.\@g.\@}

\newcommand{\ie}{i.\@e.\@}

\newcommand{\loccit}{loc.\@~cit.\@}

\newcommand{\resp}{resp.\@~}

\newcommand{\Prop}{Prop.\@~}

\newcommand{\abs}[1]{\mathopen|#1\mathclose|}

\newcommand{\an}{^{\mathrm{an}}}

\newcommand{\da}{\d\!\arg}
\newcommand{\Arg}{\mathop{\mathrm{Arg}}}
\newcommand{\dArg}{\d\!\Arg}
\newcommand{\daa}{\frac{\da}{2\pi}}
\newcommand\dab[1]{\frac{\d\!\arg{#1}}{2\pi}}
\newcommand{\dc}{\d''}
\newcommand{\di}{\d'}
\newcommand{\dl}{\d\!\loga}
\newcommand{\dll}{\d\!\log}
\newcommand{\ds}{\d^\sharp}

\newcommand{\gm}{\mathbf G_{\mathrm m}}
\newcommand{\gpm}{^{\times}}

\newcommand{\inv}{^{-1}}
\newcommand{\loga}{\mathop{\mathrm{Log}}}
\newcommand{\logb}{\mathop{\mathrm{Log}}\nolimits_\flat}

\newcommand{\std}{\mathrm{std}}

\newcommand{\val}[1]{\abs{#1}_{\flat}}

\def\llp{\mathopen{(\!(}}

\def\rrp{\mathopen{)\!)}}

\newcommand{\A}{\mathbf A}
\newcommand{\C}{\mathbf C}

\newcommand{\N}{\mathbf Z_{\geq 0}}
\renewcommand{\P}{\mathbf P}
\newcommand{\Q}{\mathbf Q}
\newcommand{\R}{\mathbf R}
\newcommand{\Z}{\mathbf Z}

\renewcommand{\d}{\mathrm d}
\renewcommand{\phi}{\varphi}
\renewcommand{\epsilon}{\varepsilon}
\renewcommand{\leq}{\leqslant}
\renewcommand{\geq}{\geqslant}

\long\def\FL#1{{\color{black} #1}}
\renewcommand{\labelitemi}{$\bullet$}

\renewcommand{\theequation}{\alph{equation}}
\SetEnumerateShortLabel{1}{\textnormal{(\arabic{enumi})}}
\SetEnumerateShortLabel{i}{\textnormal{(\roman{enumi})}}
\SetEnumerateShortLabel{a}{\textnormal{(\alph{enumi})}}
\SetEnumerateShortLabel{A}{\textnormal{(\Alph{enumi})}}
\SetEnumerateShortLabel{2}{\textnormal{(\arabic{enumii})}}
\SetEnumerateShortLabel{j}{\textnormal{(\roman{enumii})}}
\SetEnumerateShortLabel{b}{\textnormal{(\arabic{enumi}\alph{enumii})}}
\SetEnumerateShortLabel{c}{\textnormal{(\alph{enumii})}}
\SetEnumerateShortLabel{B}{\textnormal{(\Alph{enumii})}}

\begin{document}
\begin{abstract}
We explain how  non-archimedean integrals considered by Chambert–Loir and Ducros
naturally arise in asymptotics of families of complex integrals. To perform this analysis we
work over a non-standard model of the field of complex numbers, which is endowed at the
same time with an archimedean and a non-archimedean norm. Our main result states the
existence of a natural morphism between bicomplexes of archimedean and non-archimedean
forms which is compatible with integration. 
\end{abstract}
\maketitle

\tableofcontents

\section{Introduction}
\subsection{}A. Chambert-Loir and A. Ducros recently
developed 
a full-fledged theory of real valued $(p,q)$-forms and currents on Berkovich spaces which is
an analogue of the theory of differential forms on complex spaces \cite{chambertloir-d2012}. 
Their  forms are constructed as pullbacks under tropicalisation maps of the ``superforms''
introduced by  Lagerberg
\cite{lagerberg12}. 
They are able to integrate compactly supported $(n,n)$-forms for $n$ the dimension of the ambient space (the output being a real number)
and they obtain versions
of the  Poincar\'e-Lelong Theorem and the Stokes Theorem in this setting.     Their work is guided throughout by  an analogy with complex analytic geometry.  
The aim of the present work is to convert the analogy into a direct connection, showing  how the non-archimedean 
theory appears as  an asymptotic limit %over %FL
of one-parameter families of complex (archimedean) forms and integrals.

One way to view a family of complex varieties as degenerating to a non-archimedean space is to consider
the hybrid spaces first introduced by V. Berkovich
\cite{berk09}
to provide a non-archimedean interpretation of the weight zero part of the mixed Hodge structure on the cohomology of a proper complex variety.
For some other recent applications of hybrid spaces, see \cite{bj}
\cite{favre}
\cite{umm}.

The approach we follow in this paper is somewhat different. We
work over an  algebraically closed  field $C$ containing $\C$,  which is a degree 2 extension of a real closed field
$R$ containing $\R$ and is endowed at the same time with an archimedean non-standard
norm $\abs \cdot: C \to R_+$ and with a non-archimedean norm
$\val \cdot: C \to \R_+$
that essentially encapsulates the ``order of magnitude"
of $\abs \cdot$ with respect to
a given infinitesimal element which should be thought
of as a ``complex parameter tending to zero".
This presents the advantage of working on spaces that
have at the same time archimedean and non-archimedean features and allows to be able
to compare directly
archimedean constructions and their non-archimedean counterparts.
The  fields $R$ and  $C$ are constructed using ultrapowers. 
 The field $R$   was   introduced by A. Robinson in \cite{robinson1}, 
with the explicit hope that it will be useful for asymptotic analysis; 
 see also \cite{robinson2}. 
It was brought to good use in \cite{kramer-tent} following  the fundamental work of 
van den Dries and Wilkie \cite{vddw}, who have reformulated Gromov's theory of
asymptotic cones of metric spaces \cite{gromov} using 
ultrapowers.

A long-term motivation for our work is the famous conjecture by Kontsevich and Soibelman
\cite{koso} \cite{ks} relating large scale complex geometry and 
non-archimedean geometry.
Roughly speaking the conjecture describes the Gromov-Hausdorff limit of a family of complex Calabi-Yau varieties with maximal degeneration
in terms of  
non-archimedean geometry. We refer to 
\cite{gross1}\cite{gross2}\cite{gross3}\cite{oda1}\cite{oda2}\cite{sustretov} for some recent results in that direction.
Note that our results involve a renormalization in powers of $\log \vert t \vert$ which corresponds to what appears naturally when considering volume forms on Calabi-Yau varieties with maximal degeneration. From a model theoretic perspective, this is related to considering measures on certain definable sets over the value group, in  contrast to \cite{otero-b},
where measures are reduced to the residue field.

\subsection{}
Before going further, it may be useful to provide the flavor of our main results on a very elementary example.
Let $\varphi: \R \to \R$ be a smooth function with compact support.
Consider the complex $(1,1)$-form
\[\omega_t = - \frac 1{\log \abs{t}} \, 
\varphi \left(- \frac{\log \abs{z (z-t)}}{\log \abs{t}}\right)
\dll \abs{z}\wedge \dab{z}\]
on
$\P_1$, depending on the complex parameter $t$.
Fix a real number $K >1$.
One may write
\[\int_{\P_1(\C)} \omega_t
=
I_1 + I_2 +I_3\]
with 
\[I_1 = \int_{\abs{z} \leq \abs{t}/K} \omega_t, \, 
I_2 = \int_{\abs{t}/K \leq \abs{z} \leq K \abs{t}} \omega_t,
\;  \text{and} \;  I_3 = \int_{\abs{z} \geq K \abs{t}} \omega_t.\]
Using direct explicit computations, one may check
that
\[\lim_{t \to 0} I_1 =  \int_{x \leq -1} \varphi (x-1) \d x,
\, \lim_{t \to 0} I_2 = 0,
\;  \text{and} \;  \lim_{t \to 0} I_3 =  \int_{x \geq -1} \varphi (2x) \d x,\]
from which one deduces the equality
\[\lim_{t \to 0} \int_{\P_1(\C)} \omega_t
=
  \int_{x \leq -1} \varphi (x-1) \d x + 
 \int_{x \geq -1} \varphi (2x) \d x.\]
Quite remarkably, the right hand side of that equality admits a non-archimedean interpretation.
Indeed, consider the field
of Laurent series
 $\C\llp t \rrp$, fix $\tau\in (0,1)$, and endow $\C\llp t \rrp$ with the $t$-adic norm $\val{\:}$ normalized by  $\val{t} = \tau$.
On
 the
Berkovich analytification $\P_1\an$ of $\P_1$ over $\C\llp t \rrp$ one can consider the $(1,1)$-form
\[\omega_\flat = - \frac 1{\log \val{t}} \, 
\varphi \left(- \frac{\log \val{z (z-t)}}{\log \val{t}}\right)
\di \log\val{z} \wedge \dc \log \val{z}\]
in the sense of 
Chambert-Loir and Ducros \cite{chambertloir-d2012}.
Furthermore, the integral in the sense of 
Chambert-Loir and Ducros of the form $\omega_\flat$ on $\P_1\an$
is given by
\[\int_{\P_1\an} \omega_{\flat}
=
\int_{x \leq -1} \varphi (x-1) \d x + 
\int_{x \geq -1} \varphi (2x) \d x,\]
since the support of $\omega_\flat$ is contained in the standard skeleton
$(0, \infty)$ of $\mathbf{G}_m\an$, and the  function $z$ is of degree $1$ at each point of this skeleton.
Therefore we finally deduce the equality 
\[\lim_{t \to 0} \int_{\P_1(\C)} \omega_t
= 
\int_{\P_1\an} \omega_{\flat},\]
a very special case of our Corollary \ref{epilogue}.
We can already see here an instance of a general feature which will be exploited in our proof of the general case:
asymptotically the complex integrals we consider concentrate on the support of the correponding non-archimedean forms.
This support is piecewise polyhedral and only the faces of maximal dimension provide a non-zero contribution to the limit.
In general, Chambert-Loir and Ducros integrals involve also degrees over these faces, see  \ref{9.1.11} for an explanation how these relate
to the number of sheets of a complex \'etale morphism.

\subsection{}  
Let us now sketch the construction of the non-standard  ``asymptotic'' field $C$. We fix a non-principal ultrafilter $\mathcal{U}$ on $\C$
containing all  the neighbourhoods of the origin
(otherwise said, $\mathcal U$ converges to $0$)
and consider the ultrapowers
$
^*\C = \prod_{t \in \C^{\times}} \C / \mathcal{U}
$
and
$
^*\R = \prod_{t \in \C^{\times}} \R / \mathcal{U}.
$
We say an element $(a_t)$ in
$^*\C$, resp. $^*\R$, is $t$-bounded if
for some positive integer $N$, $\vert a_t \vert \leq \vert t \vert^{-N}$
along $\mathcal{U}$
(that is, the set of indices $t$ for which this inequality holds belongs to $\mathcal U$). Similarly, it is said to
be $t$-negligible
if for every positive integer $N$, 
$\vert a_t \vert \leq \vert t \vert^{N}$
along $\mathcal{U}$. 
The set of $t$-bounded elements in 
$^*\C$, resp. $^*\R$,
is a local ring which we denote by $A$, resp. $A_{\mathrm{r}}$,
with maximal ideal
the subset of $t$-negligible elements which we denote by 
$\mathfrak{M}$, resp. $\mathfrak{M}_{\mathrm{r}}$.
We now set 
$C := A / \mathfrak{M}$ and
$R := A_{\mathrm{r}} / \mathfrak{M}_{\mathrm{r}}$.
The field $R$ is a real closed field and $C \simeq R (i)$ is algebraically closed.
The norm
$\vert \cdot \vert \colon ^*\C\to {^*\R_{\geq0}}$
induces an $R$-valued norm
$\vert \cdot \vert : C \to R_{\geq0}$.
%Furthermore,
%one can endow
%$C$ with a real-valued non-archimedean norm
%$\vert \cdot \vert_\flat : C \to \R_{\geq0}$
%as follows.
%For any $z \in C^\times$, one checks that the norm of 
%$\frac{\log \vert z \vert}{\log \vert t \vert}$ is bounded by some positive real number in $\R$.
%One can thus consider its 
%standard part $\alpha = \std \Bigl(\frac{\log \vert z \vert}{\log \vert t \vert}\Bigr) \in \R$.
%Fixing $\tau \in (0, 1) \subset \R$,
%one sets 
%$\vert z \vert_\flat := \tau^{\alpha}$, so that 
%$\vert z \vert_\flat = \vert t \vert_\flat^{\alpha}$.
%With this non-archimedean norm the  field $C$ is complete (even spherically complete,  \FL{cf. \cite{luxemburg})}.

\subsection{}\label{ss-smooth-bigR}
Any usual smooth function $\phi\colon U\to \R$ defined on some semi-algebraic open subset $U$ of $\R^n$ induces formally a map $U(^*\R)\to {^*\R}$ which is still denoted by $\phi$.
Allowing ourselves to compose these functions (which arise from \emph{standard} smooth functions) with polynomial maps (which might have non-standard coefficients), we define for every smooth, separated $^*\R$-scheme $X$ of
finite type a sheaf of so-called smooth functions for the
(Grothendieck)
semi-algebraic topology
on $X(^*\R)$, which we denote by  $\mathscr C^\infty_X$.
The natural inclusion map from $X(^*\R)$ into the (underlying set of) the scheme $X$ underlies a morphism of locally ringed
sites $\psi \colon (X(^*\R),\mathscr C^\infty_X)\to (X,\mathscr O_X)$, and we can define the sheaf of smooth $p$-forms on $X(^*\R)$ by $\mathscr A^p_X:=\psi^*\Omega^p_{X/^*\R}$.
One has for every $p$
a natural differential $\d \colon \mathscr A^p_X\to \mathscr A^{p+1}_X$. 
We now assume $X$
 is of pure dimension $n$, and that $X(^*\R)$ is oriented
 (the notion of an orientation of a variety makes sense over an arbitrary real closed field, see \ref{def-semialg-topology}). 
%To develop a theory of integration of smooth $n$-forms some care is required since $^*\R$ is not locally compact.
%In our setting this issue may be overcome by using the notion of definable compactness for semi-algebraic subsets 
%\eh{This sentence still seems to me to belong to 1.4 rather than 1.3.  The class of the 
%sequence $(\int_{E_t}\omega_t)_t$ in $^*\R$ is in any case defined, provided $E_t$ is   compact a.e.
%It seems that in this paragraph no bound is claimed on this class, and teh existence of a value in $^*R$ is soft.}
%\cite{pest}.
%OK j'ai supprimé (FL)
Let $\omega$
be a smooth $n$-form on some semi-algebraic open subset $U$ of $X(^*\R)$, 
and let $E$ be a semi-algebraic subset of $U$ whose closure in $U$ is definably compact. 
Choosing a description of $(X,U,\omega,E)$ through a ``limited family" $(X_t,U_t,\omega_t,E_t)_t$,
it is possible to define 
the 
integral $\int_E\omega$ as the class of 
the sequence $(\int_{E_t}\omega_t)_t$ in $^*\R$.

\subsection{}\label{ss-smooth-r}
%Let $U$ be a semi-algebraic open subset of $\R^m$, 
%and let $\phi\colon U \to \R$ 
%be a smooth function. We shall say that $\phi$ is $\mathscr P^k$ for some $k\in \N\cup\{\infty\}$
%if for every integer $\ell\leq k$ there exists some polynomial $P_\ell$ such that
%$\abs{\mathrm D^{(\ell)}\phi}\leq \abs{P_\ell}$ on $U$. 
We now move from $^*\R$ to $R$, seeking to show that smooth functions, smooth forms and their integrals
remains well-defined on $R$.

Let $\phi\colon U\to \R$ be a usual smooth
function
defined on some semi-algebraic open subset $U$ of $\R^n$. 
Under some boundedness assumptions on $\phi$ (which are 
for instance automatically fulfilled if $\phi$ is compactly supported, or more generally if all its derivatives are polynomially bounded), the induced function $\phi \colon U(^*\R)\to {^*\R}$
in turn induces a map $U(R)\to R$, which we again denote by $\phi$.

For instance, the map $\log \abs \cdot$ from 
$\C^\times\simeq \R^2\setminus\{(0,0)\}$ 
is smooth and satisfies the boundedness conditions alluded
to above; it thus induces a map
$\log \abs \cdot \colon C^\times \to R$, which
enables us to endow 
the field $C$ with a real-valued non-archimedean norm
$\vert \cdot \vert_\flat : C \to \R_{\geq0}$
as follows.
For any $z$ belonging to $C^\times$, one checks that the norm of 
$\frac{\log \vert z \vert}{\log \vert t \vert}$ is bounded by some positive real number in $\R$.
One can thus consider its 
standard part $\alpha = \std \Bigl(\frac{\log \vert z \vert}{\log \vert t \vert}\Bigr) \in \R$.
Fixing $\tau \in (0, 1) \subset \R$,
one sets 
$\vert z \vert_\flat := \tau^{\alpha}$, so that 
$\vert z \vert_\flat = \vert t \vert_\flat^{\alpha}$.
With this non-archimedean norm the  field $C$ is complete (even spherically complete,  cf. \cite{luxemburg}).

We repeat the procedure used in \ref{ss-smooth-bigR}: allowing ourselves to compose the functions defined at the beginning
of \ref{ss-smooth-r} (which arise from \emph{standard} smooth functions) with polynomial maps (which might have non-standard coefficients), we define for every smooth, separated $R$-scheme $X$ of
finite type a sheaf of so-called smooth functions for the
(Grothendieck)
semi-algebraic topology
on $X(R)$, which we denote by  $\mathscr C^\infty_X$.
There is a natural morphism
of locally ringed
sites $\psi \colon (X(R),\mathscr C^\infty_X)\to (X,\mathscr O_X)$. One then sets $\mathscr A^p_X:=\psi^*\Omega^p_{X/R}$
and one has for every $p$
a natural differential $\d \colon \mathscr A^p_X\to \mathscr A^{p+1}_X$. 

Assume now $X$ is of pure dimension $n$ and oriented.
A substantial part of Section \ref{section3} is devoted to the construction of an $R$-valued 
integration theory on  $X(R)$.
\begin{prop}
Integration theory on $X (A_{\mathrm{r}})$ descends to 
$X (R)$.
\end{prop}
Namely, to a semi-algebraic subset $K$ of $X(R)$, with definably compact definable closure,
and a smooth $n$-form $\omega$ on a semi-algebraic
neighborhood of $K$ in $X(R)$, we assign
an integral $\int_K \omega$ which is an element of $R$.
This is achieved 
in \ref{subsec-integ} by reducing to the case when 
 $X$ is liftable.   %that is, when
% there exists a smooth affine $A_{\mathrm r}$-scheme $\mathscr X$, an isomorphism $\mathscr X_R\simeq X$, and
%$n+1$ regular functions $f_1,\ldots, f_{n+1}$ on $\mathscr X$  defining an
%immersion $\mathscr X\hookrightarrow \A^{n+1}_{A_{\mathrm r}}$ such that $(f_1,\ldots,f_n)\colon \mathscr X\to \A^n_{A_{\mathrm r}}$ is étale. 
Independence from the lifting follows from the fact, proved in Proposition \ref{prop-key-integration}, 
that the integrals obtained from two different liftings coincide up to a 
$t$-negligible element. A preliminary key statement in that direction is provided by 
Proposition \ref{prop-equiv-negli}
which states that if $D$ is a semi-algebraic subset of $(^*\R)^n$ contained in $A_{\mathrm r}^n$,
the volume of $D$ is $t$-negligible if and only if  the image of $D$ in $R^n$ through the reduction map is of dimension $\leq n-1$.

Assume that $X$ is a  smooth $C$-scheme of finite type and of pure dimension $n$.
One defines similarly the integral $\int_K \omega$ of a  complex-valued $(n,n)$-form $\omega$  defined
in a semi-algebraic neighborhood of a semi-algebraic
subset $K$ of $X(C)$, assuming  that there exists a semi-algebraic subset $K'$ of $K$ with definably compact closure such that $\omega$ vanishes on $K\setminus K'$. 

\begin{enonce}[remark]{Remark}
Note that for an arbitrary real closed field $S$ one cannot hope for a reasonable integration theory with values in $S$.
Indeed, let  for instance $S$ be the algebraic closure of $\mathbf Q$ in $\R$. Then there is no such reasonable integration theory on $S$, otherwise $\pi =\int_{x^2+y^2\leq 1}\d x\wedge \d y$ would belong to $S$.
\end{enonce}

\subsection{}
Fix a smooth $C$-scheme $X$ of finite type and pure dimension $n$, and set $\lambda := - \log \vert t \vert$.
In this text we define two Dolbeault-like complexes  $\mathsf A^{p,q}$ and $\mathsf B^{p,q}$. Informally $\mathsf A^{p,q}$ and
$\mathsf B^{p,q}$ should be thought of as living on $X(C)$ and $X\an$ respectively. But since we want to be able to compare them in some sense, we need that they be defined on the same site; this is the reason why we have chosen to define them as complexes of sheaves on the 
\emph{Zariski} site of $X$.

\subsubsection{The non-standard archimedean complex}
Let us start
with $\mathsf A^{p,q}$. We will explain what would be the most natural definition, why it is not convenient for our purpose, and what the actual definition is.

\paragraph{}\label{apq-fake-definition}
Basically, we would like a section of $\mathsf A^{p,q}$ on a given Zariski-open subset $U$
of $X$ to be a differential form on $U(C)$ which is locally
for the semi-algebraic topology on $X(C)$ of the form 
\[\omega=\frac 1{\lambda^p}\sum_{I,J} \phi_{I,J}\left(\frac{\log \abs{f_1}}\lambda,\ldots, \frac{\log \abs{f_m}}\lambda\right)
\dll \abs{f_I}\wedge \dArg{f_J}\]
where  $I$, \resp $J$, runs through the set of subsets of $\{1,\ldots, m\}$
of cardinality $p$, \resp  $q$, where the $f_i$ are regular invertible functions, 
$\dll \abs{f_I}$ stands for
the wedge product $\dll \abs{f_{i_1}}\wedge\ldots\wedge\dll \abs{f_{i_p}}$ if $i_1<i_2<\ldots<i_p$ are the elements
of $I$, and  $\dArg{f_J}$ stands for the wedge product $\daa {f_{j_1}}\wedge\ldots\wedge\daa {f_{j_q}}$ if $j_1<j_2<\ldots<j_q$ are the elements
of $J$. 

\paragraph{}
But it would be difficult to use the definition suggested in \ref{apq-fake-definition}, 
because the general forms described therein do not have non-archimedean counterparts, since there is no natural way to turn the implicit
semi-algebraic covering of $U(C)$ in their definition into an open covering of $U\an$; hence we will take a slightly more restrictive definition, albeit flexible enough for our purpose.

We thus define a section of $\mathsf A^{p,q}$ on a Zariski-open subset $U$
of $X$ is a differential form on $U(C)$ 
that is locally \emph{for the Zariski-topology of $U$} of the form
\[\omega=\frac 1{\lambda^p}\sum_{I,J} \phi_{I,J}\left(\frac{\log \abs{f_1}}\lambda,\ldots, \frac{\log \abs{f_m}}\lambda\right)
\dll \abs{f_I}\wedge \dArg{f_J}\]
where $(f_1,\ldots, f_m)$ are regular  functions
(but they are not assumed to be invertible), where $I$, \resp $J$, is running through the set of subsets of
$\{1,\ldots, n\}$ of cardinality $p$, \resp $q$,
%(with  $\dll \abs{f_I}$ standing for
%the wedge product $\dll \abs{f_{i_1}}\wedge\ldots\wedge\dll \abs{f_{i_p}}$ if $i_1<i_2<\ldots<i_p$ are the elements
%of $I$, and similarly for $\dArg{f_J}$), 
and where each function $\phi_{I,J}$ is defined on a suitable subset of $(\R\cup\{-\infty\})^m$ 
and satisfies some technical conditions which we explain now. Let $x=(x_1,\ldots, x_m)$
be a point of $(\R\cup\{-\infty\})^m$ and
let $K$ denote the set of indices $i$ such that
$x_i=-\infty$. Then: 
\begin{itemize}
\item [$\diamond$] around $x$
the function $\phi_{I,J}$ only depends on the $x_i$ for $i\notin K$, and is smooth as a function of the latter; 
\item[$\diamond$] the function $\phi_{I,J}$ even \textit{vanishes} around $x$ as soon as $K$ intersects $I\cup J$.
\end{itemize}
It is clear that sections of $\mathsf A^{p,q}$ admit 
a local description as in \ref{apq-fake-definition}.
Note that if $K=\emptyset$ the only requirement is for $\phi_{I,J}$ to 
 be smooth around $x$, and that our second condition ensures that $\dll \abs{f_i}$ or $\dab {f_i}$
 can actually appear only around points at which $f_i$ is invertible (which is necessary for integrating such a form %FL
when $p=q=n$).

There exist natural differentials $d: \mathsf A^{p,q} \to \mathsf A^{p+1,q}$ and
$\ds : \mathsf A^{p,q} \to \mathsf A^{p,q+1}$ mapping respectively
a 
form 
\[\frac 1{\lambda^p}\phi\left(\frac{\log \abs{f_1}}\lambda,\ldots, \frac{\log \abs{f_m}}\lambda\right)
\dll\abs{f_I}\wedge \dArg{f_J}\]
to 
\[\frac 1{\lambda^{p+1}}\sum_{1\leq i\leq m}\frac{\partial \phi}{\partial x_i}\left(\frac{\log \abs{f_1}}\lambda,\ldots,
\frac{\log \abs{f_m}}\lambda\right)\dll \abs{f_i}\wedge \dll \abs{f_I}\wedge \dArg{f_J}\]
and to 
\[\frac 1{\lambda^{p}}\sum_{1\leq i\leq m}\frac{\partial \phi}{\partial x_i}\left(\frac{\log \abs{f_1}}\lambda,\ldots,
\frac{\log \abs{f_m}}\lambda\right)\dab{f_i}\wedge \dll \abs{f_I}\wedge \dArg{f_J}.\]
Here the map $d$ is the usual differential, and $d^\sharp$ is 
designed to switch modulus and argument, see 
\ref{def-d-sharp}; it turns out to be analogous to the operator $\mathrm d^{\mathrm c}$ of complex analytic geometry.

\subsubsection{The non-archimedean complex}
We are now going to
describe $\mathsf B^{p,q}$. 
Set $\lambda_\flat := - \log \vert t \vert_\flat$.

\paragraph{}\label{bpq-fake-definition}
Basically, we would like a section of $\mathsf B^{p,q}$ on a given Zariski-open subset $U$
of $X$ to be a differential form on $U\an$
in the sense of \cite{chambertloir-d2012}
which is locally
on $U\an$ of the form 
%AD
\[\frac 1{\lambda_\flat^p}\sum_{I,J} \phi_{I,J}\left(\frac{\log \val{f_1}}{\lambda_\flat},\ldots, \frac{\log \val{f_m}}{\lambda_\flat}\right)
\di \log \val{f_I}\wedge \dc \log \val{f_J}\]
where  $I$, \resp $J$, runs through the set of subsets of $\{1,\ldots, m\}$
of cardinality $p$, \resp  $q$, where the $f_i$ are regular invertible functions and where $\di \log \val{f_I}$ standing for
the wedge product $\di \log \val{f_{i_1}}\wedge\ldots\wedge\di \log\abs{f_{i_p}}$ if $i_1<i_2<\ldots<i_p$ are the elements
of $I$, and similarly for $\dc \log$. 

\paragraph{}
But by analogy with $\mathsf A^{p,q}$, we shall rather define a section of $\mathsf B^{p,q}$ on a Zariski-open subset $U$
of $X$ as a differential form on $U\an$ 
that is locally \emph{for the Zariski-topology of $U$} of the form
%AD 
\[\frac{1}{\lambda_\flat^p}\sum_{I,J} \phi_{I,J}\left(\log \frac{\log
\val
{f_1}}{\lambda_\flat},\ldots,
\frac{\log \val{f_m}}{\lambda_\flat}\right)
\di \log\val{f_I}\wedge \dc \log \val{f_J}\]
where $(f_1,\ldots, f_m)$ are regular functions, where $I$, \resp $J$, is running through the set of subsets of
$\{1,\ldots, n\}$ of cardinality $p$, \resp $q$
(with  $\di \log \val{f_I}$ standing for
the wedge product $\di \log \val{f_{i_1}}\wedge\ldots\wedge\di \log \val{f_{i_p}}$ if $i_1<i_2<\ldots<i_p$ are the elements
of $I$, and similarly for $\dc \log \val{ f_J}$), and where each $\phi_{I,J}$  satisfies the same conditions as those in the definition of $\mathsf A^{p,q}$.

It is clear that sections of $\mathsf B^{p,q}$
are  locally of the form described
in \ref{bpq-fake-definition}, and that $\mathsf B^{\bullet, \bullet}$ is stable under the two differential operators 
$\di $ and $\dc$.

\subsection{}Our main result, Theorem \ref{main-theorem},
states that the two sheaves of bi-graded differential $\R$-algebras $\mathsf A^{\bullet, \bullet}$  and $\mathsf B^{\bullet, \bullet}$ 
on the site $X_{\mathrm {Zar}}$,
consisting respectively of non-standard archimedean and non-archimedean forms, are compatible in the following sense: 
\begin{theo}
There exists  a unique morphism of sheaves of bi-graded differential $\R$-algebras $ \mathsf A^{\bullet, \bullet} \to \mathsf B^{\bullet, \bullet}$, 
 sending a non-standard archimedean form  $\omega$ to the non-archimedean form $\omega_\flat$,
such that 
if $\omega$ is of the form
\[\omega =  \frac 1{\lambda^{\abs{I}}}\phi\left(\frac{\log \abs{f_1}}\lambda,\ldots, \frac{\log \abs{f_m}}\lambda\right)
\dll\abs{f_I}\wedge \dArg{f_J},\]
with $f_1,\ldots, f_m$  regular functions on a Zariski-open subset $U$ of $X$, 
$I$ and $J$ subsets of $\{1,\ldots, m\}$, and $\phi$
a quasi-smooth function,
then
\[\omega_\flat = 
\frac 1{\lambda_\flat^{\abs{I}}}\phi\left(\frac{\log \val{f_1}}{\lambda_\flat},\ldots, \frac{\log \val {f_m}}
{\lambda_\flat}\right)
\di \log\val{f_I}\wedge \dc \log \val{f_J}.\]
\end{theo}

Furthermore, we also prove in Theorem \ref{main-theorem}
that the mapping $\omega \mapsto \omega_\flat$ is compatible with integration.
A special case of that compatibility can be stated as follows:

\begin{prop}
Assume  that $\omega$ is an $(n,n)$-form defined on some Zariski open subset
$U$ of $X$
and that its support is contained in a definably compact
semi-algebraic subset of $U(C)$, then the form
$\omega_\flat$ on $X\an$
is compactly supported, $\int_{U(C)}\abs{\omega}$
is bounded by some positive real number  in $\R$ and
\[
\std\left(\int_{U(C)}\omega\right)
=\int_{U\an}\omega_\flat,
\]
with $\std$ standing for
the standard part. 
\end{prop}

Compatibility with integration is used in an essential way in proving
that the mapping $\omega \mapsto \omega_\flat$ is well defined. Indeed it allows us to use  a result of Chambert-Loir and Ducros
(\cite{chambertloir-d2012}, Cor. 4.3.7)
stating that, in the
boundaryless case, non-zero forms define non-zero currents.
A key point in the  proof of compatibility with integration is to show that the non-archimedean degree involved in the construction of non-archimedean integrals  in \cite{chambertloir-d2012}
actually shows up in the asymptotics of the corresponding archimedean integrals, which is done in \ref{9.1.11}.

\medskip

This main result has very concrete consequences: see our Theorem \ref{epilogue},
in which we express limits in the usual sense  of complex integrals depending on a parameter in terms of non-archimedean integrals.

\subsection*{Acknowledgements}
We would like to thank Alex Wilkie for proposing the current version of the proof of  Lemma \ref {lem-negli-cut}, which is much simpler than the original proof.
We are also grateful to the referees for their thorough reading  of the manuscript and their many suggestions that helped us in  significantly  improving the exposition.

Both A.D. and F.L. were partially supported  by ANR-15-CE40-0008 (D\'efig\'eo)
and by the Institut Universitaire de France.

\section{The general framework}

\subsection{}We shall use in this paper  basic facts and terminology from  Model Theory, which can be found for instance in
the books 
\cite{marker} and \cite{tz}. We shall in particular make use of  the theory {\sc doag}  of non-trivial
divisible ordered abelian groups, 
the theory  {\sc rcf} of real closed fields, and  the theory {\sc acvf}  of algebraically closed non-trivially valued fields.
Both {\sc doag} and  {\sc rcf} are examples of o-minimal theories.

\subsection{}\label{ultrafilter}
We fix a non-principal ultrafilter $\mathscr U$ on the set $\C$ of complex numbers; we assume that it converges to $0$, which
means that every neighborhood of the origin belongs to $\mathscr U$ 
(for our purpose, it
would be sufficient to consider such an ultrafilter
$\mathscr U$ on a sequence approaching $0$).
Note that since $\mathscr U$ is not principal, $\{0\}\notin \mathscr U$; 
as a consequence, every \emph{punctured}
neighborhood of $0$ also belongs to $\mathscr U$. 
\FL{In particular, there exists a family $X_i$, $i \in \mathbf{N}$, of elements of $\mathscr U$ such that
$\bigcap_{i \in \bf{N}} X_i = \emptyset$, that is, the ultrafilter $\mathscr U$ is countably incomplete.}

\begin{enonce}[remark]{Convention}\label{convention-sequences}
Unless otherwise stated,  when we introduce a ``sequence" $(a_t)_t$ the parameter $t$ is always 
understood as running through some set belonging to $\mathscr U$
(\eg, a small  punctured disc centered at the
origin), which we shall usually not make explicit. We shall 
allow ourselves to shrink this set of parameters when necessary (without mentioning it), 
for instance if we work with finitely many sequences and need
a common set of parameters. 

If we work with some sequence $(M_t)_t$ of sets
and then consider a sequence $(a_t)_t$ with $a_t\in M_t$ for every $t$, 
it will be understood that $a_t$ is defined for $t$ lying in \emph{some} set belonging to $\mathscr U$ and on which $t\mapsto M_t$
does make sense; so we do \emph{not}
require that $a_t$ be defined for every $t$ such that $M_t$ is. 

We say that some specified property $\mathsf P$ is satisfied by $a_t$ \emph{along $\mathscr U$}
if the set of indices $t$ such that $a_t$ satisfies $\mathsf P$ belongs to $\mathscr U$; \eg, $\abs{a_t}<\abs{t}$
along $\mathscr U$ means that the set of indices $t$ such that $\abs{a_t}<\abs{t}$ belongs to $\mathscr U$. 
\end{enonce}

\subsection{Ultraproducts}
Let $(M_t)_t$ be a sequence of sets. The \emph{ultraproduct of the sets $M_t$ along $\mathscr U$}
is the quotient of the set of all sequences $(a_t)_t$ with $a_t\in M_t$ for all $t$
by the equivalence
relation for which $(a_t)\sim(b_t)$ if and only if $a_t=b_t$ along $\mathscr U$
(we remind the reader that according to Convention \ref{convention-sequences}, $a_t$ needs not to be defined
for all $t$ for
which $M_t$ exists, but only for a subset of such complex numbers $t$ that belongs to $\mathscr U$).
If all the sets $M_t$ are
groups (\resp rings, \resp \ldots) the ultraproduct of the sets $M_t$ along $\mathscr U$
inherits a natural structure of 
group (\resp ring, \resp \ldots), which enjoys all first-order properties that hold for $M_t$ along $\mathscr U$; 
\eg, if the group $M_t$ is abelian along $\mathscr U$, the ultraproduct of the groups $M_t$ along $\mathscr U$ is abelian.

\begin{rema}
One can describe in a perhaps unusual way the ultraproduct of the sets $M_t$
as $\mathrm{colim}_T M_T$ where $T$ runs through the set of elements of $\mathscr U$
included in the domain of $t\mapsto M_t$, where $M_T:=\prod_{t\in T} M_t$, and where the
transition maps are the obvious ones. 
\end{rema}

\subsection{The field $^*\C$}
We apply the above by taking $M_t$ equal to the field $\C$ (\resp $\R$)
for all $t$, and we denote by $^*\C$ (\resp $^*\R$) the corresponding ultraproduct. 
The field
$^*\R$ is a real closed extension of $\R$; the field $^*\C$ is equal to $^*\R(i)$ and is an algebraically closed extension of $\C$.
We still denote by $\abs \cdot$ the ``absolute value" on $^*\C$; this is the map from $^*\C$ to $^*\R_+$ that maps
$a+bi$ to $\sqrt{a^2+b^2}$. By (harmless) abuse, the image in $^*\C$
of the sequence $(t)_t$ will also be denoted by $t$ ; it should be thought of as
a non-standard
complex number with infinitely small (but non-zero!) absolute value. 

A sequence $(a_t)_t$ of complex numbers is called: 

\begin{itemize}[label=$\diamond$]
\item \emph{bounded} if there is some $N\in \N$ such that $\abs{a_t}\leq N$ along $\mathscr U$; 
\item \emph{$t$-bounded} if there is some $N\in \N$ such that $\abs{a_t}\leq \abs{t^{-N}}$ along $\mathscr U$;
\item \emph{negligible} if $\abs{a_t}\leq \frac 1 N$ along $\mathscr U$ for all $N\in \Z_{>0}$; 
\item \emph{$t$-negligible} if $\abs{a_t}\leq \abs{t^N}$ along $\mathscr U$ for all $N\in \N$. 
\end{itemize}

An element $a$ of $^*\C$
is called \emph{bounded}, \resp \emph{$t$-bounded}, \resp \emph{negligible}, \resp {$t$-negligible}
if it is the image of some bounded, \resp $t$-bounded, \resp negligible, \resp $t$-negligible
sequence. This  amounts to requiring that
$\abs{a} \leq N$
for some integer $N\geq 0$, \resp $\abs{a}\leq \abs{t^{-N}}$ for some integer $N\geq 0$, 
\resp $\abs{a}\leq \frac 1 N$ for all integer $N>0$, \resp $\abs a\leq \abs t^N$ for all integer $N\geq 0$.
(Be aware that the above
inequalities are understood in the huge real closed field $^*\R$.)

\subsection{The field
$C$}\label{description-c}\label{ss-field-c}
The set $A$ of $t$-bounded elements of $^*\C$ is a subring of $^*\C$
which contains
$t$. This is a local ring, 
whose maximal ideal $\mathfrak m$ is the set of $t$-negligible elements; the intersection $A_{\mathrm r}:=A\cap {^*\R}$ is also 
a local ring, whose maximal ideal is $\mathfrak m_{\mathrm r}:=\mathfrak m \cap {^*\R}$. We denote by $C$ (\resp $R$)
the residue field of $A$ (\resp $A_{\mathrm r}$), and we still denote by $t$ the image
of the element $t$ of $A$ in $C$. Note that $\mathfrak m\neq 0$: for instance, 
the sequence
$(\exp (-1/\abs t))_t$ is $t$-negligible and not equal to zero along
$\mathscr U$, so it defines a non-zero element of $\mathfrak m$. 
One can describe directly $C$ as the ring of $t$-bounded sequences
modulo that of $t$-negligible sequences. The field $R$ is a real-closed extension of $\R$, we have $C=R(i)$, and $C$ is an algebraically closed
extension of $\C$. We still denote by $\abs \cdot$ the ``absolute value" on $C$; this is the map from $C$ to $R_+$ that maps
$a+bi$ to $\sqrt{a^2+b^2}$. An element $z$ of $C$ is called \emph{bounded}, \resp \emph{negligible} if it is the image of a bounded, \resp negligible, 
element of $A$. This amounts to requiring that $z$ is the image of a bounded, \resp negligible, sequence or
that $\abs z \leq N$ for some $N\in \N$, \resp
$\abs z <\frac 1 N$ for all $N\in \Z_{>0}$. 

If $z=a+bi$ is any bounded element of $C$, the subset of $\R$ consisting of those real numbers
that are $\leq a$ is non-empty
and bounded above, hence has a least upper bound $\alpha\in \R$; we define $\beta$ analogously. By construction, $z-(\alpha+\beta i)$ is negligible, 
and $\alpha+\beta i$ is the only complex number having this property; 
it is called the \emph{standard part}
of $z$
and it will be denoted by
$\std(z)$. If $z\in R$ then $\std(z)\in \R$. 

Any $t$-bounded complex-valued function $f$ on an element of $\mathscr U$ (\eg, a punctured small disc centered at the origin)
gives rise to an element of $C$, which we shall denote by $f$ if no confusion arises, as we do for $t$. Let us give some examples:

\begin{itemize}[label=$\diamond$]
\item For every $\alpha\in \R$ the sequence $(\abs t^\alpha)_t$ is $t$-bounded
and is not $t$-negligible, so it gives rise 
to an element $\abs t ^\alpha$ of $C\gpm$ (which actually belongs to $R_+\gpm$).
Note that if $\alpha\neq 0$ then $(\abs{t}^\alpha-1)_t$ is not $t$-negligible; hence $\alpha \mapsto \abs t^\alpha$ is
an injective order-reversing group homomorphism from $\R$ into $R_+\gpm$. 

\item The field $\mathscr M$ of meromorphic functions around the origin has a natural embedding into $C$. 

\item If $a$ is any non-zero element of $C$ arising from a $t$-bounded and non-$t$-negligible sequence $(a_t)_t$ then the sequence
$(\log \abs{a_t})_t$ is $t$-bounded, so it gives rise to an element of $C$. The latter depends only on $a$, and not on the specific sequence $(a_t)$. To see it, we have to check that if
$(\epsilon_t)_t$  is a $t$-negligible sequence then $(\log \abs{a_t+\epsilon_t}-\log \abs{a_t})$ 
is $t$-negligible as
well. For that purpose, we first notice that if $z$ is a standard complex number with $\abs z$ small enough then 
$\log \abs{1+z}\leq 2\abs z$. 
Now our assumptions imply that 
the sequence $(\epsilon_ta_t\inv)_t$ is $t$-negligible
and a fortiori negligible, so that 
\[\log \abs{a_t+\epsilon_t}-\log \abs{a_t}=\log \abs{(1+\epsilon_t\abs{a_t}\inv)}\leq 2\abs{\epsilon_ta_t\inv}
\]
holds along $\mathscr U$; using once again the fact that 
$(\epsilon_ta_t\inv)_t$ is $t$-negligible we get the required result. %FL

The element of $C$ defined by the sequence $(\log \abs{a_t})_t$ 
depending only on $a$, we denote it by $\log \abs a$. 
The sequence $\left(\frac{\log \abs {a_t}}{\log \abs t}\right)_t$
is bounded, so $\frac{\log \abs a}{\log \abs t}$ is bounded.

\end{itemize}

Set $\Lambda=\{r\in R_+\gpm\;| \;\abs t^{1/N}\leq r\leq \abs t^{-1/N}\;\text{for all}\;N\in \Z_{>0}\}$; this
is a convex subgroup of $R_+\gpm$, and $R_+\gpm/\Lambda$
thus inherits an ordering such that the quotient map is order-preserving.
The composition 
\[\xymatrix{{C\gpm}\ar[r]^{\abs \cdot}&{R_+\gpm}\ar[r]&R_+\gpm/\Lambda}\]
is a valuation $\val \cdot$, and $\val {C\gpm}=R_+\gpm/\Lambda$.
The valuation ring $C^\circ$ of $\val \cdot$ is the set of the elements $z\in C$ such that $\abs z<\abs t^{-1/N}$ for all integer $N>0$, %FL
and the maximal ideal of $C^\circ$ 
is the set $C^{\circ \circ}$
of elements $z$ of $C$ such that $\abs z <\abs t^{1/N}$ for some integer $N>0$
(note that $C^\circ$ contains the ring of bounded elements of $C$).

%the groupe $(C^\circ)\gpm$ is the set
%of elements $z\in C$ such that $\abs t^{1/N}\leq \abs z\leq \abs t^{-1/N}$ for all integer $N>0$. 
%
%Let $\val \cdot$ be the valuation on $C$ whose ring is $C^\circ$.
%The value group of $\val \cdot$
%is the quotient $C\gpm/(C^\circ)\gpm$, ordered in such a way that
%the semi-group of elements bounded above by $1$ is equal
%to the image of $C^\circ\setminus\{0\}$ (and the valuation is then nothing but the quotient map).
%But in fact, the ordered group
%$C\gpm/(C^\circ)\gpm$ can be described in a more
%concrete way.  Indeed, set $\Lambda=\{r\in R_+\gpm\;| \;\abs t^{1/N}\leq r\leq \abs t^{-1/N}\;\text{for all}\;n\in \Z_{>0}\}$; this
%is a convex subgroup of $R_+\gpm$, and the map $z\mapsto \abs z$ induces
%a natural order-preserving isomorphism
%\[C\gpm/(C^\circ)\gpm\simeq R_+\gpm/\Lambda\]
%(the target group is ordered so that $R_+\gpm\to R_+\gpm/\Lambda$ is non-decreasing). 
%We can thus identify $\val{C\gpm}$ with $R_+\gpm/\Lambda$
%in such a way that 
% $\val z$ is equal to the class of $\abs z$ for all $z\in C\gpm$. 

Let $z\in C\gpm$ and set $\alpha=\std(\frac {\log  \abs z}{\log \abs t})$. It follows
immediately from the definitions that $\abs z=\abs t^\alpha$ modulo $\Lambda$, and that $\abs t^\alpha$
itself belongs to $\Lambda$ if and only if $\alpha=0$. Hence
$\alpha\mapsto \abs t^\alpha\mod \Lambda$ induces an order-reversing
isomorphism between the ordered groups $\R$ and  %FL
$\val{C\gpm}$, which maps $1$ to $\val t=\abs t \mod \Lambda$. 

We fix once and for all 
an order-preserving isomorphism between $\val {C\gpm}$ and $\R_+\gpm$, which amounts to choosing
the image $\tau$ of $\val t$ in $(0,1)$. We will from now
on use this isomorphism to see $\val\cdot$
as a real valuation (with value group the whole of $\R_+\gpm$). If $z$ is any element of $C\gpm$
we have
\[\val z=\val t^{\std\left(\frac{\log \abs z}{\log \abs t}\right)}=\tau^{\std\left(\frac{\log \abs z}{\log \abs t}\right)}.\]

The residue field $\widetilde C:=C^\circ/C^{\circ \circ}$ is an algebraically closed extension of $\C$. Let us give
an example of an element of $\widetilde C$ that is transcendent over $\C$. For every complex number $\lambda$
and every integer $N>0$ 
the (complex) inequalities $1\leq \abs{\log \abs t-\lambda}\leq \abs t^{-1/N}$
hold along $\mathscr U$; as a consequence, $1\leq  \abs{\log \abs t-\lambda}\leq \abs t^{-1/N}$
in $R$ for all integer $N>0$, so $\val{\log \abs t-\lambda}=1$.
Hence $\val{\log \abs t}=1$
and
if we denote by $\widetilde{\log \abs t}$ the image of $\log \abs t$ in $\widetilde C$
then $\widetilde{\log \abs t}-\lambda\neq 0$ for all $\lambda\in \C$; as a consequence, 
$\widetilde{\log \abs t}$ is transcendent over $\C$. 

The non-archimedean field $C$ is complete, and even spherically complete (cf. \cite{luxemburg}).
Indeed, let $(B_n)_{n\in \N}$ be a
decreasing sequence of closed balls with positive radius in $C$; for every $n$, denote by $r_n$ the radius of $B_n$ and choose $b_n$ in $B_n$; we want to prove that $\bigcap B_n$ is non-empty. For every $n\geq 1$, choose
a pre-image $\mathsf b_n$ of $b_n$ in $A$, and a real number $s_n$
with $r_{n-1}>s_n>r_n$, and denote by $\mathsf B_n$ the set of those 
$x\in{^*\C}$ such that $\abs{x-\mathsf b_n}\leq \abs t^{\log s_n/\log \tau}$. 
For each $n\geq 1$, the ball $\mathsf B_n$ contains the pre-image of 
$B_n$ in $A$, and is contained in the pre-image of $B_{n-1}$. 
The fact that every $\mathsf B_n$ contains the pre-image of $B_n$ in $A$
implies that the intersection of finitely many of the sets $\mathsf B_n$ is non-empty; 
\FL{since, as noted in \ref{ultrafilter}, the ultrafilter $\mathscr {U}$ is countably incomplete,
the ultraproduct $^*\C$ is $\aleph_1$-saturated by \cite{keisler}, Cor. 2.2, thus 
the intersection of all the sets $\mathsf B_n$ is non-empty;
but this intersection is contained in the pre-image of the intersection of all the sets $B_n$, so the latter is non-empty.}

%hence by $\aleph_1$-saturation of the ultraproduct $^*\C$
%(see \cite{marker}, Def. 4.3.1), 
%the intersection of all the sets $\mathsf B_n$ is non-empty; but this intersection is contained in the pre-image of the intersection of all the sets $B_n$, so the latter is non-empty. 

\section{Smooth functions, smooth forms and their integrals \\ over $^*\R$
and $^*\C$}\label{section3}

\subsection{Semi-algebraic topology}
Let $S$ be an arbitrary real-closed field (we will use what follows for $S={^*\R}$ and $S=R$). 
Let $X$ be an algebraic variety over the field $S$; \ie, $X$ is a separated $S$-scheme
of finite type. 
The set $X(S)$ is in a natural way a definable space
of {\sc rcf}. By quantifier elimination 
in {\sc rcf}, the definable subsets of $X(S)$
are precisely its \emph{semi-algebraic}
subsets; \ie, those subsets that can be defined locally for the Zariski-topology of $X$
by a boolean combinations of inequalities (strict or non-strict) between regular functions. 

\subsubsection{}
The order topology
on the field $S$ induces a topology of $X(S)$, which is 
most of the time poorly behaved: except if $S=\R$ it is neither locally compact
nor locally connected. 

Let $U$ be a semi-algebraic subset of $X(S)$. We shall say that $U$ is open, \resp closed, 
if it is open, \resp closed for this topology. This amounts to require that $U$ can be defined, locally
for the Zariski-topology of $X$, by a positive boolean combination of strict (\resp non-strict)
inequalities between regular functions (\cite{bochnak-c-r1985}, Th. 2.7.1). 
The topological closure of a semi-algebraic subset $U$ of $X(S)$ is
semi-algebraic (and so is its topological interior, by considering
complements). 
Indeed, this can be checked on an affine chart, hence we reduce
to the case where $X=\A^n_S$; now since the topology on $S^n$ has a basis consisting of products of
open intervals, $\overline U$ is definable, so it is semi-algebraic.

\subsubsection{}
Since the interval $[0,1]$ of $S$ is not compact except if $S=\R$, naive topological
compactness is not
a relevant
notion in our setting.
We use \emph{definable compactness} instead, which itself relies
on the notion of a \emph{definable type}; see for instance Section 2.3 and Chapter 4 of \cite{hrushovski-l2016} for more information on
these topics. Let us just recall here that a
subset $E$ of $X(S)$ is called
\emph{definably compact} if every definable type lying on $E$
converges to a unique point of $E$.  
Since $X$ is separated, any definably  %FL
compact semi-algebraic subset of $X(S)$ is closed. 
If $E$ is a definably compact semi-algebraic subset of $X(S)$,%AD
a semi-algebraic subset $F$ of $E$ is closed if and only if it is definably compact.

\subsubsection{}
Assume that $X$ is affine, and let $(f_1,\ldots,f_n)$ be a family
of regular functions on $X$ that generate the $S$-algebra $\mathscr O(X)$.
If $E$ is a semi-algebraic subset of $X(S)$, then $E$ is definably compact if and only if it is
closed and \emph{bounded}; \ie, there exists $r>0$ in $S$ and
such that $\abs {f_i(x)}\leq r$ for all $i$ and all $x\in E$. 

\begin{lemm}\label{lem-compact-open-cover}
Let $X$ be a separated $S$-scheme of finite type and let $E$
be a definably compact semi-algebraic
subset of $X(S)$. Let $(U_i)_{i\in I}$ be a finite
family of definable open subsets of $X(S)$ such that $E\subset \bigcup U_i$.
There exists a family $(E_i)$ with each $E_i$
a definably compact semi-algebraic
subset of $U_i$ and $E=\bigcup_i E_i$. 
\end{lemm}

\begin{proof}
Up to refining the covering $(U_i)$
we can assume that $U_i$
is for every $i$ contained in $X_i(S)$
for some open affine subscheme $X_i$ of 
$X$. We argue by induction on $\abs I$. The statement is clear if $\abs I=0$. Assume $\abs I>0$ and
the statement is true in cardinality $<\abs I$. Choose an element $i$ in
$I$ and set $F=X(S)\setminus\bigcup_{j\in I, j\neq i}U_j$. 
By definition, $F$ is a  closed semi-algebraic subset of $X(S)$ contained in $U_i$; thus $E\cap F$ is a definably compact
semi-algebraic
subset of $U_i$.

Choose
a semi-algebraic open subset $V$ of $U_i$ that contains $E\cap F$ and whose
closure $\overline V$ is definably compact and still contained in $U_i$
(one can use a finite set of generators of the $S$-algebra $\mathscr O_X(X_i)$
to build
semi-algebraic continuous distance functions to $E\cap F$, to the boundary of $U_i$ in
$X_i(S)$ and to $(X\setminus X_i)(S)$, and then define $V$ by a suitable
positive boolean combination of non-strict inequalities involving these functions).

Set $G=X(S)\setminus V$. By definition, $G$
is a closed semi-algebraic
subset of $X(S)$ and $G\cap E$ is thus definably compact. 

We then have $E=(E\cap \overline V)\cup (G\cap E)$. Since $G\cap E$ avoids $F$, it is contained in $\bigcup_{j\neq i}U_j$. 
The conclusion follows by applying the induction hypothesis to the set $G\cap E$. 
\end{proof}

\subsection{}\label{def-semialg-topology}
Because of the bad properties of the order topology $X(S)$, 
we shall not use it except while speaking of
closed or open semi-algebraic subsets. 
Nevertheless,
we shall use a closely related set-theoretic
Grothendieck topology,
namely the
\emph{semi-algebraic}
topology. The underlying category is that of open
semi-algebraic subsets of $X(S)$ with inclusion maps; 
a family $(U_i)_{i\in I}$ is a cover of $U$ if there is a \emph{finite} subset $J$ of $I$ such that $U=\bigcup_{i\in J}U_i$; 
this amounts to requiring that $(U_i\to U)$ induces a usual (open)
cover at the level of type spaces. 

If $X$ is smooth, $X(S)$ comes equipped with a \emph{sheaf of orientations} (for the semi-algebraic topology), defined \emph{mutatis mutandis} as in the standard case. It is locally isomorphic to the constant sheaf associated with a two-element set; a global section of this sheaf is called an \emph{orientation}
on $X(S)$. 

%\begin{lemm}[Countable Borel-Lebesgue property over $R$]
%\label{lem-borel-lebesgue}
%Let $X$ be an $R$-scheme of finite type and let $K$ be a definably
%compact semi-algebraic subset of $X(R)$. Let $(U_i)$ be a countable family of semi-algebraic open subsets of $X(R)$
%such that $K\subset \bigcup_i U_i$. 
%Then there exists a finite set $J$ of indices such that 
%$K\subset \bigcup_{i\in J}  U_i$.
%\end{lemm}
%
%\begin{proof}
%In view of Lemma \ref{lem-compact-open-cover}
%we may assume that $X$ is affine, and then (by embedding it
%in $\A^n_R$ for some $n$ and by lifting the functions involved in the definitions of the $U_i$'s to the scheme $\A^n_R$)
%that $X=\A^n_R$. By enlarging $K$ if necessary, we can also assume
%that $K=[-M, M]^n$ for some positive $M$. Set
%$F_i=[-M, M]^n\setminus U_i$. By assumption, $\bigcap F_i=emptyset$ and we
%want to show that there is a finite set $J$
%of indices with $\bigcap_{i\in J}F_i=\emptyset$.
%
%Let $\pi \colon A^n\to R^n$ be the quotient map. By considering
%an explicit system of inequalities defining $F_i$, we see that $\pi\inv(F_i)$ is a countable intersection $\bigcap \Phi_{ij}$
%of semi-algebraic closed subsets of $[-\Lambda, \Lambda]^n$ for some $\Lambda \in {^*\R}$
%(for instance, $\Lambda$ can be any pre-image of $M+\abs t^\epsilon$ for some arbitrary positive standard $\epsilon$). 
%
%Since the model $^*\R$ of {\sc rcf}
%is sufficiently saturated, there is a finite set
%$E$ of indices
%$(i,j)$ such that $\bigcap_{(i,j)\in E}\Phi_{ij}=\emptyset$; 
%a fortiori, there is a finite set $J$
%of indices with $\bigcap_{i\in J}F_i=\emptyset$.
%\end{proof}

\subsection{Smooth forms and integrals over the field $^*\R$}
If $U$ is an open semi-algebraic subset of $\R^n$
for some $n$, then
every smooth function
(\ie, $\mathscr C^\infty$ function)
$\phi \colon U\to \R$ gives rise to a function 
$U(^*\R)\to \;{^*\R}$, which sends the class of a sequence $(a_t)_t$ with $a_t\in U$ along $\mathscr U$ to
the class of $(\phi(a_t))_t$; it will
still be denoted by $\phi$ if no confusion arises. 

\subsubsection{Smooth functions and smooth forms on a variety}
Let $X$ be a smooth separated $^*\R$-scheme of
finite type. Let $\mathscr F$ be the assignment that
sends a semi-algebraic open subset $U$ of $X(^*\R)$ to the set of functions from $U$
to $^*\R$ of the form 
$\phi \circ g$, where:  

\begin{itemize}
\item[$\diamond$] $g$ is a regular map from a Zariski-open
subset of $X$ containing $U$ to $\A^m_{^*\!\R}$ for some $m$;   %FL
\item[$\diamond$] $\phi$ is a smooth function from
$V$ to $\R$ where $V$
is a semi-algebraic open subset of $\R^m$ such that $g(U)\subset V(^*\R)$.
\end{itemize}
Then $\mathscr F$ is a presheaf; its associated sheaf (for the semi-algebraic topology)
is denoted by $\mathscr C^\infty$ or $\mathscr C^\infty_X$ and called the sheaf of \emph{smooth functions}
on $X(^*\R)$. It makes $X(^*\R)$ a locally ringed site. 

The natural embedding of $X(^*\R)$ into
(the underlying set of) the scheme $X$ underlies a morphism of locally ringed
sites $\psi \colon (X(^*\R),\mathscr C^\infty_X)\to (X,\mathscr O_X)$; hence $\psi^*\Omega^p_{X/^*\R}$
is for every $p$ a well-defined $\mathscr C^\infty_X$-module on  $X(^*\R)$, which we denote by $\mathscr A^p$
or $\mathscr A^p_X$. 
The sheaf $\mathscr A^0_X$ is equal to $\mathscr C^\infty_X$, and the $\mathscr C^\infty_X$-module $\mathscr A^1_X$ 
is locally free (of rank $n$ if $X$ is of pure dimension $n$); for every $p$, we have $\mathscr A^p_X=\Lambda^p\mathscr A_X^1$. 
The sheaf $\mathscr A_X^p$ is called the sheaf of \emph{smooth $p$}-forms on $X(^*\R)$. One has for every $p$
a natural differential $\d \colon \mathscr A^p_X\to \mathscr A^{p+1}_X$. 
The sheaf  $^*\C\otimes_{^*\!\R}\mathscr A_X^p$ is called the sheaf of \emph{complex-valued}  %FL
$p$-forms on $X(^*\R)$. Every complex-valued $p$-form $\omega$ defined on a semi-algebraic
open subset $U$ of $X(^*\R)$ can be evaluated at any point $u$ of $U$, giving rise to an element $\omega(u)$
of the $^*\C$-vector space ${^*\C}
\otimes_{\mathscr O_{X,u}}\Omega^p_{X,u}$.

\subsubsection{Integral of an $n$-form}\label{int-nform-bigR}
We still denote by $X$
a smooth separated $^*\R$-scheme of finite type; we assume that it is of pure dimension $n$
for some $n$, and that $X(^*\R)$ has been given an orientation.
Let $\omega$
be a complex-valued smooth $n$-form on some semi-algebraic open subset $U$ of $X(^*\R)$, 
and let $E$ be a semi-algebraic subset of $U$ whose closure in $U$ is definably compact. 

We now choose a description of $(X,U,\omega,E)$ through a ``limited family" $(X_t,U_t,\omega_t,E_t)_t$ where $X_t$
is for every $t$ a smooth separated $\R$-scheme of pure dimension $n$ endowed with an orientation of $X_t(\R)$, 
$U_t$ is an open subset of $X_t(\R)$, $\omega_t$
is a complex valued smooth form on $U_t$, and $E_t$ is a relatively compact semi-algebraic subset
of $U_t$. 
 The expression ``limited family"
means that the sequence $(X_t,U_t,\omega_t,E_t)$
can be defined using finitely many smooth functions (defined on real intervals), a given set $T\in \mathscr U$,
and finitely many polynomials with coefficients in $\R^T$. 

For every $t$ the smooth manifold $X_t(\R)$ is oriented, hence the
integral
$\int_{E_t}\omega_t$
is well-defined. The sequence $(\int_{E_t}\omega_t)_t$ defines an element of $^*\C$ that depends only
on $(X,U,\omega,E)$, and the chosen orientation on $X(^*\R)$. We denote it by $\int_E\omega$; if $\omega$
is real-valued, then $\int_E\omega$ is an element of  $^*\R$.

\subsubsection{The case of a non-standard complex variety}\label{int-nnform-bigR}
Let $X$ be now a smooth quasi-projective scheme over $^*\C$, and let $Y$ be
the Weil restriction 
$\mathrm R_{^*\R/^*\C}X$; this is a quasi-projective scheme over $^*\R$, equipped by definition
with a canonical bijection
$Y(^*\R)\simeq X(^*\C)$. This allows us to transfer to the set $X^*(\C)$ all notions introduced above. Moreover,
for every $p$, the sheaf $\mathscr A^p\otimes_{^*\!\R}{^*\C}$ of  %FL
complex-valued smooth $p$-forms on $X^*(\C)$ is equipped
with a natural decomposition $\mathscr A^p\otimes_{^*\!\R}{^*\C}=\bigoplus_{i+j=p}\mathscr A^{i,j}$, where %FL
$\mathscr A^{i,j}$ is the sheaf of $(i,j)$-forms; \ie, of complex-valued $p$-forms generated over $\mathscr C^\infty$ by forms of
the type \[\d f_1\wedge \ldots \wedge \d f_i\wedge \d\overline {g_1}\wedge \ldots
\wedge \d \overline{g_j}\]
for some regular functions $f_1,\ldots, f_i,g_1,\ldots, g_j$. 

Assume that $X$ is of pure dimension $n$
for some $n$,
let $U$ be a semi-algebraic open subset of $X(^*\C)$, and let $\omega$ be a smooth $(n,n)$-form on $U$.
Let $E$ be a semi-algebraic
subset of $X(^*\C)$
whose closure is definably compact.
The
$(n,n)$-form $\omega$ can then be integrated on $E$,
using the canonical
orientation of $X(^*\C)$. Indeed, 
choose a description of
$(X,U,\omega,E)$ through a ``limited family" $(X_t,U_t,\omega_t,E_t)_t$ where $X_t$
is for every $t$ a smooth separated $\C$-scheme of pure dimension $n$, $U_t$ is an open
semi-algebraic subset of $X_t(\C)$, $\omega_t$
is a complex valued smooth $(n,n)$-form on $U_t$, and $E_t$ is a
relatively compact semi-algebraic
subset of $U_t$;  the integral $\int_E\omega$ is then given
by the sequence $\int_{E_t}\omega_t$. 

\subsubsection{}\label{sss-def-absomega}
We have considered so far only differential forms with smooth coefficients. But 
by replacing the class of usual smooth functions (on open subsets of $\R^m$) 
by a broader class $\mathscr C$, we can define 
in the same way differential forms over $^*\R$ with coefficients in $\mathscr C$, 
and integrate those of maximal rank on relatively compact definable
subsets (provided
$\mathscr C$ consists of locally integrable
functions). 

For instance, if we consider $\omega$ and $E$ as in \ref{int-nform-bigR}
or \ref{int-nnform-bigR}
we can define $\abs \omega$, which is a form with continuous piecewise smooth coefficients, 
and also define the integral $\int_E \abs \omega$, which is a non-negative element of $^*\R$.

\subsection{}
We are thus able to integrate smooth forms on the field $^*\R$, but what we are actually seeking for is a similar integration theory over $R$.
Our basic strategy is very simple: it consists in lifting a differential form on the field
$^*\R$, integrating it, and reducing the result modulo $t$-negligible elements. But of course, one has to check that it does not depend on our lifting. 
This requires a good understanting of the way our integrals interact with $t$-negligibility; this is the purpose of what follows.

\begin{enonce}[remark]{Notation}
Let $S$ be a real closed field and let $\mathfrak D$ be a non-empty, bounded above convex subset of $S$ with no
least upper bound in $S$. 
Then such a least upper bound nevertheless exists, but as a type on $S$; we denote it
by $d$. We shall allow ourselves to say that a given 
definable subset $I$ of $S$ contains $d$, resp. that a given definable formula $\Phi$ is satisfied by $d$, if $I$, resp. the set of $x\in S$ satisfying
$\Phi$, contains $(\lambda,+\infty)\cap \mathfrak D$ for some $\lambda \in \mathfrak  %FL
D$. 
\end{enonce}%AD

\begin{lemm}\label{lem-negli-cut}
Let $I$
be a definable interval of $S_{\geq 0}$ that contains $d$ and
let $f$ be a definable 
function from $I$ to $S$. 
Assume that there exists $a \in I$ with $a<d$ such that $f(x)>d$ for all
$x$ with $a <x<d$; then there exists $x>d$ in $I$ with $f(x)>d$.
\end{lemm}

\begin{proof}
Let $J$ be the set of those $x\in I$ such that $f(x)>x$. This is a definable subset of $S$ which contains all elements $y\in S$
with $a<y<d$. By o-minimality,
$J$ is a finite union of intervals with bounds in $S\cup\{-\infty,+\infty\}$, thus
it contains some interval of the form $(a,b)$ for some element $b\in S$ with $b>d$. Then for all $x\in S$ such that $d<x<b$ we have $f(x)>x>d$.
\end{proof}

\subsection{}
Let $D$ be a definable subset of $(^*\R)^n$ with definably compact closure.
The integral $\int_D \d x_1\wedge\ldots \wedge \d x_n$
is called the \emph{volume}
of $D$ and is denoted by $\mathrm{Vol}(D)$. 

If $D$ is a cube, \ie, $D$ is of the form $\prod_{1\leq i\leq n}[a_i, b_i]$, 
then $\mathrm{Vol}(D)=\prod_i(b_i-a_i)$. 

We remind that $A_{\mathrm r}$ is the set of $t$-bounded elements of $^*\R$, 
and that $a\mapsto \overline a$ denote the reduction modulo the maximal 
ideal $\mathfrak m_{\mathrm r}$ of $A_{\mathrm r}$, \emph{cf.}
\ref{ss-field-c}.

\begin{prop}\label{prop-equiv-negli}
Let $D$ be a definable subset of $(^*\R)^n$ contained in $A_{\mathrm r}^n$. 
The following are equivalent:

\begin{enumerate}[i]
\item the volume of $D$ is $t$-negligible; 
\item for every $n$-form $\omega=\phi \d x_1\wedge \ldots\wedge \d x_n$ with
$\phi$ a smooth function defined in a neighborhood of the closure of $D$ and taking only $t$-bounded values on the latter, 
the integral $\int_D \omega$ is $t$-negligible; 
\item every cube contained in $D$ has $t$-negligible volume; 
\item the image $\overline D$ of $D$ in $R^n$ through the reduction map is of dimension $\leq n-1$. 
\end{enumerate}
\end{prop}

\begin{rema}
It is known that $\overline D$ is a
 closed definable subset of $R^n$  %FL
(no matter whether $D$ is closed or not), see for instance \cite{broecker1991}. 
Thus its dimension is well-defined. But the reader could also rephrase (iv) by simply saying ``$\overline D$ contains
no $n$-cube with non-empty interior"; and this is indeed this rephrasing of (iv) that we shall actually use in the proof. 
\end{rema}

\begin{proof}
We are going to prove (i)$\Rightarrow$(ii)$\Rightarrow$(iii)$\Rightarrow$(i), and then (iv)$\Rightarrow$(iii) and (i)$\Rightarrow$(iv). 

Assume that (i) is true and let $\omega$ as in (ii). By
definable compactness of the closure of $D$
there exists a $t$-bounded positive element $M$
such that $\abs\phi\leq M$ on $D$. Then $\left|\int_D \omega\right| \leq M\mathrm{Vol}(D)$; the volume
of $D$ being $t$-negligible, $\int_D \omega$ is $t$-negligible as well. 

Now if (ii) is true then in particular $\mathrm{Vol}(D)$ is $t$-negligible (take $\phi=1$); this 
implies that the volume of every definable subset of $D$, including any cube contained in $D$, is $t$-negligible. 

Assume now that (iii) is true, and let us prove (i). We argue by induction on $n$. If $n=0$ there is nothing to prove. 
So assume that $n>0$ and the result holds in dimension $n-1$. Let $p\colon (^*\R)^n\to  (^*\R)^{n-1}$
be the projection on the first $n-1$ coordinates, and set $\Delta=p(D)$. If $(D_i)$ is any finite covering of $D$
by definable subsets it is sufficient to prove that (i) holds for every $D_i$ (note that $D_i$ obviously satisfies (iii)). 

Hence using cellular decomposition we can assume that we are in one of the following two cases: 
\begin{itemize}[label=$\diamond$]
\item there exists a continuous definable function $f$ on $\Delta$ such that $D$ is the graph of $f$; 
\item there exists two continuous definable functions $f$ and $g$ on $\Delta$ with $f<g$ such that 
$D=\{(x,y), f(x)<y<g(x)\}$.
\end{itemize}

In the first case $D$ is at most $(n-1)$-dimensional and its volume is zero. Let us assume
from now on
that we are in the second case.
Since $D\subset A_{\mathrm r}^n$, there is a positive $t$-bounded element $M$ such that $g-f<M$. 

Let $\phi$ be the function that sends an element $a$ of $[0,M]$ to the
least upper bound of the $(n-1)$-volumes of all cubes contained
in $\Delta$ over which $g-f>a$. 

\subsubsection{}
Let us prove by
contradiction that there exists some $t$-negligible element
$a$ such that $\phi(a)$ is $t$-negligible. 
We call ``$t$-significant" an element  which is not $t$-negligible, and we assume that $\phi(a)$ is $t$-significant 
for all $t$-negligible $a$; we are going to exhibit a cube inside $D$ with $t$-significant volume, which
will contradict our assumptions. 

By Lemma \ref{lem-negli-cut}
(which we apply by taking for
$d$ the least upper bound
of the set $\mathfrak D$
of $t$-negligible elements)
there exists some $t$-significant $a$
with $\phi(a)$ also $t$-significant. Therefore there exists some cube $K$ inside $\Delta$ with $t$-significant 
$(n-1)$-volume over which $g-f>a$. For each family 
$\epsilon=(\epsilon_1,\ldots, \epsilon_{2n-2})$ of elements of $\{-1,1\}$ 
let $K_\epsilon$ be the subset of $K$ on which $\partial_i g\in \epsilon_i (^*\R_{\geq 0})$
and $\partial_i f\in \epsilon_{n-1+i} (^*\R_{\geq 0})$ or all $1\leq i\leq n-1$. Then $K$ is the union of the sets $K_\epsilon$, 
so one of the sets $K_\epsilon$ has a $t$-significant volume, hence contains a cube $K'$ with $t$-significant volume
(by the induction hypothesis). Replacing $\Delta$ by $K'$, we assume from now on that $\Delta$ is a cube with $t$-significant
volume on which each partial derivative of $f$ and $g$ has constant sign and on which $g-f>a$. %FL

\FL{Write $\Delta = \prod[\alpha_i,\beta_i]$. 
Set $M = \sup_\Delta \vert f \vert$
and $K = 4M (\beta_1-\alpha_1)^{-1}$.
Since $M$ is $t$-bounded
and since $\beta_1-\alpha_1$ is $t$-significant
(because $\Delta$
has $t$-significant volume), $K$ is $t$-bounded.
Let $\Delta_K = \{x \in \Delta, \abs{\partial_1 f(x)} \geq K\}$.
We claim that
$\mathrm{Vol}(\Delta_K) \leq \frac{\mathrm{Vol}(\Delta)}2$.

Indeed, fix $z= (z_2, \cdots, z_{n-1})$
in $\prod_{i\geq 2}
[\alpha_i,\beta_i]$, and set 
$\Delta_{K, z} = \{y\in [\alpha_1,\beta_1],
(y,z)\in \Delta_K\}$. By o-minimality, $\Delta_{K,z}$
is a finite union of closed intervals; let $\lambda$ be the
one-dimensional volume (or, otherwise said, the total length)
of $\Delta_{K,z}$.
If $\gamma$ and $\delta$ are two elements of $[\alpha_1,\beta_1]$
such that $\gamma\leq \delta$
and $[\gamma,\delta]\subset \Delta_{K,z}$ then
by the mean value theorem one has 
$\abs{f(\delta,z))-f(\gamma,z)}\geq K(\delta-\gamma)$. 
By monotonicity of $f(\cdot, z)$ this implies that
$\abs{f(\beta_1,z)-f(\alpha_1,z)}\geq K\lambda$. 
Since $\abs{f(\beta_1,z)-f(\alpha_1,z)}\leq 2M$ by the definition
of $M$, we see that
$\lambda\leq 2M/K=(\beta_1-\alpha_1)/2.$
Thus, by Fubini, $\mathrm{Vol}(\Delta_K) \leq \frac{\mathrm{Vol} (\Delta)}2,$ as announced.

It follows that
the complement of 
$\Delta_K$ in $\Delta$ has $t$-significant volume.
By the
induction hypothesis, it contains a cube with $t$-significant volume.
Iterating this argument (which works for $g$
as well as for $f$, and for the $i$-th component as well as
for the first one), we can furthermore assume 
that $\Delta$ is a cube with $t$-significant volume on which each partial derivative of $f$
and $g$ has an absolute value bounded above by some positive $t$-bounded constant $N$}.

Let $x$ be the point $(\frac{\alpha_i+\beta_i}2)_i$ of $\Delta$. 
Set $y=\frac{g(x)+f(x)}2$; the point $(x,y)$ belongs to $D$. Set $r=(g(x)-f(x))/4$ ; since $g(x)-f(x)\geq a$, the number $r$
is $t$-significant. Let $N'$ be a $t$-bounded number such that $N'>\FL{\sqrt{n - 1}} N$ and $r/N'<\min_i (\beta_i-\alpha_i)/4$ -- such $N'$ exists since $\beta_i-\alpha_i$
is $t$-significant for every $i$. Let $\Gamma$ be the cube in $(^*\R)^n$ 
with center $(x,y)$ and polyradius $(r/N', \ldots, r/N',r)$. 
\FL{If $(\xi, \eta)$ belongs to $\Gamma$,
then $\xi\in \Delta$. By the mean value theorem $\abs{f(\xi) - f (x)} \leq \frac{\sqrt{n - 1}rN}{2N'} < \frac{r}{2}$ and similarly 
$\abs{g(\xi) - g(x)} < \frac{r}{2}$. Thus
$f(\xi)<\eta<g(\xi)$ and therefore $D$ contains the cube $\Gamma$ which has $t$-significant volume.}

\subsubsection{}
By the above, there exists some $t$-negligible element $a$ such that $\phi(a)$ is $t$-negligible. Let $\Delta'$ be the
subset of $\Delta$ consisting of points over which $g-f>a$. By assumption, every cube contained in $\Delta'$ has $t$-negligible volume; 
by our induction hypothesis, the volume of $\Delta'$ is $t$-negligible. Since $g-f$ is uniformly $t$-bounded, it follows from Fubini's theorem
that the volume of $p\inv(\Delta')$ is $t$-negligible. 
Let $\Delta''$ be the complement of $\Delta'$ in $\Delta$. The $(n-1)$-volume of $\Delta''$ is $t$-bounded, and $g-f\leq a$ on $\Delta''$. 
Applying again Fubini's theorem, we see that $p\inv(\Delta'')$ has $t$-negligible volume. Hence
$D=p\inv(\Delta')\cup p\inv(\Delta'')$ has $t$-negligible volume. This ends the proof of (i)$\iff$(ii)$\iff$(iii).

\subsubsection{Proof of  (iv)$\Rightarrow$(iii) and (i)$\Rightarrow$(iv).}
It is clear that (iv)$\Rightarrow$(iii), since the reduction of every cube in $A_{\mathrm r}^n$ with $t$-significant volume is
a cube with non-empty interior.
We are going to prove (i)$\Rightarrow$(iv)
by contraposition. So assume that $\overline D$ is $n$-dimensional.
Under this assumption, it contains a cube with non-empty interior; let us write it $\prod [\overline {a_i}, \overline {b_i}]$ where
$a_i$ and the $b_i$ are $t$-bounded and $b_i-a_i$ is $t$-significant for all $i$. Let $B$ be the definable set %FL
$\prod_i [a_i, b_i]\setminus D$. 

We claim 
that every cube contained in the definable
subset $B$ has $t$-negligible volume.
Indeed, let $\Delta=\prod[\alpha_i,\beta_i]$ be such a cube. 
If $x$ is a point of $A_{\mathrm r}^n$
with $\overline x\in \prod \;( %FL
\overline {\alpha_i}, \overline {\beta_i})$   %FL
then $x\in \Delta$
(and hence $x\notin D$) so 
$\prod\; (\overline {\alpha_i}, \overline {\beta_i})$ does not intersect $\overline D$. %FL
\FL{On the other hand, since
$\prod\;  (\overline {\alpha_i}, \overline {\beta_i})$ is contained in 
$\prod\;  [\overline {a_i}, \overline {b_i}]$ (because $\Delta\subset \prod_i[a_i,b_i]$), and 
$\overline D$ contains $\prod_i[\overline{a_i},
\overline{b_i}]$,
the open cube $\prod\;  (\overline {\alpha_i}, \overline {\beta_i})$ is contained in  %FL
$\overline D$. Thus $\prod\;  (\overline {\alpha_i}, \overline {\beta_i})$ is empty, and} there
is at least one index $i$ such that 
$\overline{\beta_i}-\overline{\alpha_i}=0$, which means that $\beta_i-\alpha_i$ is $t$-negligible; a fortiori, the volume of $\Delta$
is $t$-negligible. 

Now by what we have already proved, this implies that $\int_{\prod [a_i,b_i]\setminus D}\d x_1\wedge\ldots\wedge \d x_n$ is
$t$-negligible. As a consequence
\[\int_{\prod [a_i,b_i]\cap D}\d x_1\wedge\ldots\wedge \d x_n=\int_{\prod [a_i,b_i]}\d x_1\wedge\ldots\wedge \d x_n\]
modulo a $t$-negligible element; but $\int_{\prod [a_i,b_i]}\d x_1\wedge\ldots\wedge \d x_n=\prod(b_i-a_i)$, which is
$t$-significant. Thus $\int_{\prod [a_i,b_i]\cap D}\d x_1\wedge\ldots\wedge \d x_n$ is $t$-significant as well, and so is
$\int_D \d x_1\wedge\ldots\wedge \d x_n$. 
\end{proof}

\subsection{}\label{definition-almost-coincidence}
A
definable subset $D$ of $(^*\R)^n$ is called $t$-bounded if
it is contained in $A_{\mathrm r}^n$; it is called 
$t$-negligible if it is $t$-bounded and satisfies the equivalent properties of \Prop \ref{prop-equiv-negli}. We shall
say that two $t$-bounded definable subsets $D$ and $D'$ of $(^*\R)^n$ 
almost coincide (\resp are almost disjoint) if their symmetric difference (\resp their intersection) is $t$-negligible. 
If $D$ is a $t$-bounded definable subset of $(^*\R)^n$, a finite family $(D_i)$ of $t$-bounded definable subsets of $(^*\R)^n$
will be called an almost partition of $D$ if $\bigcup D_i$ is almost equal to $D$ and the subsets $D_i$ are pairwise almost disjoint.

A definable subset $D$ of $R^n$ is called 
negligible if it is of dimension $\leq n-1$. We shall
say that two definable subsets $D$ and $D'$ of $R^n$ 
almost coincide (\resp are almost disjoint) if their symmetric difference (\resp their intersection) is negligible. 
If $D$ is a definable subset of $R^n$, a finite family $(D_i)$ of definable subsets of $R^n$
will be called an almost partition of $D$ if $\bigcup D_i$ is almost equal to $D$ and the subsets $D_i$ are pairwise almost disjoint. 

\begin{lemm}\label{lem-almost-disjoint}
Let $D$ and $\Delta$ be two $t$-bounded definable subsets of $(^*\R)^n$. Then $D$ and $\Delta$ are almost disjoint if and only if
$\overline D$ and $\overline \Delta$ are almost disjoint. 
\end{lemm}

\begin{proof}
If $\overline D$ and $\overline \Delta$ are almost disjoint
then $\overline{D\cap \Delta}\subset \overline D\cap \overline \Delta$
is negligible, so $D\cap \Delta$ is $t$-negligible by \Prop \ref{prop-equiv-negli}.
Conversely, assume that $D\cap \Delta$ is $t$-negligible and let us prove that $\overline D$ and $\overline \Delta$
are almost disjoint. 
We argue by contradiction, so we assume that there exist elements $a_1,\ldots,a_n, b_1, \ldots, b_n$ in $A_{\mathrm r}$
with $b_i-a_i>0$ and $t$-significant for all $i$ such that $\prod [\overline {a_i}, \overline{b_i}]\subset \overline D\cap \overline \Delta$. Set $P=\prod [a_i, b_i]\subset A_{\mathrm r}^n$. The volume of
the cube $P$ is $t$-significant and the volume of $P\cap D\cap \Delta$ is $t$-negligible, so the volume of $P\setminus (D\cap \Delta)=(P\setminus D)\cup (P\setminus \Delta)$ is $t$-significant. So at least one of the two definable sets $P\setminus D$
and $P\setminus \Delta$ has $t$-significant volume. Assume without loss of generality that
$P\setminus D$ has $t$-significant volume. By \Prop \ref{prop-equiv-negli}
there exists $c_1,\ldots,c_n, d_1, \ldots, d_n$ in $A_{\mathrm r}$
with $d_i-c_i>0$ and $t$-significant for all $i$ such that $\prod [c_i, d_i]\subset P\setminus D$.
Set $x=(\frac{c_1+d_1}2, \ldots, \frac{c_n+d_n}2)$. Then $x$ is a point of $P$ whose distance to $D$ is $t$-significant. 
As a consequence, $\overline x\notin \overline D$. But since $x\in P$, its reduction $\overline x$
belongs to $\prod [\overline {a_i}, \overline{b_i}]\subset \overline D\cap \overline \Delta$, contradiction. 
\end{proof}

\begin{prop}\label{prop-bar-almost}
Let
$D$ and $\Delta$ be two $t$-bounded definable 
subsets of $(^*\R)^n$.
\begin{enumerate}[1]
\item The set $D$ is almost equal to $\Delta$
if and only if $\overline D$ is almost equal to $\overline \Delta$. 
\item The set $\overline{D\cap \Delta}$ is almost equal to 
$\overline D\cap \overline \Delta$. 
\end{enumerate}
\end{prop}

\begin{proof}
Set $P=D\setminus \Delta$ and $Q=\Delta\setminus D$.
By Lemma \ref{lem-almost-disjoint}
above, $\overline Q$ and $\overline{D\cap \Delta}$ are almost disjoint, 
and so are $\overline P$ and $\overline{D\cap \Delta}$ as well as $\overline P$ and $\overline Q$. 
Moreover, 
we
have \[\overline D=\overline P\cup \overline{D\cap \Delta}\;\;\text{and}\;\;
\overline \Delta=\overline Q\cup \overline{D\cap \Delta}.\]
Hence $\overline D$ is almost equal to $\overline \Delta$ if and only if $\overline P$ and $\overline Q$ are
negligible, which amounts to requiring
that $P$ and $Q$ be $t$-negligible (\Prop \ref{prop-equiv-negli}), that is to say, that $D$ and $\Delta$ 
almost coincide, whence (1). Moreover, $\overline D\cap \overline \Delta=\overline {D\cap \Delta}\cup (\overline P
\cap \overline Q)$, and in view of the negligibility of $\overline P\cap \overline Q$ this implies (2). 
\end{proof}

\begin{coro}\label{coro-lift-compact}
Let $K$ be a definable definably compact subset of $R^n$. There exists a definable, definably compact
and $t$-bounded subset $E$ of $(^*\R)^n$ such that $\overline E$ almost coincides with $K$.
\end{coro}

\begin{proof}
Choose $a_1,\ldots,a_n$ and $b_1,\ldots, b_n$ in $A_{\mathrm r}$ such that 
$b_i>a_i$ for all $i$ and $K\subset \prod [\overline {a_i}, \overline {b_i}]$. 
By using
the description of definably closed subsets of $R^n$ provided
by Théorème 2.7.1 of \cite{bochnak-c-r1985}, we can assume that there exist finitely many polynomials $f_1,\ldots, f_m$
in $R[T_1,\ldots, T_n]$ such that $K$ is the intersection of $\prod [\overline {a_i}, \overline {b_i}]$
with the set of points $x$ such that $f_j(x)\geq 0$ for all $j$. By \Prop \ref{prop-bar-almost}
above me may assume that $m=1$, and write $f$ instead of $f_1$. If $f$ is constant the set $K$ is either empty or the whole of $ \prod [\overline {a_i}, \overline {b_i}]$ and the statement is obvious. If $f$ is non-constant, let $g$ be a polynomial with $t$-bounded coefficients that lifts $f$. Let $E$ be the intersection of $\prod [a_i,b_i]$ and the non-negative locus of $g$; it suffices to prove that $\overline E$ is almost equal to $K$. By definition, $\overline E\subset K$. Now let $x$ be a point on $K$ at which $f$ is positive, and let $\xi$ be any pre-image of $x$ on $\prod[a_i,b_i]$. Since $f(x)>0$ we have $g(\xi)>0$, hence $\xi \in E$ and $x\in \overline E$. Thus the difference $K\setminus E$ is contained in the zero-locus of $f$, which is at most $(n-1)$-dimensional since $f$ is non-constant. 
\end{proof}

\section{Smooth functions and smooth forms over $R$ and $C$}

\subsection{Smooth functions and smooth forms over the field $R$}
\label{sss-pinfty}
Recall that $A$ denotes the ring of $t$-bounded elements
of $^*\C$, $\mathfrak m$ denotes its maximal ideal
(\ie, the set of $t$-negligible elements)
and $A_{\mathrm r}$ and $\mathfrak m_r$
denote the intersections of $A$ and $\mathfrak m$
with $^*\R$. The reduction modulo $\mathfrak m$
will be denoted by $a\mapsto \overline a$.

\subsubsection{}Let $U$ be a semi-algebraic open subset of $\R^m$, for some $m$.

\paragraph{}
If $x$ is a point of $R^m$ lying on $U(R)$ and
if $\xi$ is any point of $A_{\mathrm r}^m$ lifting $x$, then 
$\xi$ lies on $U(^*\R)$: this comes from the fact that $U$ can be defined by a positive boolean combination of strict inequalities (which follows from Théorème 2.7.1 of \cite{bochnak-c-r1985}). 
For short, we shall call such a $\xi$ a lifting of $x$ in $U(^*\R)$.

\paragraph{}\label{par-def-tame}
Let $\phi$
be a smooth function from $U$ to $\R$. Let $x\in U(R)$. 
We shall say that $\phi$ is \emph{tame}
at $x$
if it satisfies the following condition: for every 
lifting $\xi$
of $x$ in $U(^*\R)$ and
every multi-index $I$, 
the element $\partial^I \phi(\xi)$ of $^*\R$ is $t$-bounded. 

If this is the case, then for every $\xi$ and every $I$ as above, 
the element  $\overline{\partial^I\phi(\xi)}$ of $R$ does not depend on $\xi$
(since $\partial^I\phi$ is Lipshitz with $t$-bounded constant around $\xi$).

\paragraph{}\label{par-tame-sum}
If $\phi$ is tame at $x$, so are all
of its partial derivatives; the sum
and the product of two smooth functions on $U$ that are
tame at $x$ 
are themselves tame at $x$.

\paragraph{}If $\phi$ is tame at $x$, we shall denote by $\phi(x)$ the element
$\overline{\phi(\xi)}$ for $\xi$ any lifting of $x$ in $U(^*\R)$ (it is well defined in 
view of \ref{par-def-tame}).

\subsubsection{Examples}\label{sss-tame-basic}
In each of the following examples, the function $\phi$ is tame
at every point of $U(R)$: 

\begin{itemize}[label=$\diamond$]

\item$U=\C^\times$ (viewed as a semi-algebraic subset of $\C\simeq \R^2$)
and $\phi=\abs \cdot$ ; 
\item 
$U=\R^\times$ and 
$\phi=z\mapsto z^n$ for some $n\in \Z$; 
\item $U=\R_{>0}$ and $\phi=\log$; 
\item $U=\R$ and $\phi$ is any trigonometric polynomial.
\end{itemize}

The function
$x\mapsto \exp(1/x)$ (defined on $\R^\times$)
is not tame at the element $t$
of $^*\R^\times$: indeed, $\exp (1/t)$ of $^*\R$
is not $t$-bounded.

\subsubsection{Composition of tame functions}
\label{sss-compose-tame}
Let $U$ be a semi-algebraic open subset 
of $\R^m$, and let $V$ be a semi-algebraic open subset of $\R^n$. 
Let $\phi=(\phi_1,\ldots, \phi_n)$ be smooth functions from
$U$ to $\R^n$ and assume
that $\phi(U)\subset V$. Let $\psi$ be a smooth function on $V$. 

Let $x$ be a point of $U$ such that every $\phi_i$ is tame at $x$, 
and such that $\psi$ is tame at $\phi(x)$. It follows straightforwardly
from the definition that 
$\psi \circ \phi$
is tame at $x$.

Using this together with \ref{sss-tame-basic}, we see that 
\[\C^\times \to \R, z\mapsto \log \abs z\] is tame at every point of 
$C^\times$, and that 
\[\C^\times\setminus\;\{z,\abs z=1\} \to \R, z\mapsto 1/\log \abs z\]
is tame at every point of $C^\times\setminus \;\{z\in C^\times,\abs z=1\}$.

\subsubsection{Smooth functions and smooth forms on a variety}\label{smoothfunc}
Let $X$ be a smooth separated $R$-scheme of
finite type.

Let $U$ be a semi-algebraic open subset of $X(R)$
and let $g$ be a regular map from a Zariski-open
subset of $X$ containing $U$ to $\A^m_R$ for some $m$. 
A \emph{$(U,g)$-tame} smooth function
is a smooth function $\phi$ defined on some semi-algebraic
open subset $V$ of $\R^m$ with $g(U)\subset V(R)$, such that 
$\phi$ is tame at $g(x)$ for every $x\in U$. 

Let $\mathscr F$ be the assignment that
sends a semi-algebraic open subset $U$ of $X(R)$ to the set of functions from $U$
to $R$ of the form 
$\phi \circ g$, where
$g$ is a regular map from a Zariski-open
subset of $X$ containing $U$ to $\A^m_R$ for some $m$, and 
where $\phi$ is a $(U,g)$-tame smooth function. 

Then $\mathscr F$ is a presheaf; its associated sheaf for the semi-algebraic topology
is denoted by $\mathscr C^\infty$ or $\mathscr C^\infty_X$ and called the sheaf of \emph{smooth functions}
on $X(R)$. It makes $X(R)$ a locally ringed site. 

The natural embedding of
$X(R)$ into the scheme $X$ underlies a morphism of locally ringed
sites $\psi \colon (X(R),\mathscr C^\infty_X)\to (X,\mathscr O_X)$; hence $\psi^*\Omega^p_{X/R}$
is for every $p$ a well-defined $\mathscr C^\infty_X$-module on  $X(R)$, which we denote by $\mathscr A^p$
or $\mathscr A^p_X$. 
The sheaf $\mathscr A^0_X$ is equal to $\mathscr C^\infty_X$, and the $\mathscr C^\infty_X$-module $\mathscr A^1_X$ 
is locally free (of rank $n$ if $X$ is of pure dimension $n$); for every $p$, we have $\mathscr A^p_X=\Lambda^p\mathscr A_X^1$. 
The sheaf $\mathscr A_X^p$ is called the sheaf of \emph{smooth $p$}-forms on $X(R)$. One has for every $p$
a natural differential $\d \colon \mathscr A^p_X\to \mathscr A^{p+1}_X$. 
The sheaf  $C\otimes_R\mathscr A_X^p$ is called the sheaf of \emph{complex-valued}
$p$-forms on $X(R)$. 

\subsection{The case of a variety over $C$}\label{sss-def-cplxforms}
By considering Weil restriction we can apply the above to smooth schemes of finite type
over the field $C$. For such a scheme $X$ and every $m$ we get a sheaf $\mathscr A^m_X$
of $R$-vector spaces on $X(C)$ (equipped with the semi-algebraic topology). This sheaf
comes with a natural decomposition 
\[C\otimes_R \mathscr A^m_X=\bigoplus_{p+q=m}\mathscr A^{p,q}_X,\]
where
$\mathscr A^{p,q}$ is the sheaf of $(i,j)$-forms; \ie, of $C$--valued $p$-forms generated over $\mathscr C^\infty$ by forms of
the type \[\d f_1\wedge \ldots \wedge \d f_p\wedge \d\overline {g_1}\wedge \ldots
\wedge \d \overline{g_q}\]
for some regular functions $f_1,\ldots, f_p,g_1,\ldots, g_q$
(this is analogous to \ref{int-nnform-bigR}).

\subsubsection{Polar coordinates}
The usual real functions $\cos$ and $\sin$ are tame at every point of $R$, 
hence $\theta\mapsto \cos \theta+i\sin \theta$ is a well-defined
smooth $C$-valued function on $R$, which we
denote by $\theta\mapsto e^{i\theta}$.
The map $\theta\mapsto e^{i\theta}$ is a surjective homomorphism from $R$ to $\{z\in C^\times, \abs z=1\}$. The map
$\theta\mapsto e^{i\theta}$ is not injective; its kernel consists of elements of the form $2\pi n$ where $n$ is a (possibly) non-standard integer; \ie, it
can be written as the (class of)
the
limit of a $t$-bounded sequence of integers. For every $a\in R$, the restriction of $\theta\mapsto e^{i\theta}$
to $[a,a+2\pi)$ and $(a,a+2\pi]$ is injective. 

Every element $z$ of $C^\times$ can be written $re^{i\theta}$ with $r\in R_{>0}$ and $\theta\in R$. The element $r$ is unique
(it is equal to $\abs z$), but $\theta$ is not -- we say that $\theta$ is an argument of $z$. 

Making $z$ vary, we get two ``functions" $r$ and $\theta$ on $C^\times =\mathbf G_{\mathrm m}(C)$. More precisely, $r$ is an actual  function which is tame at every point and
takes its values in $R_{>0}$,
and $\d r$ and $\d \log r=\frac{\d r}r$ are well-defined differential forms on $C^\times$. But $\theta$
is only a multivalued function; nevertheless, the differential form $\d \theta$
is also well-defined. Let us quickly explain how. Let $z_0\in C^\times$ and let $a$ be any element of $R$ such that 
$z_0$ has an argument $\theta_0$ in $(a-\pi,a+\pi)$ (this always holds for $a=0$ or $a=\pi$). Then on a suitable semi-algebraic neighborhood $U$ of $z_0$ in $C^\times$ we have a single-valued
smooth argument function $\theta$ with values in $(a-\pi,a+\pi)$ (and $\theta(z_0)=\theta_0$). The
smooth form $\d \theta$ is well-defined on $U$. From the equality $z=re^{i\theta}$ we 
get 
\[\d z=e^{i\theta}\d r +rie^{i\theta}\d \theta,\]
and then
\[\d \theta=-\frac i r e^{-i\theta}\d z -i\frac{\d r}r.\]
This last formula does not involve the choice of $z_0$, $a$ and $\theta_0$ anymore, and we use it to define $\d \theta$ on the whole of $C^\times$. 

If we see $z$ as an invertible function on $C^\times$ we shall write $\dll{\abs z}$ instead of $\frac {\d r}r$ and
$\da z$ instead of $\d \theta$. 

From the equality $z\overline z=r^2$ we get
\[\dll \abs{z}=\frac 1 2 \cdot \frac {2\d r}r=\frac 1 2\left( \frac {\d z}z+\frac {\d \overline z}{\overline z}\right).\]
From the equality $\frac z{\overline z}=e^{2i\theta}$ we get
\[\da z=\frac 1 2 \cdot 2 \d \theta=\frac 1 {2i} \cdot \frac{\d(e^{2i\theta})}{e^{2i\theta}}=\frac 1 {2i} \left( \frac {\d z}z-\frac {\d \overline z}{\overline z}\right).\]

\subsubsection{The definition of $\ds$}
\label{def-d-sharp}
Let $X$ be a smooth scheme of finite type over $C$. Our purpose is to define an operator $\ds$ on complex-valued smooth forms
on $X(C)$ (which is a non-standard avatar of $\mathrm d^{\mathrm c}$ up to a constant).

Let us denote for short by $\mathscr C^\infty_{X,C}$ (\resp $\mathscr A^p_{X,C}$) 
the sheaf $C\otimes_R \mathscr C^\infty_X$ (\resp $C\otimes_R \mathscr A^p_X$).
The sheaf $\mathscr A^1_{X,C}$ of complex-valued smooth $1$-forms on $X(C)$ admits a canonical decomposition
$\mathscr A^1_{X,C}=\mathscr A^{1,0}\oplus \mathscr A^{0,1}$. The formula $(\omega,\omega')\mapsto 
(-i\omega, i\omega')$
defines an order 4 automorphism $\mathrm J$ of the $\mathscr C^\infty_{X,C}$-module
$\mathscr A^1_{X,C}$; we still denote by $\mathrm J$ the induced automorphism of $\mathscr A^p_{X,C}$. 
We remark that $\mathscr A^{2n}_{X,C}\simeq \mathscr A^{n,0}\otimes_{\mathscr C^\infty_{X,C}}\mathscr A^{0,n}$, 
so that the operator $\mathrm J$ on $\mathscr A^{2n}_{X,C}C$ is nothing but $(-i)^ni^n\mathrm{Id}=\mathrm{Id}$.

We then define the derivation $\ds \colon \mathscr C^\infty_{X,C}\to \mathscr A^1_{X,C}$
as being equal to $(\mathrm J\circ \d)/2\pi$
(this is an avatar of the classical 
operator $\mathrm d^{\mathrm c}$); it extends to a compatible system of exterior derivations

\[\ds:=\frac 1{2\pi}\mathrm J\circ d\circ \mathrm J^{-1}\colon \mathscr A^p_{X,C}\to \mathscr A^{p+1}_{X,C}.\]

Let us see how it acts on polar coordinates. We have 
\begin{align*}
\ds (\log r)&=\frac 1{2\pi}\mathrm J(\dll r)\\
&=\frac 1{2\pi}\mathrm J\left( \frac 12\left( \frac {\d z}z+\frac {\d \overline z}{\overline z}\right)\right)\\
&=\frac 1{2\pi}\left(\frac 1 2\left( -i\frac {\d z}z+i\frac {\d \overline z}{\overline z}\right)\right)\\
&=\frac 1{2\pi}\left(\frac 1 {2i}\left(\frac {\d z}z-\frac {\d \overline z}{\overline z}\right)\right)\\
&=\frac{\d \theta}{2\pi}
\end{align*}
and
\begin{align*}
\ds (\theta)&=\frac 1{2\pi}\mathrm J(\d \theta)\\
&=\frac 1{2\pi}\mathrm J\left(\frac 1{2i}\left( \frac {\d z}z-\frac {\d \overline z}{\overline z}\right)\right)\\
&=\frac 1{2\pi}\left(\frac 1 {2i}\left( -i\frac {\d z}z-i\frac {\d \overline z}{\overline z}\right)\right)\\
&=\frac 1{2\pi}\left(\frac 1 2\left(-\frac {\d z}z-\frac {\d \overline z}{\overline z}\right)\right)\\
&=-\frac{\dll r}{2\pi}.
\end{align*}

Note that since $(\ds)^2=0$ this implies that $\ds (\dll r)=0$ and $\ds(\d \theta)=0$. 

More generally if $f$ is an invertible regular function defined on some Zariski-open subset $U$ of $X$
we can define $\dll \abs f$ and $\da f$. Those are smooth forms on $U(C)$ and we have the following equalities
\begin{align*}
\dll \abs f&=\frac 1 2\left(\frac {\d f}f+\frac{\d \overline f}{\overline f}\right)\\
\da  f&=\frac 1 {2i}\left(\frac {\d f}f-\frac{\d \overline f}{\overline f}\right)\\
\ds(\log \abs f)&=\frac{\da f}{2\pi}\\
\ds(\arg f)&=-\frac{\dll \abs f}{2\pi}.
\end{align*}

\subsection{}
Now we introduce a particular class of smooth functions and forms on
smooth schemes over $C$
that will play a crucial role in our work. 
Roughly speaking, these are the functions and forms that have a natural counterpart in the Berkovich setting - we will make this rather vague
formulation more precise later.

\begin{defi}\label{def-psmooth}
Let $V$ be an open subset of $(\R\cup\{-\infty\})^m$ which can be defined by a boolean combination of
$\Q$-linear inequalities
and let $\phi$
be a function
from $V$ to $\C$.
We shall say that $\phi$ is a
\emph{reasonably smooth} function
if there exists:
\begin{itemize}[label=$\diamond$]
\item a finite open cover $(V_i)_i$ of $V$, 
where each $V_i$ is also defined by $\Q$-linear inequalities; 
\item for every $i$, a 
subset $J_i$ of $\{1, \ldots, m\}$ with
$\Omega_i:=p_{J_i}(V_i)\subset \R^{J_i}$, where $p_{J_i}$ is the projection onto the coordinates
belonging to $J_i$); 
\item for every $i$, a smooth function $\phi_i$ on $\Omega_i$ 
such that $\phi|_{V_i}=\phi_i\circ p_{J_i}|_{V_i}$. 
\end{itemize}
The data $(V_i, J_i, \Omega_i, \phi_i)_i$ will be called a \emph{nice description}
of $\phi$. 

If $J$ is some subset of $\{1,\ldots, m\}$ we shall say that
$\phi$ is \emph{$J$-vanishing}
if there exists an open subset $V'$ of $V$ satisfying the following: 
\begin{itemize}[label=$\diamond$]
\item $V'$ can be defined by $\Q$-linear inequalities; 
\item $\phi|_{V'}=0$;
\item for every $x=(x_1,\ldots, x_m)\in V\setminus V'$ and every $i\in J$, the coordinate $x_i$ 
is not equal to $(-\infty)$.
\end{itemize}

Note that $\phi$ is automatically $\emptyset$-vanishing; 
indeed, if $J=\emptyset$ then the above conditions are fulfilled by $V'=\emptyset$.
\end{defi}

For instance, a
reasonably smooth function $\phi$ on $\R\cup\{-\infty\}$ is nothing but a smooth function $\phi$
on $\R$ such that there exists
$\lambda \in \R$ with $\phi(x)=\lambda$ for $x\ll 0$ (and the value of $\phi$ at $-\infty$ is then set equal to $\lambda$); it is $1$-vanishing if and only if $\lambda=0$. 

\subsection{}
Let $V$ be an open subset of $(\R\cup\{-\infty\})^m$ which can be defined by a boolean combination of
$\Q$-linear inequalities. The following facts follow straightforwardly from the definition. 

\subsubsection{}
If $\phi\colon V\to \R$ is a reasonably smooth function, then it is continuous, and $\phi|_{V\cap \R^m}$ is smooth.

\subsubsection{}
For $V\subset \R^m$, a function from $V$ to $\R$ is  reasonably smooth if and only if it is smooth. %FL

\subsubsection{}\label{sss-pseudo-der}
The set of reasonably smooth functions on $V$ is a subalgebra of the algebra of $\R$-valued functions on $V$.
It is endowed with 
partial derivation operators defined in the obvious way. 

Let $\phi$ be a reasonably smooth function on $V$ that is $J$-vanishing for some
subset $J$
of $\{1,\ldots, m\}$. Let $j\in J$. Let us show that $\partial_j \phi$ is $(J\cup\{j\})$-vanishing. 

Let $(V_i, J_i,\Omega_i,\phi_i)_i$ be a nice description of $\phi$ and let $V'$
be an open subset of $V$ that witnesses
the fact that $\phi$ is $J$-vanishing.
Let $V''$ be the union of $V'$ and of all the open sets $V_i$ such that 
$j\notin J_i$. We claim that $V''$ witnesses the fact
that $\partial_j \phi$
is $(J\cup\{j\})$-vanishing. Indeed, $\partial_j \phi$
is zero on $V'$ since so is $\phi$; and if $i$ is such that $j\notin J_i$ then $\phi|_{V_i}$
dose not depend on the $j$-th coordinate, so $\partial _j\phi$ is zero on $V_i$; thus
 $\partial_j \phi$ is zero on $V''$. 
 
Let $x\in V\setminus V''$; choose $i$ such that $x=(x_1,\ldots, x_m)\in V_i$. By definition of $V''$, the set 
$J_i$ contains $j$. Hence $x_j\neq (-\infty)$, whence our claim.

\subsection{Smooth functions and smooth forms on a $C$-scheme: a fundamental example}
Let $V$ be an open subset of $(\R\cup\{-\infty\})^m$ which can be defined by a boolean combination of
$\Q$-linear inequalities,
and let $\phi$
be a reasonably smooth function
from $V$ to $\C$.

Let 
$W$ be the
semi-algebraic open subset of $\C^{m+1}$
consisting of points 
$(a_1,\ldots,a_m, b)$ such that $0<\abs b<1$  and $(-\log \abs{a_i}/\log
\abs b)_i \in V$. 
By construction,
\[\Phi
\colon
(a_1,\ldots, a_m,b)\mapsto \phi(-\log \abs {a_1}/\log \abs b, \ldots,- \log
\abs {a_m}/\log \abs b)\]
is a well-defined $\mathscr C^\infty$ map from $W$ to $\C$. 

Let $X$ be a smooth $C$-scheme of finite type and let $U$ be a semi-algebraic open
subset of $X(C)$. Let $g=(g_1,\ldots, g_m)$ be a regular map from a Zariski-open subset of $X$ containing $U$
to $\A^m_C$, and assume that
$(g_1,\ldots, g_m,t)(U)\subset W(C)$
(here the element $t$ of $C$ is viewed as a constant regular function).

\subsubsection{The smooth function $\Phi$ on $W$
is $(U,(g_1,\ldots, g_m,t))$-tame}
To see it, fix a nice description 
$(V_i, J_i,\Omega_i, \phi_i)_i$ of $\phi$.
For every $i$, denote by 
$W_i$ the pre-image of $V_i$ in $W$
under the map $(a_1,\ldots, a_m,b)\mapsto (-\log \abs{a_j}/\log \abs b)$, 
and let $U_i$ denote the pre-image of
$W_i$ in $U$ under the map 
$(g_1,\ldots, g_m,t)$. 

We fix $i$, and we are going to show
that $\Phi$ is
$(U_i, (g_1,\ldots, g_m,t))$-tame,
which will imply %FL
our claim. In view of \ref{par-tame-sum}
and \ref{sss-compose-tame}, it suffices to prove that for every $x\in U_i$, 
the map $\phi_i$ is tame
at the point $y:=(-\log \abs{g_j(x)}/\log \abs t)_{j\in J_i}$ of $\Omega_i(R)$.
But the coordinates of $y$ are bounded
(as is $\log r / \log \abs t$ for every $r\in R_{>0}$),
so the coordinates of $\eta$ are bounded for
every lifting $\eta$ of $y$, which implies that all partial derivatives of $\phi_i$
are bounded, and a fortiori $t$-bounded, at $\eta$; thus $\phi_i$ is tame at $y$.

\subsubsection{}
We can thus compose $\Phi$ and
$(g_1,\ldots, g_m,t)$
to get a smooth map on $U$, 
which we can safely write 
\[x\mapsto \phi(-\log {\abs{g_1}}/\log \abs t, \ldots, -\log \abs{g_m}/\log \abs t);\]
its restriction to every $U_i$ can be written 
\[x\mapsto \phi_i(-\log {\abs{g_j}}/\log \abs t)_{j\in J_i}.\]

\subsubsection{}\label{sss-typical-smoothform}
Let $I$ and $J$ be two subsets of $\{1,\ldots, m\}$ of respective cardinalities $p$
and $q$ such that $\phi$ is $(I\cup J)$-vanishing. 

Let $U'$ be the pre-image of $V'$ in $U$ under
$(-\log \abs{g_1}/\log \abs t, \ldots, -\log \abs{g_m}/\log \abs t)$, 
and let $U''$ be the subset of $U$ consisting of points at which all the functions $g_i$
with $i\in I\cup J$ are invertible. 
Let $\omega$ be the $(p,q)$-form on $U''$ equal to
\[\left(\frac{-1}{\log \abs t}\right)^p\phi(-\log {\abs{g_1}}/\log \abs t, \ldots, -\log \abs{g_m}/\log \abs t)
\dll \abs{g_I}\wedge \da g_J\]
(where $\dll  \abs{g_I}
=\dll \abs{g_{i_1}}\wedge\ldots\wedge \dll \abs{g_{i_p}}$
if $i_1<i_2<\ldots<i_p$ are the elements of $I$, and similarly for
$\da\abs{g_J}$). 
Since $\phi$ is $(I\cup J)$-vanishing, the restriction of $\omega$
to $U'\cap U''$ is zero, so that $\omega$ and the zero form on $U'$ glue to 
a $(p,q)$-form on $U$ which (obviously) does not depend
on $V'$; we shall allow ourselves to denote it
by
\[\left(\frac{-1}{\log \abs t}\right)^p\phi(-\log {\abs{g_1}}/\log \abs t, \ldots, -\log \abs{g_m}/\log \abs t)
\dll \abs{g_I}\wedge \da g_J.\]

\section{Integrals of smooth forms over $R$ and $C$}

\subsection{}
The purpose of this section is to integrate forms on a smooth scheme defined over the field $R$. The  rough idea is quite natural
(and unsurprising): lift the situation over $A_{\mathrm r}$, compute the integral over $^*\R$ like in section
\ref{section3}, and then take its class modulo the ideal $\mathfrak m_{\mathrm r}$ of $t$-negligible elements. 

First of all, we shall assume that we are given two different liftings of a very specific form, and show that the  integrals over $^*\R$ 
to
which they give rise 
coincide modulo $\mathfrak m_{\mathrm r}$ (\Prop \ref{prop-key-integration}
below); the proof rests in a crucial way on our former study of cubes with $t$-negligible volume and uses the notion of ``almost equality"
over $^*\R$ as well as over $R$ (see \Prop \ref{prop-equiv-negli}, \ref{definition-almost-coincidence},
and \Prop \ref{prop-bar-almost}), together with Hensel's lemma. 

Then we shall handle the general case, the point being that a form on a smooth $R$-scheme always admits locally for the Zariski topology a lifting
of the kind dealt with by \Prop \ref{prop-key-integration}; so this part is somehow tedious but rather formal once \Prop
\ref{prop-key-integration}
is taken for granted.

\subsection{}
If $\mathscr X$ is an affine
$A_{\mathrm r}$-scheme of finite type, a definable subset $E$
 of $\mathscr X(^*\R)$ will be called \emph{$t$-bounded}
 if it is contained in $\mathscr X(A_{\mathrm r})$. We remark that $E$ is $t$-bounded if
and only if
its topological closure
is $t$-bounded, and if this is the case then the latter is even
definably compact. Indeed, by embedding $\mathscr X$ in an affine space and arguing componentwise
 we reduce to the case where $\mathscr X=\A^1_{A_{\mathrm r}}$, 
 for which our statement follows from o-minimality.

\begin{prop}\label{prop-key-integration}
Let $Z$ be a smooth $R$-scheme of finite type and pure dimension $n$, and let
$h=(h_1,\ldots, h_n)$ be an étale map $Z\to \mathbf A^n_R$
factorizing through an immersion
$(h, h_{n+1})\colon Z\hookrightarrow \A^{n+1}_R$.
Let $\mathscr X$ and $\mathscr Y$ be two smooth
$\mathrm{A_r}$-schemes of finite type
and of pure relative dimension $n$, equipped with identifications
$\mathscr X_R\simeq Z$ and $\mathscr Y_R\simeq Z$. Let $f=(f_1,\ldots, f_n) \colon \mathscr X\to
\mathbf A^n_{A_{\mathrm r}}$ and $g=(g_1,\ldots, g_n) \colon \mathscr Y\to
\mathbf A^n_{A_{\mathrm r}}$ be two étale maps, factorizing 
respectively through a closed immersion $(f, f_{n+1}) \colon \mathscr X\hookrightarrow
\mathbf A^{n+1}_{A_{\mathrm r}}$
and 
$(g, g_{n+1}) \colon \mathscr Y\hookrightarrow
\mathbf A^{n+1}_{A_{\mathrm r}}$; assume that for all $i$ one has $f_i|_Z=g_i|_Z=h_i$. 

Let
$E$, resp. $F$, be a $t$-bounded 
semi-algebraic
subset of $\mathscr X(^*\R)$, resp. $\mathscr Y(^*\R)$; assume that
the subsets $\overline E$ and $\overline F$ of $Z(R)$ almost coincide. 

Let $\phi$ be a smooth function defined on a neighborhood of $E$ in $\mathscr X(A_{\mathrm r})$, 
of the form $\phi_0\circ \lambda$ with $\phi_0$ a $\mathscr C^\infty$ function and $\lambda$ a tuple 
of regular functions on $\mathscr X$; let $\psi$ be a smooth function defined on a neighborhood of $F$ in $\mathscr Y(A_{\mathrm r})$, 
of the form $\psi_0\circ \mu$ with $\psi_0$ a $\mathscr C^\infty$ function and $\mu$ a tuple 
of regular functions on $\mathscr Y$. Assume that
there exists a semi-algebraic open subset $O$ of $Z(R)$ containing
$\overline E$
and $\overline F$ such that $\phi_0$ is
$(O,\lambda|_Z)$-tame, $\psi_0$ is $(O,\mu|_Z)$
tame, and the smooth functions $\phi_0\circ (\lambda|_O)$ 
and $\psi_0\circ (\mu|_O)$ 
coincide on some semi-algebraic
subset of $O$
almost equal to $\overline E$ and $\overline F$. 

Then
$\int_E \phi \d f_1\wedge \ldots \wedge \d f_n$
and $\int_F \psi \d g_1\wedge \ldots \wedge \d g_n$
are $t$-bounded and coincide up to a $t$-negligible element.
\end{prop}

\begin{proof}
We begin with noting that our tameness assumptions on $\phi_0$, \resp $\psi_0$, 
imply
that $\phi$, \resp $\psi$, takes only $t$-bounded values on $E$,
\resp $F$; this in turn implies that it is \emph{uniformly}
$t$-bounded on $E$, \resp $F$. The $t$-boundedness of the integrals
involved in our statement follows immediately. 

Throughout the proof, we will
use the map $f_{n+1}$, resp. $g_{n+1}$, resp. $h_{n+1}$ to see 
any fiber of $f$, resp. $g$, resp. $h$, 
as
a subset of the affine line over its ground field, and we will
repeatedly use the following fact, which is a consequence of the Henselian property
of the local ring $A_{\mathrm r}$ : if $w$ is a point of $A_{\mathrm r}^n$ with image $\overline w$ in
$R^n$, then 
for every $z\in Z(R)$ lying above $\overline w$ there exists a unique pre-image $\zeta$
of $w$ in $\mathscr X(A_{\mathrm r})$ (resp. $\mathscr Y(A_{\mathrm r})$) with $\overline \zeta=z$.

The subsets $\overline E$ and $\overline F$ of $Z(R)$ are definable, closed and bounded (because $E$ and $F$ are bounded); so they are definably compact.
The sets $h(\overline E)$ and $h(\overline F)$ are definably compact, and they almost coincide since $\overline E$ and $\overline F$ almost coincide.
so they have the same $n$-dimensional locus $\Theta$; and the set $h(\overline E\bigtriangleup \overline F)$ 
is negligible. 
It follows that there exists an almost partition $(\Theta_i)$ of $\Theta$ (and thus of $h(\overline E)$
as well as of $h(\overline F)$)
by definably compact definable subsets satisfying the following: for every $i$ there exists an integer $n_i$ such that 
the subset $\Theta'_i$ of $\Theta_i$ consisting of points having exactly $n_i$ pre-images in $\overline E\cap \overline F$ and no pre-image
in $\overline E\bigtriangleup \overline F$
is almost equal to $\Theta_i$. 

Now for every $i$ there exists a $t$-bounded definable subset $\Omega_i$ of $(^*\R)^n$ such that %FL
$\overline {\Omega_i}$ is almost equal to $\Theta_i$ (Corollary \ref{coro-lift-compact}). By \Prop \ref{prop-bar-almost} the family $(\Omega_i)$ is an almost partition of $f(E)$ as well as of $g(F)$. For every $i$, let $\Omega'_i$ be the subset of $\Omega_i$
consisting of
points
having exactly $n_i$ pre-images in $E$ under $f$
and exactly $n_i$ pre-images in $F$
under $g$.

\subsubsection{}\label{sss-number-of-sections}
Let us fix $i$, and prove that $\Omega'_i$ is almost equal to $\Omega_i$. 
It is sufficient (since $f$ and $g$ play exactly the same role) to prove that the set $H$
of points of $\Omega_i$ having
exactly $n_i$ pre-images in $E$ under $f$ is almost equal to $\Omega_i$.

We argue by contradiction, so we assume that the set $H$ consisting of points $x\in \Omega_i$ such that  $f\inv(x)\cap E$
has cardinality different from $n_i$ has $t$-significant volume. Then its image $\overline H$ is a non-negligible subset of $\Theta_i$, which implies that $\overline H\cap \Theta'_i$ has dimension $n$. Let us choose a cube (with non-zero volume) $\mathcal C$ in $\overline H\cap \Theta'_i$ having the following property: there exist an integer $N$, a subset $I$ of $\{1,\ldots, N\}$ of cardinality $n_i$ and a $t$-bounded element $A>1$ in $^*\R$ such that each fiber of $h$ over $\mathcal C$ consists of exactly $N$ points $z_1<z_2<\ldots <z_N$ all contained in $[1-\overline A, \overline A-1]$ and  such that $z_j\in \overline E$ if and only if $j\in I$.

Let us choose a cube $\mathcal D\subset A_{\mathrm r}^n$ lifting $\mathcal C$. Since $\mathcal C\subset \overline H$, the intersection $\mathcal D\cap H$ is not $t$-negligible
(Lemma \ref{lem-almost-disjoint}), hence contains a cube $\mathcal D'$ with $t$-significant volume. Every point of $\mathcal D'$ has exactly $N$ $t$-bounded pre-images, all contained in $[-A,A]$; let $\sigma_1<\ldots <\sigma_N$ denote the corresponding continous sections of the étale map $f$ above $\mathcal D'$. If $x\in \mathcal D'$ and if $j\in \{1,\ldots, N\}\setminus I$ then $\sigma_j(x)\notin E$, because $\overline{\sigma_j(x)}\notin \overline E$ by the very definition of $I$. 
For each $j\in I$, set $\mathcal D'_j=\sigma_j\inv(E)$. Let $\xi\in \overline {\mathcal D'}$ and let $j\in I$. The point $\xi$ belongs to $\mathcal C$, so its $j$-th pre-image $\zeta$ under $h$ belongs to $\overline E$, so there exists a point $z\in E$ such that $\overline z=\zeta$, which implies that 
$\overline{f(z)}=\xi$; thus $z=\sigma_j(f(z))$ and $f(z)$ belongs to $\mathcal D'_j$; as a consequence, 
$\overline{\mathcal D'_j}=\overline{\mathcal D'}$. In view of \Prop \ref{prop-bar-almost}, it follows that 
$\overline{\bigcap_{j\in I}\mathcal D'_j}$ is almost equal to $\overline{\mathcal D'}$. In particular
$\bigcap_{j\in I}\mathcal D'_j$ is non-empty; but for every $x\in \bigcap_{j\in I}\mathcal D'_j$
the intersection $f\inv(x)\cap E$ has exactly $n_i$ elements, contradiction. 

\subsubsection{}\label{sss-negligible-integral}
Now we remark that if $\mathcal N$ is a $t$-negligible $t$-bounded definable subset of $(^*\R)^n$, then
\[\int_{E\cap f\inv(\mathcal N)} \phi \d f_1\wedge \ldots \wedge \d f_n\;\;\text{and}\;\;
\int_{F\cap g\inv(\mathcal N)} \psi \d g_1\wedge \ldots \wedge \d g_n\]
are $t$-negligible. 
Indeed, let $N$ be an integer such that the fibers of $f|_E$ 
and of $g|_F$ all have cardinality $\leq N$, and let $M$ be a $t$-bounded positive element such that $\abs \phi$ and $\abs \psi$ are bounded by $M$ on $E$ and $F$ respectively. 

Then
\[\left|\int_{E\cap f\inv(\mathcal N)} \phi \d f_1\wedge \ldots \wedge \d f_n\right|\leq NM\int_{\mathcal N}\d T_1\wedge \ldots \wedge \d T_n\]
and
\[\left|\int_{F\cap g\inv(\mathcal N)} \psi \d g_1\wedge \ldots \wedge \d g_n\right|\leq NM\int_{\mathcal N}\d T_1\wedge \ldots \wedge \d T_n,\]
whence our claim. 

\subsubsection{Conclusion}
In view of \ref{sss-number-of-sections} and \ref{sss-negligible-integral} it is sufficient to prove that for all $i$ the integrals
\[\int_{E\cap f\inv(\Omega'_i)} \phi \d f_1\wedge \ldots \wedge \d f_n\;\;\text{and}\;\;
\int_{F\cap g\inv(\Omega'_i)} \psi \d g_1\wedge \ldots \wedge \d g_n\]
agree up to a $t$-negligible element. So let us fix $i$. We denote by 
$\sigma_1<\sigma_2<\ldots <\sigma_{n_i}$ the continuous sections of $f|_E$
over $\Omega'_i$ and by $\tau_1<\tau_2<\ldots <\tau_{n_i}$ the continuous sections of $g|_F$
over $\Omega'_i$. For all $x\in \Omega'_i$ and all $j$ between $1$ and $n_i$ the elements $\overline{\sigma_j(x)}$ and $\overline{\tau_j(x)}$ coincide: both are the $j$-th pre-image of $\overline x$ in $\overline E\cap \overline F$. 
We have by construction 
\[\int_{E\cap f\inv(\Omega'_i)} \phi \d f_1\wedge \ldots \wedge \d f_n=\sum_j \int_{\Omega'_i}(\phi\circ \sigma_j) \d T_1\wedge \ldots \d T_n\]
and
\[\int_{F\cap g\inv(\Omega'_i)} \psi \d g_1\wedge \ldots \wedge \d g_n=\sum_j \int_{\Omega'_i}(\psi\circ \tau_j) \d T_1\wedge \ldots \d T_n.\]
The difference 
\[\int_{E\cap f\inv(\Omega'_i)} \phi \d f_1\wedge \ldots \wedge \d f_n-
\int_{F\cap g\inv(\Omega'_i)} \psi \d g_1\wedge \ldots \wedge \d g_n\]
is thus equal to 
\[\sum_j \int_{\Omega'_i}(\phi\circ \sigma_j-\psi \circ \tau_j)\d T_1\wedge \ldots \wedge \d T_n.\]
By our assumptions on $\psi$ and $\psi$ the difference $\abs{\phi\circ \sigma_j-\psi \circ \tau_j}$
is $t$-negligible for every $j$ at every point of $\Omega'_i$. Therefore there exists a positive $t$-negligible
element $\epsilon$ such that $\abs{\phi\circ \sigma_j-\psi \circ \tau_j)}\leq \epsilon$ for all $j$ at every point of $\Omega'_i$. 
As a consequence 
\begin{align*}
&\left|\int_{E\cap f\inv(\Omega'_i)} \phi \d f_1\wedge \ldots \wedge \d f_n-
\int_{F\cap g\inv(\Omega'_i)} \psi \d g_1\wedge \ldots \wedge \d g_n\right|\\
&\leq  n_i \epsilon \int_{\Omega'_i}\d T_1\wedge \ldots \wedge \d T_n,
\end{align*}
which ends the proof.
\end{proof}

\begin{coro}\label{coro-negligible-general}
Let $\mathscr X$
be a smooth
$A_{\mathrm r}$-scheme of finite type
of pure relative dimension $n$
Let $f=(f_1,\ldots, f_n) \colon \mathscr X\to
\mathbf A^n_{A_{\mathrm r}}$ be an étale map factorizing 
through an immersion $(f, f_{n+1}) \colon \mathscr X\hookrightarrow
\mathbf A^{n+1}_{A_{\mathrm r}}$. 
Let
$E$ be a $t$-bounded
semi-algebraic
subset of $\mathscr X(^*\R)$; 
we remind the reader that $\overline E$ denotes the
image of $E$ under the reduction map (and
not its topological closure).
The following are equivalent: 

\begin{enumerate}[i]
\item The image $f(E)$ is $t$-negligible. 
\item The image $f(\overline E)$ is of dimension $<n$.
\item The reduction $\overline E$ is of dimension $<n$. 
\item For every smooth function $\phi$ of the form $\phi_0\circ \lambda$ with $\phi_0$ a $\mathscr C^\infty$ function and $\lambda$ a tuple 
of regular functions on $\mathscr X$
such that $\phi_0$ is $(O,\lambda|_{\mathscr X_R})$-tame
on some semi-alegbraic open subset $O$ of $\mathscr X(R)$ containing
$\overline E$, the integral $\int_E \phi \d f_1\wedge \ldots \wedge \d f_n$
is $t$-negligible. 
\item The integral $\int_E \d f_1\wedge \ldots \wedge \d f_n$
is $t$-negligible. 
\end{enumerate}
\end{coro}

\begin{proof}
Implication (i)$\Rightarrow$(ii)
comes from the fact that $f(\overline E)=\overline{f(E)}$. Implication (ii)$\Rightarrow$(iii)
comes from étaleness of $f$. Implication (iii)$\Rightarrow$(iv) follows from 
Proposition \ref{prop-key-integration} (apply it with $\mathscr Y=\mathscr X, g=g, g_{n+1}=f_{n+1}$
and $F=\emptyset$). Implication (iv)$\Rightarrow$(v) is obvious. 
Assume that (v) holds. For every $i$, let $D_i$ denote the subset of $f(E)$ consisting of points having
exaclty $i$ pre-images on $E$, and let $N$ be such that $D_i=\emptyset$ for $i>N$. 
We then have
\[
\int_E  \d f_1\wedge \ldots \wedge \d f_n=\sum_{i=1}^N \int_{D_i}\d T_1\wedge \ldots \wedge \d T_n.\]
As a consequence, $\int_{D_i}\d T_1\wedge \ldots \wedge \d T_n$ is $t$-negligible for every $i$, 
so $\int_{f(E)} \d T_1\wedge \ldots \wedge \d T_n$ is $t$-negligible, whence (i).
\end{proof}

\subsection{}
Let us keep the notation of Corollary \ref{coro-negligible-general}
above. We shall say that $E$ is $t$-negligible if it satisfies the equivalent conditions (i)--(v)
(note that condition (iii) does not involve the functions $f_i$, so the notion of $t$-negligibility does
not depend on the choice of the functions $f_i$). 
We shall say that two $t$-bounded definable subsets of $\mathscr X(^*\R)$ almost coincide if their
symmetric difference is $t$-negligible, and that two definable subsets of $\mathscr X(R)$ almost coincide if their
symmetric difference is of dimension $<n$.

\begin{lemm}\label{lem-almost-disjoint2}
Let $\mathscr X$
be a smooth
$A_{\mathrm r}$-scheme of finite type
of pure relative dimension $n$. Assume that there exists 
an étale map $f=(f_1,\ldots, f_n) \colon \mathscr X\to
\mathbf A^n_{A_{\mathrm r}}$ factorizing 
through an immersion $(f, f_{n+1}) \colon \mathscr X\hookrightarrow
\mathbf A^{n+1}_{A_{\mathrm r}}$. Let $E$ and $F$ be two $t$-bounded
definable subsets of $\mathscr X(^*\R)$. Then $E$ and $F$ are almost disjoint if and only if
$\overline E$ and $\overline F$ are almost disjoint. 
\end{lemm}

\begin{proof}
If $\dim (\overline  E\cap \overline F)<n$ then $\dim \overline {E\cap F}<n$ because  $\overline {E\cap F}\subset \overline E\cap\overline F$; thus
if
$\overline E$ and $\overline F$ are almost disjoint, so are $E$ and $F$. 
Assume now that $E$ and $F$ are almost disjoint. Set $G=f(E\cup F)$.
For every triple $(i,j,k)$, denote by $\overline G_{i,j,k}$ the subset of points of $\overline G$
having $i$ pre-images in $\overline
E$, $j$ in $\overline F$ and $k$ in $\overline E\cup \overline F$. 
By Corollary \ref{coro-lift-compact}
there exists for every $(i,j,k)$ a $t$-bounded definably compact definable subset
$\Gamma_{i,j,k}$ of $(^*\R)^n$ such that $\overline{\Gamma_{i,j,k}}$ is almost equal to 
the definable closure of $\overline G_{i,j,k}$, hence is also almost equal
to  $\overline G_{i,j,k}$. By the same reasoning as in \ref{sss-number-of-sections}, 
the subset of points of $\Gamma_{i,j,k}$ having exactly $i$ pre-images in $E$, \resp $j$ pre-images
in $F$, \resp $k$ pre-images in $E\cup F$ is almost equal to $\Gamma_{i,j,k}$; hence so is the intersection 
$\Gamma'_{i,j,k}$ of these three subsets. The family $(\Gamma'_{i,j,k})$ is an almost
partition of $G$. 

Let $(i,j,k)$ be a triple with $k<i+j$. Since $\overline E\cap \overline F$ has dimension $<n$, the set 
$\overline{G_{i,j,k}}$ is negligible; as a consequence, $\Gamma_{i,j,k}$ and $\Gamma'_{i,j,k}$ are $t$-negligible. 
This implies that $f(E\cap F)$ is $t$-negligible, whence the $t$-negligibility of $E\cap F$ itself. 
\end{proof} 

\begin{prop}\label{prop-bar-almost2}
Let $\mathscr X$ be as in Lemma \ref{lem-almost-disjoint2}
above and let $E$ and $F$ be two $t$-bounded definable subsets of $\mathscr X(^*\R)$. 

\begin{enumerate}[1]
\item The set $E$ is almost equal to $F$ if and only if $\overline E$ is almost equal to
$\overline F$. 
\item The set $\overline{E\cap F}$ is almost equal to $\overline E\cap \overline F$.
\end{enumerate}
\end{prop}

\begin{proof}
The proof is the same as that of Proposition \ref{prop-bar-almost}, except that one uses
Lemma \ref{lem-almost-disjoint2}
instead of Lemma \ref{lem-almost-disjoint}. 
\end{proof}

\begin{coro}\label{coro-lift-compact2}
Let $\mathscr X$ be as in Lemma \ref{lem-almost-disjoint2}
and let $K$ be a definable definably compact subset of $\mathscr X(R)$. There exists a definable, definably compact
and $t$-bounded subset $E$ of $\mathscr X(^*\R)$ such that $\overline E$ almost coincides with $K$.
\end{coro}

\begin{proof}
By writing $K$ as the union of its intersections with the Zariski-connected components of $\mathscr X_R$, 
we can assume that it lies on such a component $X$. 
By boundedness of $K$ and the henselian property of $A_{\mathrm r}$
(which ensures that any $R$-point of $\mathscr X$ can be lifted to an $A_{\mathrm r}$-point),
we can choose a $t$-bounded, definably compact definable subset $F$ of $\mathscr X(^*\R)$
such that $K\subset \overline F\subset X(R)$. 
By Theorem 2.7.1 of \cite{bochnak-c-r1985}, we can assume that there exists finitely many regular functions $f_1,\ldots, f_m$
on $\mathscr X_R$ such that $K$ is the intersection of $\overline F$
with the set of points $x$ such that $f_j(x)\geq 0$ for all $j$. By \Prop \ref{prop-bar-almost2}
above me may assume that $m=1$, and write $f$ instead of $f_1$. If $f$ is constant on $X$
the set $K$ is either empty or the whole of $\overline F$ and the statement is obvious. If $f$ is non-constant on $X$, let $g$ be a
regular function on $\mathscr X$ that lifts $f$. Let $E$ be the intersection of $F$ and the non-negative locus of $g$; it suffices to prove that $\overline E$ is almost equal to $K$. By definition, $\overline E\subset K$. Now let $x$ be a point on $K$ at which $f$ is positive, and let $\xi$ be any pre-image of $x$ on $F$. Since $f(x)>0$ we have $g(\xi)>0$, hence $\xi \in E$ and $x\in \overline E$. Thus the difference $K\setminus E$ is contained in the zero-locus of $f$ in $X(R)$ which is at most $(n-1)$-dimensional since $f|_X$ is non-constant. 
\end{proof}

\begin{defi}\label{def-x-liftable}
Let $X$ be a smooth $R$-scheme of finite type and of pure dimension $n$. We shall say for short that $X$ is \emph{liftable}
if there exists a smooth affine $A_{\mathrm r}$-scheme $\mathscr X$, an isomorphism $\mathscr X_R\simeq X$, and
$n+1$ regular functions $f_1,\ldots, f_{n+1}$ on $\mathscr X$ such that $(f_1,\ldots,f_{n+1})$ defines an
immersion $\mathscr X\hookrightarrow \A^{n+1}_{A_{\mathrm r}}$ and $(f_1,\ldots,f_n)\colon \mathscr X\to \A^n_{A_{\mathrm r}}$ is étale. 
\end{defi}

\subsection{Integral of a smooth $n$-form}\label{subsec-integ}
Let $X$ be a smooth $R$-scheme of finite type and of pure dimension $n$,
let $K$ be a definable subset of $X(R)$ with definably compact definable closure, and let $\omega$
be a smooth $n$-form on a semi-algebraic
open
neighborhood $O$ of $K$ in $X(R)$. 
The purpose of what follows is to define $\int_K \omega$. 

\subsubsection{}
We first assume that $X$ is liftable
and $\omega$ is of the form $\phi(u_1,\ldots, u_m)\omega_0$ almost everywhere on $K$, with  $u_i$  regular functions,  $\phi$ an 
$(O,(u_1,\ldots, u_m))$-tame smooth function, and where $\omega_0$ is an algebraic $n$-form. Choose $\mathscr X$ and
$f_1,\ldots, f_{n+1}$ as in Definition \ref{def-x-liftable}. The sheaf $\Omega_{X/R}$ is then free with basis
$(\d f_i|_X)_{1\leq i\leq n}$; therefore up to multiplying $\phi$ with a regular function we might assume
that $\omega_0=(\d f_1\wedge \ldots\wedge \d f_n)|_X$. 

Choose a $t$-bounded definable subset $E$ of $\mathscr X(^*\R)$ such that $\overline E$ is almost equal
to the definable closure of $K$ (Corollary \ref{coro-lift-compact2}) and for every $i$, choose a regular function
$v_i$ on $\mathscr X$ lifting $u_i$. 

By Proposition \ref{prop-key-integration}, the integral $\int_E \phi(v_1,\ldots, v_m)\d f_1\wedge \ldots \d f_n$
does not depend on our various choices up to a $t$-negligible element. We can thus set
\[\int_K\omega =\overline{\int_E \phi(v_1,\ldots, v_m)\d f_1\wedge \ldots \d f_n}\;;\]
this is an element of $R$. 
Note that if $K'$ is any definable subset almost equal to $K$ then $\int_{K'}\omega=\int_K \omega$
(since the same $E$ can be used for both computations).

The assignment $K\mapsto \int_K \omega$ is finitely additive. Indeed, if $K$ is a finite union $\bigcup_{j\in J} K_j$
of definable subsets, we can choose for every $j$ an almost lifting $E_j$ of $K_j$; now for every subset $I$ of $J$
the sets $\overline{\bigcap_{j \in I}E_j}$ and $\bigcap_{j\in I}K_j$ almost coincide by Proposition \ref{prop-bar-almost2}, 
and additivity follows from additivity of integrals over the field $^*\R$. 

\subsubsection{}
We still assume that $X$ is liftable, but we no longer
assume that $\omega$ is of the form $\phi(u_1,\ldots, u_m)\omega_0$ on $K$.
By the very definition of an $n$-form there exist finitely many definably open subsets $U_1,\ldots, U_r$ of $X(R)$ that cover $K$
and such that $\omega|_{U_i}$ has the required form.
By Lemma \ref{lem-compact-open-cover} we can write the definable closure of $K$ as a finite union $\bigcup_{j\in J} K_j$ with each $K_j$ definably compact and contained in $U_j$. 
By additivity $\sum_{\emptyset \neq I\subset J}(-1)^{\abs I+1}\int_{\bigcap_{j\in I}K_j}\omega$
does not depend on the choice of the sets $U_j$ and  $K_j$, and we can use this formula as a definition for $\int_K \omega$.
The assignment $K\mapsto \int_K \omega$ remains additive in this more general setting, and $\int_K\omega$ only depends on the class
of $K$ modulo almost equality.

\subsubsection{}\label{sss-x-ds-integrals}
We still assume that $X$ is liftable. Let $s$ be an algebraic function on $X$, set $X'=D(s)$
(the invertibility locus of $s$)
and assume that
the definable closure of $K$ is contained in $X'(R)$. We then have a priori two different definitions for $\int_K\omega$, the one using
$X$ and the other one using  the principal open subset $X'$, which is also (obviously) liftable. Let us show that both integrals coincide. 
By replacing $K$ by its closure (to which it is almost equal) we can assume that it is definably compact. 

By cutting $K$ into finitely many
sufficiently small pieces
(Lemma \ref{lem-compact-open-cover}) and using additivity, we can assume that $\omega$ is of the form $\phi(u_1,\ldots, u_m)\omega_0$ almost everywhere on $K$,
with $u_i$ regular functions on $X$, $\phi$ an 
$(O,(u_1,\ldots, u_m))$-tame smooth function,  and $\omega_0$ a section of $\Omega^n_{X/R}$ (this can be achieved since $\Omega_{X/R}$ is free because $X$ is liftable). Lift every $u_i$ to a regular function $v_i$ on $\mathscr X$, and lift $\omega_0$ to a relative $n$-form $\omega'$ on $\mathscr X$. 

Let us choose data $(\mathscr X, f_1,\ldots, f_{n+1})$ that witness the liftability of $X$.
Lift every $u_i$ to a regular function $v_i$ on $\mathscr X$, lift $\omega_0$ to a relative $n$-form $\omega'$ on $\mathscr X$, and lift
$s$ to a regular function 
$\sigma$ on $\mathscr X$.
Set $\mathscr X'=D(\sigma)$. Then $(\mathscr X', f_1,\ldots, f_n, f_{n+1})$
witnesses the liftability of $\mathscr X'$. Now choose a $t$-bounded definable subset of $\mathscr X'(^*\R)$
that almost lifts $K$. Then it is definable, $t$-bounded and an almost lifting of $K$ as a subset of $\mathscr X(^*\R)$ as well. 
Therefore the $X$ and the $X'$ version of $\int_K\omega$ both are equal to the class of $\int_E \phi(v_1,\ldots, v_m)
\omega'$ modulo the $t$-negligible elements. 

\subsubsection{}
The scheme $X$ is no longer assumed to be liftable, but we assume that there exist two liftable affine open subsets $X'$ and $X''$ of $X$ such that the definable closure of $K$ is contained in $X'(R)\cap X''(R)$. We then have a priori two different definitions for $\int_K\omega$, the one using
$X'$ and the other one using $X''$. We want to prove that they coincide. By replacing $K$ by its closure (to which it is almost equal) we can assume that it is definably compact.

Let us first note the following. Let $x$ be a point of $X'\cap X''$. Choose an affine neighborhhood $Y$
of $x$ in $X'\cap X''$ equal to $D(s)$ as a subset of $X'$, for some regular function $s$ on $X'$. Now choose an affine neighborhood 
$Z$ of $x$ in $Y$ equal to $D(\FL{w})$ as a subset of $X''$, for some regular function $\FL{w}$ on $X''$. The restriction of $\FL{w}$ to   %FL
$Y$ is equal to $a/s^\ell$ for some $\ell\geq 0$ and some regular function $a$ on $X'$; as a consequence $Z=D(as)$
as a subset of $X''$. 

Hence we can cover $X'\cap X''$ by finitely many
open subschemes $Y_1,\ldots, Y_r$, each of which is principal in both $X'$
and $X''$. Now write $K=\bigcup K_i$ with every $K_i$ definable, definably compact and contained in $Y_i$ (Lemma \ref{lem-compact-open-cover}). For every non-empty subset $I$ of $\{1,\ldots, r\}$ it follows
from \ref{sss-x-ds-integrals} that $\int_{\bigcap_{i\in I}K_i}\omega$ does not
depend whether one is working with $X'$ or or $X''$
(because it can be computed working with $Y_j$ where $j$ is any element of $I$). 
By additivity it follows that $\int_K\omega$ also does not
depend whether one is working with $X'$ or or $X''$. 

\subsubsection{}
Now let us explain how to define $\int_K\omega$ in general. Let $K'$ be the closure of $K$, which is definably compact. We choose a finite cover $(X_i)_{i\in I}$ of $X$ by liftable open subschemes (which is possible since $X$ is smooth). We then  write $K'$ as a finite union $\bigcup K_i$ where every $K_i$
is a definably compact semi-algebraic
subset of $X_i(R)$ (Lemma \ref{lem-compact-open-cover}). 

We then set 
\[\int_K \omega=\sum_{\emptyset\neq J \subset I}
(-1)^{\abs J+1}\int_{\bigcap_{i\in J}K_i}\omega,\]
which makes sense because,
as it follows straightforwardly
from the above,
it
does not depend on $(X_i)$ nor on $(K_i)$.

\subsection{}
Let $X$ be a smooth $R$-scheme of finite type of pure dimension $n$. 
It follows from our construction that
\[(K,\omega)\mapsto \int_K \omega\]
(where $K$ is a semi-algebraic
subset of $X(R)$ with definably
compact closure and $\omega$ is an $n$-form defined
on a
definable neighborhood of $K$) is $\R$-linear in $\omega$, additive in $K$, 
and that it depends on $K$ only up to almost equality.

We can extend this definition to forms with coefficients in a reasonable class of functions (like piecewise smooth) by requiring everywhere
in the above that $\phi$ belongs to the involved class (instead of being smooth) and 
satisfies some tameness condition. For instance, $\int_K \abs \omega$ makes sense (and is non-negative), see \ref{sss-def-absomega}. It follows from the definition that $\int_K\omega$ only depends on $\omega|_K$; in particular, it is zero if $\omega$ vanishes almost everywhere on $K$. We can thus extend the definition of $\int_K\omega$ when we only assume that there exists a definable subset $K'$ of $K$ with definably compact closure such that $\omega$ vanishes on $K\setminus K'$. 

And of course, we can also define  by linearity the integrals of complex-valued forms~(\ref{smoothfunc}). 

\subsection{The complex case}
We now consider a smooth $C$-scheme of finite type $X$ of pure dimension $n$, and a complex-valued $(n,n)$-form $\omega$ with coefficients belonging to a reasonable class of functions defined
in a semi-algebraic open
neighborhood of a
semi-algebraic
subset $K$ of $X(C)$. Assume that there exists a semi-algebraic subset $K'$ of $K$ with definably compact closure such that $\omega$ vanishes on $K\setminus K'$. 

Then $\int_K \omega$ 
is well defined. Its computation requires (amongst other things) to lift locally $\mathrm R_{C/R}X$ to a smooth $A_{\mathrm r}$-scheme
and $\omega$ to a $(2n)$-form on this scheme, 
which can be achieved by lifting locally $X$ to a smooth $A$-scheme and $\omega$ to an $(n,n)$-form on this scheme. 

\section{The archimedean and non-archimedean complexes of forms}

\subsection{}
We denote by $\lambda$ the element $-\log \abs t$ of $R_{>0}$, 
and by $\loga$ the normalized logarithm
function $a\mapsto \log a /\lambda$ from $R_{>0}$ to $R$. 

We recall that $C$ is equipped with a non-archimedean absolute value $\val\cdot$, which 
sends a non-zero element $z$ to $\tau^{\std(\frac {\log \abs z}{\log \abs t})}$ where $\tau$ is an element of $(0,1)$ which has been fixed once and for all,
and where $\std(\cdot)$ denotes
the standard part
(see \ref{ss-field-c}). We set $\lambda_\flat=-\log \tau =-\log \val t\in \R_{>0}$. 
 and we denote by $\logb$ the
 normalized logarithm
 function $a\mapsto \log a /\lambda_\flat$ from $\R_{>0}$ to $\R$. 
 
 If $a$ is any element of $C^\times$, it follows from the definitions that 
 
 \[\logb \val a=\std(\loga \abs a).\] 

\subsection{Analytification of $C$-schemes}
\label{aff-neighb-weierstr}
The field $C$ is a complete
non-archimedean field, so 
Berkovich geometry
makes sense
over it. 

Let $X$ be a $C$-scheme of finite type, and let $X\an$
denote
its Berkovich analytification.
Let $x$ be a point of $X\an$. In the proof of our
main theorem, we shall use the fact that 
$x$ has a basis of open, \resp affinoid, neighborhoods $V$
in $X\an$  satisfying the following: there exists
an affine open subscheme $\Omega$ of $X$ such that $V$ is an open subset, \resp a Weierstraß~ domain, of $\Omega\an$ that admits a description
by a system of inequalites of the form 
\[\val {\phi_1}<R_1,\ldots, \val{\phi_n}<R_n, \text{\resp}\;
\val {\phi_1}\leq R_1,\ldots, \val{\phi_n}\leq R_n\]
where the functions $\phi_i$ belong to $\mathscr O(\Omega)$,
and with $R_i$ positive real numbers. 

Let us prove it. We first chose
an open affine subscheme $U$ of $X$ with $x\in U\an$, a family $(f_1,\ldots, f_n)$ of regular functions
on $U$ that generate $\mathscr O(U)$ over $C$, and let $R$ be a positive real number such that $\val{f_i(x)}<R$ for all $i$; let $W$
be the Weierstraß~domain of $U\an$ defined by the inequalities $\val{f_i}\leq R$.
Now it follows from the general theory of Berkovich
spaces that $x$ has a basis of open, \resp affinoid,
neighborhoods described by a system of inequalities of the form
\[\val {f_1}< R, \ldots, \val{f_n}<R, \val{g_1}< r_1,\ldots, \val{g_m}< r_m,\val{h_1}>s_1,\ldots, \val{h_\ell}> s_\ell,\]
\[\text{\resp}\; \val {f_1}\leq  R, \ldots, \val{f_n}\leq R, \val{g_1}\leq  r_1,\ldots, \val{g_m}\leq r_m,
\val{h_1}\geq s_1,\ldots, \val{h_\ell}\geq s_\ell\]
with $g_i$ and $h_i$ analytic functions on $W$, and  $s_i$ and $r_i$ positive real numbers. But
$\mathscr O(U)$ is dense in $\mathscr O(W)$, so we can assume by approximation that the functions $g_i$ and  $h_i$ belong
to $\mathscr O(U)$. Then the domain described by the above system of inequalities
can also be described as the locus of validity of
 \[\val {f_1}<R, \ldots, \val{f_n}<R, \val{g_1}< r_1,\ldots, \val{g_m}<r_m,\val{h_1\inv}  %FL
 <s_1\inv ,\ldots, \val{h_\ell\inv}< s_\ell\inv, \]
 \[\text{\resp}\;\;\val {f_1}\leq R, \ldots, \val{f_n}\leq R, \val{g_1}\leq r_1,\ldots, \val{g_m}\leq r_m,\val{h_1\inv}
 \leq s_1\inv ,\ldots, \val{h_\ell\inv}\leq s_\ell\inv \]
 on $D(h_1h_2\ldots h_\ell)\an$, whence our claim.

\subsection{Two complexes of differential forms}
We fix a smooth $C$-scheme of finite type $X$ of pure dimension $n$. 

\subsubsection{}
Let us begin with some notation. Let $U$ be an open subscheme
of $X$ and let $f=(f_1,\ldots, f_m)$ be a family of regular functions on $U$. Let $I$ and $J$ be two subsets of $\{1,\ldots, m\}$. 
We shall denote by $\mathscr S^{I,J,(f_i)}$ the set of pairs
$(V,\phi)$ where:

\begin{enumerate}[a]

\item $V$ is an open subset of $(\R\cup\{-\infty\})^m$, defined by $\Q$-linear inequalities such that $V_R$
contains $(\loga\abs{f_1}, \ldots,\loga  \abs{f_m})(U(C))$;
\item $\phi$ is a reasonably smooth
function on $V$ which is $(I\cup J)$-vanishing.
\end{enumerate}
We identify two pairs $(V,\phi)$ and $(V',\phi')$ if $\phi$ and $\phi'$ agree on $V\cap V'$; therefore
we shall most of the time omit to mention $V$ and elements
of $\mathscr S^{I,J,(f_i)}$ 
will be called $(I\cup J))$-vanishing
reasonably smooth functions. 

We denote by $\mathscr S^{I,J,(f_i)}_\flat$ the set of pairs 
$(V,\phi)$ satisfying condition

\begin{itemize}
\item[$(\mathrm a_\flat)$]
$V$ is an open subset of $(\R\cup\{-\infty\})^m$, defined by $\Q$-linear inequalities and containing
$(\logb\val{f_1}, \ldots,\logb  \val{f_m})(U\an)$
\end{itemize}
and condition (b) above. 
Here also, we  identify two pairs $(V,\phi)$ and $(V',\phi')$ if $\phi$ and $\phi'$ agree on $V\cap V'$ 
and
elements
of $\mathscr S^{I,J,(f_i)}_\flat$ 
will be called $(I,J)_\flat$-vanishing smooth functions.

Note that $\mathscr S^{I,J,(f_i)}
\subset \mathscr S^{I,J,(f_i)}_\flat$.

\subsubsection{The non-standard archimedean complex}\label{nsac}
Let $U$ be a Zariski open subset of $X(C)$. Let us denote by $\mathsf A^{p,q}_{\mathrm{presh}}(U)$ the set of those $(p+q)$-smooth forms
$\omega$ on $U(C)$ for which there exist: 

\begin{itemize}[label=$\bullet$]
\item a finite family $(f_1,\ldots, f_m)$ of regular functions on  $U$ ;  
\item for every pair $(I,J)$ with $I$ and $J$ two subsets of $\{1,\ldots, m\}$ of respective cardinality $p$ and $q$, 
an $(I\cup J)$-vanishing 
reasonably smooth function $\phi_{I,J}\in \mathscr S^{I,J,(f_i)},$ 
\end{itemize}
such that
\[\omega=\sum_{I,J} \phi_{I,J}\left(\loga \abs{f_1},\ldots, \loga \abs{f_m}\right)
\dl\abs{f_I}\wedge \dArg{f_J}\]
where $\dl\abs f_I$ stands for
$\dl \abs{f_{i_1}}\wedge\ldots\wedge\dl \abs{f_{i_p}}$ if $i_1<i_2<\ldots<i_p$ are the elements
of $I$, and \FL{$\dArg{f_J}$ stands for
$\dab{f_{j_1}}\wedge\ldots\wedge\dab{f_{j_q}}$ if $j_1<j_2<\ldots<j_q$ are the elements
of $J$}. 

We denote by $\mathsf A^{p,q}$ the sheaf on $X^{\mathrm{Zar}}$ associated to the presheaf $\mathsf A^{p,q}_{\mathrm{presh}}$, and
by $\mathsf A^{\bullet, \bullet}$ the direct sum $\bigoplus_{p,q} \mathsf A^{p,q}$. 
% (we insist on externality because we do not know whether the 
%internal sum of the $\mathsf A^{p,q}$'s inside the sheaf of smooth forms is direct). 

We set for short $\mathsf A^0=\mathsf A^{0,0}$. By construction, 
$\mathsf A^0(X)$ is the subsheaf (of $C$-algebras)
of the push-forward of $C\otimes_R \mathscr C^\infty_X$
to $X^{\mathrm{Zar}}$, whose sections are the smooth functions
that are locally on $X^{\mathrm{Zar}}$ of the form 
$\phi(\loga \abs{f_1},\ldots, \loga \abs{f_m})$ for some
finite family $(f_1,\ldots, f_m)$ of regular functions
and some reasonably smooth
function 
$\phi$ on a suitable
open subset of $(\R\cup\{-\infty\})^m$. 

The sheaf $\mathsf A^{\bullet, \bullet}$ has a natural structure of bi-graded $\mathsf A^0$-algebra; it follows
from \ref{sss-pseudo-der}
that
the differentials $\d$ and
$\ds$ induce two differentials on $\mathsf A^{\bullet, \bullet}$, which are still denoted by $\d$ and $\ds$. The differential $\d$ is
of bidegree $(1,0)$ and maps a 
form 
\[\phi\left(\loga \abs{f_1}, \ldots, \loga \abs{f_m}\right)
\dl\abs{f_I}\wedge \dArg{f_J}\]
to 
\[\sum_{1\leq i\leq m}\frac{\partial \phi}{\partial x_i}\left(\loga \abs{f_1},\ldots,
\loga \abs{f_m}\right)\dl \abs{f_i}\wedge \dl \abs{f_I}\wedge \dArg{f_J}.\]

The differential $\ds$ is
of bidegree $(0,1)$ and maps a 
form 
\[\phi\left(\loga \abs{f_1},\ldots, \loga \abs{f_m}\right)
\dl\abs{f_I}\wedge \dArg{f_J}\]
to 
\[\sum_{1\leq i\leq m}\frac{\partial \phi}{\partial x_i}\left(\loga \abs{f_1}, \ldots,
\loga \abs{f_m}\right)\daa{f_i}\wedge \dl \abs{f_I}\wedge \dArg{f_J}.\]

The operator $\mathrm J$ also acts on $\mathsf A^{\bullet, \bullet}$; it maps a form 
\[\phi\left(\loga \abs{f_1},\ldots,\loga \abs{f_m}\right)
\dl\abs{f_I}\wedge \dArg{f_J}\]
to
 \[(-1)^q
 (2\pi)^{p-q}\phi\left(\loga \abs{f_1},\ldots, \loga \abs{f_m}\right)
 \dArg{f_I}\wedge \dl\abs{ f_J}\]
and acts trivially on $\mathsf A^{n,n}$.

\subsubsection{The non-archimedean complex}\label{nac}
Let $U$ be a Zariski-open subset of
$X$.
Let us denote by $\mathsf B^{p,q}_{\mathrm{presh}}(U)$ the set of those
$(p,q)$-smooth forms 
$\omega$ on $U\an$ in the sense of \cite{chambertloir-d2012} for which there exist:
\begin{itemize}[label=$\bullet$]
\item a finite family $(f_1,\ldots, f_m)$ of regular functions on  $U$;
\item for every pair $(I,J)$ with $I$ and $J$ two subsets of $\{1,\ldots, m\}$ of respective cardinality $p$ and $q$, 
an $(I,J)_\flat$-vanishing
reasonably smooth function $\phi_{I,J}
\in \mathscr S^{I,J,(f_i)}_\flat$,
\end{itemize}
such that

\[\omega=\sum_{I,J} \phi_{I,J}\left(\logb \val{f_1},\ldots, \logb \val {f_m}\right)
\di \logb\val{f_I}\wedge \dc \log \val{f_J}\]
where $\di \logb \val{f_I}$ stands for
$\di \logb \val{f_{i_1}}\wedge\ldots\wedge\di \logb \val{f_{i_p}}$ if $i_1<i_2<\ldots<i_p$ are the elements
of $I$, and similarly for $\dc \log \val{f_J}$. 

We denote by $\mathsf B^{p,q}$ the sheaf on $X^{\mathrm{Zar}}$ associated to the presheaf $\mathsf B^{p,q}_{\mathrm{presh}}$. We denote
by $\mathsf B^{\bullet, \bullet}$ the 
direct sum $\bigoplus_{p,q} \mathsf B^{p,q}$. 
We set for short $\mathsf B^0=\mathsf B^{0,0}$. By construction, 
$\mathsf B^0$ is
the subsheaf (of $C$-algebras)
of the push-forward of $C\otimes_R A^0_{X\an}$
to $X^{\mathrm{Zar}}$, whose sections are the smooth functions
that are locally on $X^{\mathrm{Zar}}$ of the form 
$\phi(\logb \val{f_1},\ldots, \logb \val{f_m})$ for some
finite family $(f_1,\ldots, f_m)$ of regular functions
and some reasonably smooth
function 
$\phi$ on a suitable
open subset of $(\R\cup\{-\infty\})^m$.

The sheaf $\mathsf B^{\bullet, \bullet}$ is a bi-graded $\mathsf B^0$-algebra which is stable under $\di$ and $\dc$. 

\begin{rema}\label{rem-difference-withCLD}
Every $(p,q)$-form in the sense of
\cite{chambertloir-d2012}
can be written locally \emph{for the Berkovich topology}
 as a sum
\[\sum \psi_{I,J}(\log \val{g_i},\ldots, \log \val{g_m})
\di \log \val{g_I}\wedge \dc \log \val{g_J}\]
where the $\psi_{I,J}$ are smooth and with
$g_i$ \emph{invertible}
analytic functions. 

By the very definition of an $(I\cup J)$-vanishing reasonably
smooth function, a section 
\[\omega=\sum_{I,J} \phi_{I,J}\left(\logb \val{f_1},\ldots, \logb \val {f_m}\right)
\di \logb\val{f_I}\wedge \dc \log \val{f_J}\]
of $\mathsf B^{p,q}_{\mathrm{presh}}$
fulfills this condition, because locally for the Berkovich topology, every non-zero term of the sum
can be rewritten by involving only the functions $f_i$ which are invertible.
But the reader should be aware that $\omega$
can not in general be written
locally \emph{for the Zariski topology}
as a sum
\[\sum \psi_{I,J}(\log \val{g_i},\ldots, \log \val{g_m})
\di \log \val{g_I}\wedge \dc \log \val{g_J}\]
with $g_i$ invertible algebraic functions. 

(Consider for example
a non-zero smooth function $\phi$ on $\R$
that vanishes on $(-\infty, A)$ for some $A$,  and the section $\phi(\logb \val T)
\di \logb \val T \wedge \dc\log \val T$
of $\mathsf B^{1,1}$ on $\mathbf A^{1,\mathrm{an}}_C$.)
\end{rema}

\section{Pseudo-polyhedra}

The purpose of this section is to describe the domains on which we shall integrate our forms, in both the archimedean
and non-archimedean settings. These domains will be the preimages under functions of the form $\loga \abs f$ (\resp $\logb \val f$) of some
specific subsets of $(R\cup\{-\infty\})^n$ (\resp $(\R\cup\{-\infty\})^n$) that we  call \emph{pseudo-polyhedra}.

\begin{defi}\label{def-pseudo-polyhedron}
Let $S$ be a non-trivial divisible ordered abelian group
with additive
notation
(in practice we shall consider only cases where $S$ underlies a real-closed field).
A subset of $(S\cup\{-\infty\})^m$ 
is called a \emph{pseudo-polyhedron}
if it is a finite union of sets of the form

\[\left\{(x',x'')\in \prod_{i\in I}[-\infty,b_i]\times \prod_{i\in J}[a_i,b_i]
\;\text{s.t.}\;\phi_1(x'')\leq 0,\ldots, \phi_r(x'')\leq 0\right\}\]
where
\begin{itemize}[label=$\diamond$]
\item $I$ and $J$  are  subsets of $\{1,\ldots, m\}$ 
that partition it; 
\item  for $1 \leq i \leq m$, $a_i$ and  $b_i$ are elements of $S$; 
\item for $1 \leq j \leq r$, $\phi_j$ is an affine form whose linear part has
coefficients in $\Q$.
\end{itemize}

A subset of $S^m$ is a \emph{polyhedron}
if this is a pseudo-polyhedron of $(S\cup\{-\infty\})^m$. This amounts
to requiring that it be a finite union of sets of the form
\[\left\{x\in \prod_{i\in \{1,\ldots, n\}}[a_i,b_i]
\;\text{s.t.}\;\phi_1(x)\leq 0,\ldots, \phi_r(x)\leq 0\right\},\]
with  $a_i$, $b_i$ and  $\phi_i$ as above. 
\end{defi}

\subsubsection{}
Let $X$ be an analytic space over $C$ and let $f_1,\ldots, f_m$ be analytic functions on $X$. 
Let $P$ be a pseudo-polyhedron of $(\R\cup\{-\infty\})^m$. The set
\[(\logb \val {f_1}, \ldots, \logb \val{f_m})\inv(P)\] is a closed analytic domain of $X$.

\subsubsection{}
\label{sss-tp-definable}
Let $P$ be a pseudo-polyhedron of
$(R\cup\{-\infty\})^m$. The subset 
\[\abs t^{-P}:=\{\abs t^{-x}, x\in P\}\]
(with the convention that  $\abs t^\infty=0$)
is an {\sc rcf}-definable subset of $R_{\geq 0}^m$; indeed, 
it is defined by monomial inequalities. One sees easily that if
$P$ depends
{\sc doag}-definably on some set of parameters 
$a_1,\ldots, a_\ell\in R$ then $\abs t^{-P}$ depends {\sc rcf}-definably 
on $\abs t^{a_1},\ldots, \abs t^{a_\ell}$.

\subsubsection{}
In practice, we shall encounter pseudo-polyhedra 
over the real closed fields $\R$ and $R$. 

\paragraph{}
Let $P\subset (\R\cup\{-\infty\})^m$ be a pseudo-polyhedron over $\R$. It gives rise by
base-change to a pseudo-polyhedron over $P_R\subset (R
\cup\{-\infty\})^m$ over the field $R$ which has the following properties: it can be written as a finite union of subsets of 
$(R
\cup\{-\infty\})^m$ admitting a description like in definition 
\ref{def-pseudo-polyhedron} with the additional requirement that 
all the elements  $a_i$ and $b_i$ are bounded; we shall say for short
that such 
a pseudo-polyhedron is bounded. 

\paragraph{}
Let $\Pi$ be a bounded pseudo-polyhedron in $(R\cup\{-\infty\})^m$.
For every $x$ in $R\cup\{-\infty\}$ which is either
negative unbounded or equal to $-\infty$
we set $\std(x)=-\infty$; with this convention, 
the definition 
\[\std(\Pi):=\{(\std(x_1),\ldots, \std(x_m))\}_{(x_1,\ldots, x_m)
\in \Pi}\]
makes sense, and $\std(\Pi)$ 
is a pseudo-polyhedron of $(\R\cup\{-\infty\})^m$.

To see this, we can assume that $\Pi$ is of the form
\[\left\{(x',x'')\in \prod_{i\in I}[-\infty,b_i]\times \prod_{i\in J}[a_i,b_i]
\;\text{s.t.}\;\phi_1(x'')\leq 0,\ldots, \phi_r(x'')\leq 0\right\}\]
where the notation is as in Definition \ref{def-pseudo-polyhedron} 
and where the elements $a_i$ and  $b_i$ are all bounded. 
Set \[\Theta=\left\{x\in \prod_{i\in J}[a_i,b_i]
\;\text{s.t.}\;\phi_1(x)\leq 0,\ldots, \phi_r(x)\leq 0\right\}.\]
This is a bounded polyhedron of $R^J$ and 
one 
has \[\std(\Pi)=\left(\prod_{i\in I}[-\infty, \std(b_i)]\right)\times \std(\Theta).\]
So it suffices to prove that $\std(\Theta)$ is a polyhedron. Otherwise said, 
we can assume that $I=\emptyset$ and $J=\{1,\ldots, m\}$ and it suffices to show
that $\std(\Pi)$ is a polyhedron. 

In fact we shall prove more generally that $\std(\Pi)$ is a polyhedron when $\Pi$ is any bounded {\sc doag}-definable 
subset of $R^m$. We use induction on $m$; there is nothing to prove if $m=0$. Assume now that $m>0$ and that the 
result holds for integers $<m$. By cell decomposition for an  o-minimal theory, we can assume that $\Pi$ is an open cell. 
So there exists an open {\sc doag}-definable subset $\Delta$ of $R^{m-1}$ and two {\sc doag}-definable functions $\lambda$ 
and $\mu$ from $\Delta$ to $R$ such that $\lambda<\mu$ on $\Delta$ and 
$\Pi$ is equal to the set of those $m$-uples $(x_1,\ldots, x_m)\in R^m$
such
that 
\[(x_1,\ldots, x_{m-1})\in \Delta
\;\text{and}\;
\lambda(x_1,\ldots, x_{m-1})<x_m<\mu(x_1,\ldots, x_{m-1})\}.\]
Up to refining the original cell decomposition, we can even assume that $\lambda$ and $\mu$ are affine with their linear parts
having coefficients in $\Q$. 

Since the cell $\Pi$ is bounded, its projection $\Delta$ onto $R^{m-1}$ is bounded as well, and the constant terms of both $\lambda$ 
and $\mu$ are bounded too, thus the standard parts $\std(\lambda)$ and $\std (\mu)$ make sense as affine functions from $\R^{m-1}\to \R$, 
with linear parts having coefficients in $\Q$. 

Now a direct computation shows that $\std(\Pi)$  is equal to the set of those $m$-uples $(x_1,\ldots, x_m)\in \R^m$
such
that 
\[(x_1,\ldots, x_{m-1})\in \std(\Delta)
\;\text{and}\;
\std(\lambda)(x_1,\ldots, x_{m-1})\leq x_m\leq \std(\mu)(x_1,\ldots, x_{m-1})\}.\]
Since $\std(\Delta)$ is a polyhedron of $\R^{m-1}$ by our induction hypothesis, we are done.

\subsection{}
\label{computation-log}
Let $U$ be a Zariski-open subset of $X$, let $g_1,\ldots, g_\ell$ be regular functions on $U$ and let $P$ be a 
pseudo-polyhedron 
of $(\R\cup\{-\infty\})^\ell$.
Let $Q$
be the closed analytic
domain $(\logb \val g)^{-1}(P)$ of $U\an$
(with $g=(g_1,\ldots, g_\ell)$).
A point $x$ of $U(C)$ belongs to $Q$ 
if and only if $\logb \val{g(x)}\in P$, which is equivalent to
\[\frac{-\log \left(\tau^{\std\left(\frac{\log \abs {g(x)}}{\log \abs t}\right)}\right)}{\log \tau}\in P,\]
which we may rewrite as
\[-\std\left(\frac{\log \abs {g(x)}}{\log \abs t}\right)\in P\]
or equivalently as
\[\loga \abs{g(x)}\in P_R+\mathfrak n^m\]
where we denote by $\mathfrak n$ the set of negligible elements
of $R$.

%\begin{align*}
%&\frac{-\log \left(\tau^{\std\left(\frac{\log \abs {g(x)}}{\log \abs t}\right)}\right)}{\log \tau}\in P\\
%&\iff-\std\left(\frac{\log \abs {g(x)}}{\log \abs t}\right)\in P\\
%&\iff\loga \abs{g(x)}\in P_R+\mathfrak n^m\end{align*}
%where we denote by $\mathfrak n$ the set of negligible elements
%of $R$.

\begin{enonce}[remark]{Notation}
If $\Pi$ is a pseudo-polyhedron of $(R\cup\{-\infty\})^\ell$ for some $\ell$ and if $a$ is a non-negative
element of $R$ we shall denote by $\Pi_a$ the
pseudo-polyhedron $\Pi+[-a,a]^\ell$. If $\Pi$ and $a$ are bounded then $\Pi_a$ is bounded as well. 
\end{enonce}

\begin{lemm}
\label{lem-u-g}
Let $X$ be a $C$-scheme of finite type, 
let $g\colon X\to \A^{\ell}_C$ %FL
be
a morphism,
and let $\Pi$ be a bounded
pseudo-polyhedron of $(R\cup\{-\infty\})^\ell$ .
The following are equivalent:
\begin{enumerate}[i]
\item the analytic domain $(\logb \val g)\inv(\std(\Pi))$
of $X\an$ is compact; 
\item there exists a positive \emph{standard}
number $\epsilon$ such that the
semi-algebraic subset
$(\loga \abs g)\inv(\Pi_\epsilon)$ of $X(C)$ is definably compact. 

\end{enumerate}
\end{lemm}

\begin{proof}
Choose a finite affine open cover $(X_i)$ of $X$ and for each $i$,
a finite family $(f_{ij})$
of regular functions on $X_i$ that generate $\mathscr O_X(X_i)$
as a $C$-algebra. 
For every $i$ and every positive bounded element $M$ of $R$
(resp. every positive real number $M$), denote by 
$K_i^M$ (resp. $K_{i\flat}^M$) the subset of $X_i(C)$
consisting of points at which
$\loga \abs{f_{ij}}\leq M$ for all $j$ (resp. 
the subset of $X_i\an$
consisting of points at which
$\logb
\val{f_{ij}}\leq M$ for all $j$). 
For every positive real number $M$ and every positive
real number $\epsilon$
we have the inclusions
\[
K_i^M\subset
K_{i\flat}^M(C)\subset K_i^{M+\epsilon}.\]

Assume that (i) holds. 
As
$(\logb \val g)\inv(\std(\Pi))$ is compact, 
it is contained  in $ \bigcup_i K_{i\flat}^M$
for some positive real number $M$. 

Let
$a$
be a positive infinitesimal element of $R$.
The 
subset
$(\loga \abs g)\inv(\Pi_a)$ of $X(C)$ 
is contained in $(\logb \val g)\inv(\std(\Pi_a))
=
(\logb \val g)\inv(\std(\Pi))$; 
it is thus contained in the definably compact
semi-algebraic subset $\bigcup_i K_i^{M+1}$. 

Let $I$ be the set of positive elements $a$
of $R$ such that  $(\loga \abs g)\inv(\Pi_a)
\subset \bigcup_i K_i^{M+1}$. This is a
definable subset of $R_{>0}$ which contains by the 
above every positive infinitesimal element;
thus it contains some standard positive element $\epsilon$. 
The semi-algebraic subset $(\loga \abs g)\inv(\Pi_\epsilon)$ of 
$X(C)$ is closed by its very definition, and is contained 
in the definably compact semi-algebraic subset
$\bigcup_i K_i^{M+1}$
by the choice of $\epsilon$, so 
it
is definably compact; thus (ii) holds. 

Conversely, assume that (ii) holds.
Then there exists a positive real
number $M$
such that 
$(\loga \abs g)\inv(\Pi_\epsilon)\subset
\bigcup_i K_i^M$. 

The set of of $C$-points of 
$(\logb \val g)\inv(\Pi)$ is contained in 
$(\loga \abs g)\inv(\Pi_\epsilon)$, 
hence in $\bigcup_i K_i^M$. The latter is itself contained
in the set of $C$-points of $\bigcup_i K_{i\flat}^M$; 
thus $(\logb \val g)\inv(\Pi)\subset
\bigcup_i K_{i\flat}^M$, which implies that $(\logb \val g)\inv(\Pi)$
is compact.
\end{proof}

\begin{enonce}[remark]{Notation}
Let $X$ and $g$ be as in Lemma
\ref{lem-u-g}
above. The set of 
bounded pseudo-polyhedra
$\Pi$ of $(R\cup\{-\infty\})^\ell$
such that the equivalent assertions
(i) and (ii) of Lemma
\ref{lem-u-g}
hold
will be denoted by $\Theta(g)$. 
For any $\Pi\in
\Theta(g)$, 
we will denote
by $\Lambda(g,\Pi)$ the set of positive
real numbers $\epsilon$ as in (ii).
\end{enonce}

\begin{rema}\label{rem-theta-epsilon}
Let $X$ and $g$ be as in Lemma
\ref{lem-u-g}
above, and let $\Pi\in \Theta(g)$. 
The set $\Lambda(g,\Pi)$ is non-empty by definition; choose
$\epsilon$ therein. 
If $\eta$ is any real number in $(0,\epsilon)$ then it is clear
that $\Pi_\eta\in \Theta(g)$ and that 
$(\epsilon-\eta)\in \Lambda(g,\Pi_\eta)$. 
\end{rema}

%\newpage

\section{The main theorem: statement and consequences}

\begin{theo}\label{main-theorem}
Let $X$ be a smooth scheme over $C$ of pure dimension $n$.
There exists a unique morphism of sheaves of bi-graded differential $\R$-algebras on $X^{\mathrm{Zar}}$
%\vspace{-2.5mm}
\begin{align*}
\mathsf A^{\bullet, \bullet}&\longrightarrow
\mathsf B^{\bullet, \bullet}\\
 \omega &\longmapsto \omega_\flat
 \end{align*}
such that for every Zariski-open subset $U$ of $X$, every finite family $(f_1,\ldots, f_m)$ of regular functions
on $U$, 
every pair $(I,J)$ of subsets of $\{1,\ldots, m\}$
and every  $(I\cup J)$-vanishing reasonably
smooth function $\phi$ in $\mathscr S^{I,J,(f_i)}$, 
one has
%\vspace{-2mm}
\begin{multline*}
\left[\phi\left(\loga \abs{f_1},\ldots,\loga \abs{f_m}\right) 
\dl\abs{f_I}\wedge \dArg f_J\right]_\flat =\\
\phi\left(\logb \val{f_1},\ldots,\logb \val {f_m}
\right)
\di \logb\val{f_I}\wedge \dc \log \val{f_J}.
\end{multline*}
Moreover, this morphism enjoys the following properties;
let $U$ be a
Zariski-open
subset 
of
$X$ and let $\omega \in  \mathsf A^{p,q}(U)$. 
\begin{enumerate}[1]

\item Assume that the support of $\omega$ is contained in
some definably compact
semi-algebraic
subset of $U(C)$. Then $\omega_\flat$ is compactly supported. 

We assume from now on that $p=q=n$. 

\item Let $g\colon U\to \A_C^\ell$ be a morphism
and let $\Pi$ be an element of $\Theta(g)$. 
The integral
$\int_{\left(\loga \abs g\right)\inv (\Pi)}\abs\omega$
is bounded, which implies that
$\int_{\left(\loga \abs g\right)\inv (\Pi)}\omega$ is bounded too.

\item Let $(V_i)$ be a finite family
of Zariski-open subsets of $U$; for every $i$, 
let $g_i$ be a morphism
from $V_i\to \A^{\ell_i}_C$
and let $\Pi_i$ be an element of $\Theta(g_i)$. Then
\setcounter{equation}{0}
\begin{minipage}{5in}
\begin{align}
\label{theo-limomega}
\std\left(\int_{\bigcup_i \left(\loga \abs{g_i}\right)\inv(\Pi_{i,\epsilon})}\omega\right)
&\longrightarrow  \int_{\bigcup_i
\left(\logb \val{g_i}\right)\inv(\std(\Pi_i))}\omega_\flat\\
&and\notag\\
\label{theo-limabsomega}
\std\left(\int_{\bigcup_i \left(\loga \abs{g_i}\right)\inv(\Pi_{i,\epsilon})}\abs \omega\right)
&\longrightarrow \int_{\bigcup_i
\left(\logb \val{g_i}\right)\inv(\std(\Pi_i))}\val{\omega_\flat}\end{align}
\end{minipage}
when the positive \emph{standard}
number $\epsilon$
belongs to $\bigcap_i \Lambda(g_i,\Pi_i)$ 
and tends to $0$. 

Moreover there exists a positive negligible element $\alpha\in R$ such that
\begin{minipage}{5in} 
\begin{align}\label{theo-negligible}
\std\left(\int_{\bigcup_i \left(\loga \abs{g_i}\right)\inv(\Pi_{i,\epsilon})
\setminus\bigcup_i \left(\loga \abs{g_i}\right)\inv(\Pi_{i,\alpha})}\abs\omega\right)
&\longrightarrow 0\end{align}
\end{minipage} 
when the positive \emph{standard}
number $\epsilon$
belongs to $\bigcap_i \Lambda(g_i,\Pi_i)$
and tends to $0$.

\item Assume that the support of $\omega$ is contained in a definably compact
semi-algebraic
subset of $U(C)$, which implies by (1)
that $\omega_\flat$ is compactly supported. Then $\int_{U(C)}\abs \omega$
is bounded and

\begin{minipage}{5in}
\begin{align}
\label{theo-omega-exact}
\std\left(\int_{U(C)}\omega\right)
&=\int_{U\an}\omega_\flat\\
&and\notag\\
\label{theo-absomega-exact}
\std\left(\int_{U(C)}\abs \omega\right)
&=\int_{U\an}\val{\omega_\flat}.
\end{align}
\end{minipage}

\end{enumerate}
\end{theo}

\begin{rema}\label{rem-limit-moregeneral}
Statement (3c) has the following consequence. Assume that we are given for every
small enough positive standard $\epsilon$ 
in $\bigcap_i \Lambda(g_i,\Pi_i)$
a
semi-algebraic subset $D_\epsilon$ of $U(C)$
satisfying
\[\bigcup_i \left(\loga \abs{g_i}\right)\inv(\Pi_{i,\alpha})\subset
D_\epsilon\subset
\bigcup_i \left(\loga \abs{g_i}\right)\inv(\Pi_{i,\epsilon}).\]
Then
\begin{align}
\std\left(\int_{D_\epsilon}\omega\right)
&\longrightarrow \int_{\bigcup_i
\left(\logb \val{g_i}\right)\inv(\std(\Pi_i))}\omega_\flat\\
&\text{and}\notag\\
\std\left(\int_{D_\epsilon}\abs\omega\right)
&\longrightarrow \int_{\bigcup_i
\left(\logb \val{g_i}\right)\inv(\std(\Pi_i))}\val{\omega_\flat}
\end{align}
when the positive \emph{standard}
number $\epsilon$ tends to $0$. 
\end{rema}

\subsection{A statement about ordinary limits of complex integrals}\label{classical}
Our purpose is now to state a corollary of our main theorem in a more classical language, 
namely, in terms of limits of usual complex integrals, without using any ultrafilter nor any non-standard model of
$\R$ or $\C$. 

Let us recall that $\mathscr M$ denotes the field of meromorphic functions around the origin of $\C$. %FL
Let $X$ be a smooth $\mathscr M$-scheme of finite type and of pure dimension $n$, and let $(U_i)$ be a finite Zariski-open cover of $X$. For every $i$, let $(f_{ij})_{1\leq j\leq n_i}$ be a finite family of regular functions on $U_i$; for every subset $I$ and $J$ of $\{1,\ldots, n_j\}$ of cardinality $n$, let $\phi_{i,I, J}$ be a reasonably smooth and $(I\cup J)$-vanishing
complex-valued function defined on some suitable open subset of $(\R\cup\{-\infty\})^{n_i}$. 

Since $\mathscr M$ is the field of meromorphic function around the origin, $X$ gives rise to a complex analytic space, relatively algebraic, over a small enough punctured disc $D^*$, which we still denote by $X$. Up to shrinking $D^*$ we can assume that every $U_i$ is a relative Zariski-open subset of the analytic space $X$, and that the functions $f_{ij}$ are relatively algebraic holomorphic functions on $U_i$. 

Assume that there exists a relative $(n,n)$-form $\omega$ on $X$ whose support is proper over $D^*$
and such that 
\[\omega|_{U_i}=\left(\frac {-1}{\log \abs t}
\right)^n\sum_{I,J}\phi_{i,I,J}
\left(-\frac{\log \abs{f_{i1}}}{\log \abs t},\ldots,-\frac{\log
\abs{f_{in_i}}}{\log \abs t}\right)
\d\log \abs{f_{i,I}}\wedge \dArg {f_{i,J}}\]
for every $i$
(otherwise said, the forms locally defined by the above formulas coincide on overlaps, and the global form
obtained by glueing them is relatively compactly supported).

The $t$-adic completion of $\mathscr M$ is the field $\C\llp t \rrp$ of Laurent series. Fix $\tau\in (0,1)$
and endow $\C\llp t \rrp$ with the $t$-adic absolute value $\val \cdot $ that maps $t$ to $\tau$; let us denote by $X\an$ the
Berkovich analytification of $X\times_{\mathscr M}\C\llp t \rrp$. 

Then the existence of our morphism of sheaves of bi-graded
differential $\R$-algebras implies the existence of a
compactly supported $(n,n)$-form $\omega_\flat$ on
$X\an$ (in the sense of \cite{chambertloir-d2012})
such that 
\begin{multline*}
\omega_\flat|_{U_i\an}=\\
\left(\frac {-1}{\log  \tau}\right)^n\sum_{I,J}\phi_{i,I,J}
\left(-\frac{\log \val{f_{i1}}}{\log  \tau},\ldots,-\frac{\log
\val{f_{in_i}}}{\log \tau}\right) 
\di \log \val{f_{i,I}}\wedge \dc  \log \val{f_{i,J}}
\end{multline*}
for every $i$. 

Now assertion (4) has the following consequence.

\begin{theo}\label{epilogue}
We have
 \[\lim_{t \to 0}\int_{X_t}\omega|_{X_t} =\int_{X\an}\omega_\flat.\]
\end{theo}

\begin{proof}
Let $(z_n)$ be a zero-sequence of non-zero complex numbers
such that $\int_{X_{z_n}}\omega|_{X_{z_n}}$
has a limit
in $\R\cup\{-\infty,+\infty\}$
when $n$ tends to infinity, 
and let $\mathscr U$ be any ultrafilter
on $\C$ containing all cofinite subsets of
$\{z_n\}_n$. Then applying our general construction with this specific $\mathscr U$ (we recall that 
$\mathscr M$ has a natural embedding into our field $C$ of non-standard complex numbers)
we see that
\[\int_{X_{z_n}}\omega|_{X_{z_n}}\longrightarrow \int_{X\an}\omega_\flat\]
when $n$ tends to infinity. As this holds for an arbitrary
sequence $(z_n)$ as above, we are done. 
\end{proof}

\section{Proof of the main theorem}

\subsection{Compatibility with integration}\label{proof-integration-compatible}
We shall in some sense
establish the good behavior with respect to integration
\emph{before} showing the existence of the morphism $\omega\mapsto \omega_\flat$.
Let us make this more precise.

\subsubsection{Our setting}\label{setting-omega-explicit}
We assume that $\omega$ can be written 
\[\sum_{I,J}
\phi_{I,J}\left(\loga \abs{f_1},\ldots,\loga \abs{f_m}\right)
\dl\abs{f_I}\wedge \dArg{f_J}\]
where $I$ and $J$ run through the set of subsets
of $\{1,\ldots, m\}$ of cardinality $n$, 
where $(f_i)_{1\leq i\leq m}$ is a family of regular invertible functions
on
$U$, 
and where $\phi_{I,J}$ is an $(I\cup J)$-vanishing
reasonably
smooth function in
$\mathscr S^{I,J,(f_i)}$ for each $(I,J)$. 
We denote by $\omega_\flat$ the form
\[
\sum_{I,J}\phi_{I,J}\left(\logb \val{f_1}\ldots,\logb \val {f_m}\right)
\di \logb\val{f_I}\wedge \dc \log \val{f_J}\]
(we insist that our morphism has not yet been defined, so $\omega_\flat$
is
currently 
just a notation for the form above).

We also assume that the open covering $(V_i)$ is the trivial covering consisting of one open subset $V_1=U$
and we write $g$ instead of $g_1$,
$\Pi$ instead of $\Pi_1$
and $\ell$ instead of $\ell_1$.

The whole subsection \ref{proof-integration-compatible}
will be devoted to the proof of (2) and (3) in this setting.

\subsubsection{Proof of
\textnormal{(2)}}\label{proof-bounded}
We shall in fact prove that $\int_K \abs \omega$ is bounded for any
$t$-bounded
definably compact
semi-algebraic  subset $K$
of $U(C)$; so, let us fix such a subset $K$. 
Since
$K$ is definably compact and since
$\loga \abs {f_i}$ only takes bounded values
on the invertible locus of $f_i$,
there exists a positive standard real number
$A$ such that $\loga \abs {f_i}\leq A$ on $K$ for all $i$; thus there exists a positive standard real number
$N$ such that 
$\abs{\phi_{I,J}(\loga \abs {f_1},\ldots, \loga \abs {f_m})}\leq N$ on $K$
for all $(I,J)$.

Fix $I$ and $J$. By the very definition of {$(I\cup J)$-vanishing
reasonably smooth functions,
there exist two open subsets
$V_{I,J}'\subset V_{I,J}$ of $(\R\cup\{-\infty\})^m$, defined
by $\Q$-linear inequalities, and such that the following holds:
\begin{itemize}[label=$\diamond$]
\item $\phi_{I,J}$ is defined on $V_{I,J}$ and $(\loga \abs{f_1},\ldots,
\loga \abs{f_m})(U(C))\subset V_{I,J}(R)$; 
\item $\phi_{I,J}|_{V'_{I,J}}=0$, and for every $i\in I\cup J$, 
the $i$-th coordinate function does not take the value $-\infty$
on $V_{I,J}\setminus V'_{I,J}$.
\end{itemize}
Let $K_{I,J}$ be the pre-image of $V_{I,J}
\setminus V'_{I,J}$ in $K$ under
$(\loga \abs{f_1},\ldots,
\loga \abs{f_m})$. This is a definably compact semi-algebraic
subset of $K$ on which $\abs{f_i}$ does not vanish as soon as
$i\in I\cup J$; by construction, 
$\phi_{I,J}\left(\loga \abs{f_1},\ldots,\loga \abs{f_m}\right)$
vanishes on $K\setminus K_{I,J}$. 

By enlarging $A$, we may assume that for all $I,J$ and all $i\in I\cup J$ one has the minoration $\log \abs{f_i}\geq -A$ on $K_{I,J}$.

For every subset $L$ of $\{1,\ldots, m\}$, denote by $\mathrm D_L$
the subset of $U(C)$  consisting of points at which
every $f_i$ with $i\in L$ is invertible.
Let $i\in \{1,\ldots, m\}$; on
$\mathrm D_{\{i\}}$ we set $f_i=r_ie^{2i\pi\alpha_i}$ for every $i$
(where $r_i=\abs{f_i}$ and $\alpha_i$ is a multi-valued function,
which we will use
only through the well-defined
differential form $\d \alpha_i$).
Let $I$ and $J$ be two subsets of $\{1,\ldots, m\}$
of cardinality $n$. 
Let $i_1<\ldots<i_n$ be the elements of $I$, and $j_1<\ldots<j_n$ be those of $J$; on $\mathrm D_{I\cup J}$, we
set
$\frac{\d r_I}{r_I}=\frac{\d r_{i_1}}{r_{i_1}}\wedge\ldots\wedge
\frac{\d r_{i_n}}{r_{i_n}}$
and
$\d \alpha_J=\d\alpha_{j_1}\wedge\ldots\wedge \d\alpha_{j_n}.$
Let $S^1_R$ denote the ``unit circle" $\{z\in C,\abs z=1\}$. Let $u_{I,J}$ be the map from $\mathrm D_{I,J}$ to
$(R_{>0})^n\times (S^1_R)^n$ that maps a point $x$ to
$\left(\abs{f_{i_1}(x)}, \ldots, \abs {f_{i_n}(x)}, \frac {f_{j_1}(x)}{\abs{f_{j_1}(x)}}, \ldots, \frac{f_{j_n}(x)}
{\abs{f_{j_n}(x)}}\right).$

We denote by $\rho_j$
the coordinate function on $(R_{>0})^n\times(S^1_R)^n$ 
corresponding to the $j$-th factor $R_{>0}$, and by $\varpi_j$
the multi-valued argument function corresponding to the $j$-th factor $S^1_R$. The form
$\d \varpi_j$
is well-defined (we can describe it alternatively as the pull-back under the projection to the $j$-th
factor $S^1_R\simeq \{(x,y)\in R^2,x^2+y^2=1\}$
of the form
$x\d y-y\d x$). Let $E_{I,J}$ denote the étale locus of $u_{I,J}$; by definability and o-minimality, there exists an integer $d$ such that
the fibers of $u_{I,J}|_{E_{I,J}\cap K}$ are all of cardinality $\leq d$ for all $I$ and $J$.

We then have (we recall that $\lambda=-\log\abs t$)
\begin{align}
\int_K\abs \omega&\leq\frac N{\lambda^n}\sum_{I,J}\int_{K_{I,J}}\left|\frac{\d r_I}{r_I}\wedge \d\alpha_J\right|\\
&=\frac N{\lambda^n}\sum_{I,J}\int_{K_{I,J}\cap E_{I,J}}\left|\frac{\d r_I}{r_I}\wedge\d\alpha_J\right|\\
\label{ineq-continuity}
&\leq\frac{Nd}{\lambda^n}\sum_{I,J}\int_{u_{I,J}(K_{I,J}\cap E_{I,J})}\left|\frac{\d \rho_1}{\rho_1}\wedge\ldots
\wedge \frac{\d \rho_n}{\rho_n}\wedge\frac{d \varpi_1}{2\pi}\wedge \ldots \wedge \frac{\d \varpi_n}{2\pi}\right|\\
\label{equation-reference}
&\leq\frac{Nd}{\lambda^n}\sum_{I,J}\int_{\abs{f_I}(
K_{I,J}\cap E_{I,J})}\frac{\d \rho_1}{\rho_1}\wedge\ldots
\wedge \frac{\d \rho_n}{\rho_n}\\
&\leq{m\choose n}^2\frac{Nd}{\lambda^n}\int_{\left[\abs t^A,\abs t^{-A}\right]^n}\frac{\d \rho_1}{\rho_1}\wedge\ldots
\wedge \frac{\d \rho_n}{\rho_n}\\
&\leq{m\choose n}^2\frac{Nd}{\lambda^n}
\left(-2A\log \abs t\right)^n\\
&={m\choose n}^2Nd(2A)^n.
\end{align}

Hence $\int_K\abs \omega$ is bounded, as announced.

\subsubsection{Proof of \textnormal{(3) (\ref{theo-negligible})}}
The proof of (a) and (b) will rest on several steps allowing ourselves to reduce to a simpler case, 
in which it will be possible to perform some explicit computations that are the core of our proof. But for achieving this reduction we shall need (c), hence we
start by proving it.

Let $\Pi\in \Theta(g)$. Choose a positive standard real %FL
number $a$
in $\Lambda(g,\Pi)$ (such an $a$ exists in view of Remark
\ref{rem-theta-epsilon}).
For every non-negative standard real number
$\epsilon$ we set $P_\epsilon=\std(\Pi)+[-\epsilon, \epsilon]^\ell\subset \R^\ell$
(so $P_0=\std(\Pi)$).
Let us introduce some notation: 

\begin{itemize}[label=$\diamond$]
\item  $V_\epsilon=\left(\logb \val g\right)\inv(P_\epsilon)\subset U\an$, 
for $\epsilon$ a standard element of $[0,a]$; 
\item $V_{\epsilon, \eta}=\left(\logb \val g\right)\inv(\overline{P_\epsilon\setminus P_\eta})
\subset U\an$, 
for $\epsilon$ a standard element of $[0,a]$ and $\eta$ a standard element of
 $(0,\epsilon)$; 
 \item $K_\epsilon=\left(\loga \abs g\right)\inv(\Pi_\epsilon)\subset U(C)$, 
for $\epsilon$ any element of $R$ lying on $[0,a]$; 
\item $K_{\epsilon, \eta}=\left(\loga \abs g\right)\inv(\overline{\Pi_\epsilon\setminus \Pi_\eta})
\subset U(C)$, 
for $\epsilon$ any element
of $R$ lying on $[0,a]$ and $\eta$
any element of $R$ lying on 
 $(0,\epsilon)$. 
\end{itemize}

We fix
two subsets $I$ and $J$ of $\{1,\ldots, m\}$ of cardinality $n$. For every standard real number $A$
we shall need the following extra notation: 
\begin{itemize}[label=$\diamond$]
\item $V_\epsilon^A$ (\resp $V_{\epsilon,\eta}^A$)
for the intersection of $V_\epsilon$ (\resp $V_{\epsilon,\eta}$)
with the closed analytic domain of $U\an$
defined by the inequalities $\logb \val{f_i}\geq A$ for all $i\in I$; 

\item $K_\epsilon^A$ (\resp $K_{\epsilon,\eta}^A$)
for the intersection of $K_\epsilon$ (\resp $K_{\epsilon,\eta}$)
with the closed semi-algebraic subset of $U(C)$
defined by the inequalities $\loga \abs{f_i}\geq A$ for all $i\in I$; 
\end{itemize}

The pre-image of $V_{I,J}\setminus V'_{I,J}$ 
(the notation is introduced
in the
second paragraph of \ref{proof-bounded}) in $K_a$ under
$(\loga \abs{f_i})_{1\leq i\leq m}$ is definably compact, and
none of the functions $f_i$ with $i\in I$ vanishes on it;  thus there exists some standard real number $A$ such that
every point of $K_a$ at which at least one
of
the $\loga \abs{f_i}$ is smaller than $A$ belongs to the pre-image of $V'_{I,J}$,
so $\phi_{I,J}(\loga \abs f)$ vanishes at such a point. 
Using mutatis mutandis the same argument 
and up to decreasing $A$ if necessary, we can ensure that
 $\phi_{I,J}(\logb \val f)$ vanishes at every point of $V_a$ at which at least one the $\logb \val{f_i}$ is smaller than $A$.

Otherwise said, 
there exists a
standard real number $A$ such that 
for
every element $\epsilon$ of $R$ lying
on $[0,a]$, the function 
$\phi_{I,J}(\loga \abs f)$ vanishes on
$K_\epsilon\setminus K_\epsilon^A$ and 
the function
$\phi_{I,J}(\logb \val f)$ vanishes on $V_\epsilon\setminus V_\epsilon^A$.

We are now going to show
that 
$\mathrm{Vol}(\logb \val{f_I}(V_\epsilon^A\setminus V_0^A))$
tends to zero when $\epsilon$ tends to zero, which is the core
of the proof of (3) (\ref{theo-negligible}). Our method for proving this claim
consists in describing $\logb \val{f_I}(V^A_\epsilon)$ more or less as the image under $\logb \val{f_I}$ 
of a \emph{piecewise-linear subset} of $V^A_a$, which allows us
to get rid of non-archimedean geometry and only deal with usual real integration.

Recall that the skeleton of $\mathbf G_{\mathrm m}^{I,\mathrm{an}}$
is the closed subspace of
$\mathbf G_{\mathrm m}^{I,\mathrm{an}}$
homeomorphic to $\R^I$
via the mapping
$\mathrm{sk} : \R^I \to \mathbf G_{\mathrm m}^{I,\mathrm{an}}$
sending $(\log (r_i))_{i \in I}$
to the seminorm
assigning the real number $\max_{m \in \Z^I} \abs {a_m} \prod_{i \in I} r_i^{m_i}$
to a Laurent polynomial
$\sum_{m \in \Z^I} a_m T^m$.
Let $\Sigma$ be the pre-image 
of the skeleton of $\mathbf G_{\mathrm m}^{I,\mathrm{an}}$
under $f_I|_{V_a^A}
$. This is a 
skeleton of $V_a^A$ in the sense of
\cite{ducros2012b}, 4.6 (see \loccit, Thm. 5.1; note that some mistake in this paper is
corrected in \cite{ducros2013c}); in particular it inherits
a canonical piecewise-linear structure 
and $(\logb \val {f_I})|_\Sigma$ is piecewise-linear.
Moreover if $W$ is any compact analytic domain of $V_a^A$, the intersection
$\Sigma\cap W$
is a piecewise-linear subset of $\Sigma$
and
$\logb \val{f_I}(W)^{(n)}=\logb \val{f_I}((\Sigma\cap W)^{(n)})$,
where the
superscript $^{(n)}$ denotes the pure $n$-dimensional part
of a piecewise-linear set (this last equality
is a lemma which is shown in
a forthcoming version of \cite{chambertloir-d2012}}; its proof is not difficult and rests on the description of a skeleton in terms of tropical dimension, see \cite{chambertloir-d2012}, 2.3.3); in particular,
the volume of  $\logb \val{f_I}(W)$ is equal to that
of $\logb \val{f_I}(W\cap \Sigma)$.

Choose $\epsilon \in (0,a]$. 
From the equality
$V_\epsilon^A\setminus V_0^A=\bigcup_{0<\eta<\epsilon}
V_{\epsilon,\eta}^A$, 
we get
\begin{align*}
\mathrm{Vol}(\logb \val{f_I}((V_\epsilon^A\setminus V_0^A))
&=\sup_{0<\eta<\epsilon}\mathrm{Vol}(\logb \val{f_I}(V_{\epsilon,\eta}^A))\\
&=\sup_{0<\eta<\epsilon}\mathrm{Vol}(\logb \val{f_I}(\Sigma\cap V_{\epsilon,\eta}^A)
)\\
&=\mathrm{Vol}(\logb \val{f_I}((\Sigma\cap V_\epsilon^A)\setminus (\Sigma\cap V_0^A)).
\end{align*}
Now $(\Sigma\cap V_\epsilon^A)_{0<\epsilon\leq a}$ is a non-increasing family of compact piecewise linear subsets of
$\Sigma$ with intersection $\Sigma\cap V_0^A$, and $\logb \val {f_I}|_\Sigma$ is piecewise linear.
Since $\dim \Sigma\leq n$,  this implies that
that $\mathrm{Vol}(\logb \val{f_I}(\Sigma\cap V_\epsilon^A\setminus \Sigma\cap V_0^A))$
tends to zero when $\epsilon$ tends to zero. 
By the above, this means that $\mathrm{Vol}(\logb \val{f_I}(V_\epsilon^A\setminus V_0^A))$
tends to zero, as announced.

In order to end the proof of (3) (\ref{theo-negligible}) we now have to understand
the consequences  in the non-standard world
of the limit statement above (which involves only standard objects); 
this step rests in a crucial way on {\sc doag}-definability.

For every standard $\epsilon\in (0,a]$ the set $
\Lambda_\epsilon:=\logb \val{f_I}(V_\epsilon^A\setminus V_0^A)$ is {\sc doag}-definable,
and depends {\sc doag}-definably on $\epsilon$.
Thus $\Lambda_{\epsilon,R}$
makes sense for every element $\epsilon\in R$ with $0<\epsilon \leq a$.

Let $D$ be the set of positive elements $x\in R$ such that 
$x<a/2$ and
\[\loga \abs{f_I}\left(K_\epsilon^A\setminus K^A_x\right)
\subset
\Lambda_{2\epsilon,R}\]
for all $\epsilon\in (x,\frac a2)$.
An element $x$ of $R$ belongs to $D$ 
if and only if the implication 
\[(\abs {g(z)}\in \abs t^{-(\Pi_\epsilon \setminus \Pi_x)}
\;\;\text{and}\;\;
\abs{f_I(z)}\in\abs t^{[A,+\infty)^I})
\Rightarrow
\abs{f_I(z)}
\in \abs t^{-\Lambda_{2\epsilon, R}}
\]
holds for all $z\in U(C)$. 
It thus  follows from
\ref{sss-tp-definable}
that $\abs t^D$ is definable; but since it is one-dimensional, it is a finite union of intervals by o-minimality, so $D$
is also such a union, hence is definable as well. 
Moreover, it contains by definition every
bounded $x$
whose standard part belongs to $(0,\frac a2]$. 
As a consequence, $D$ contains $[\alpha,\frac a2]$ for some positive negligible element $\alpha$. 

For all elements $\epsilon$
of $R$ lying on $(\alpha,a/2)$ we have
\[\loga \abs{f_I}(K_\epsilon^A\setminus K_\alpha^A)
\subset
\Lambda_{2\epsilon,R}.\]
The inclusion above holds in particular for every positive
standard $\epsilon<a/2$; for such an $\epsilon$
we thus have 
\[\frac 1 {\lambda^n}\int_{\abs{f_I}(K_\epsilon^A\setminus K_\alpha^A)}\frac{\d \rho_1}{\rho_1}\wedge\ldots
\wedge \frac{\d \rho_n}{\rho_n}\leq \mathrm{Vol}(\Lambda_{2\epsilon}).\]
Since $ \mathrm{Vol}(\Lambda_{2\epsilon})\longrightarrow 0$ when $\epsilon\longrightarrow 0$, 
it follows that 
\[\std\left(\frac 1{\lambda^n}\int_{\abs{f_I}(K_\epsilon^A\setminus K_\alpha^A)}
\frac{\d \rho_1}{\rho_1}\wedge\ldots
\wedge \frac{\d \rho_n}{\rho_n}\right)\longrightarrow 0\] when $\epsilon\longrightarrow 0$. 
In view of inequality (\ref{equation-reference}) of paragraph \ref{proof-bounded}, 
this implies that
\[\std\left(\int_{K_\epsilon^A\setminus K_\alpha^A}\left|
\phi_{I,J}(\loga \abs{f_1},\ldots, \loga \abs{f_m})\dl\abs{f_I}\wedge \dArg{f_J}\right|\right)
\longrightarrow 0\]
when $\epsilon\longrightarrow 0$.
But by the choice of $A$
the integral \[\int_{K_\epsilon^A\setminus K_\alpha^A}\left|\phi_{I,J}(\loga \abs{f_1},\ldots, \loga \abs{f_m})\dl\abs{f_I}\wedge \dArg{f_J}\right|\]
is equal to 
\[\int_{K_\epsilon\setminus K_\alpha}\left|\phi_{I,J}(\loga \abs{f_1},\ldots, \loga \abs{f_m})\dl\abs{f_I}\wedge \dArg{f_J}
\right|,\]
so that 
\[\std\left(\int_{K_\epsilon\setminus K_\alpha}\abs
{\phi_{I,J}(\loga \abs{f_1},\ldots, \loga \abs{f_m})\dl\abs{f_I}\wedge \dArg{f_J}}\right)
\longrightarrow 0\]
when $\epsilon\longrightarrow 0$.

The infinitesimal element $\alpha$ above depends a priori on $(I,J)$; but by taking it large enough
(and still infinitesimal) we can ensure that it does not. Then 
\[\std\left(\int_{K_\epsilon\setminus K_\alpha}\abs
\omega\right)
\longrightarrow 0\]
when $\epsilon\longrightarrow 0$, which ends the proof of 
(3) (\ref{theo-negligible})
in our particular setting.

\subsubsection{Proof of \textnormal {(3) (\ref{theo-limomega})}
and \textnormal{(3)
(\ref{theo-limabsomega})}
in our setting}
Assertions \textnormal {(3) (\ref{theo-limomega})}
and \textnormal{(3)
(\ref{theo-limabsomega})}
involve the form to be integrated $\omega$, which is defined with an explicit formula using the functions $f_i$, 
and the domain of integration, whose definition uses another family of functions $g$ and a pseudo-polyhedron $\Pi$. 
We will first simplify slightly this set of data, by showing that we may assume that $f=g$ and $\Pi$ is of the form $P_R$ for some
pseudo-polyhedron $P\subset (\R\cup\{-\infty\})^\ell$ (and so $\std (\Pi)=P$),
with moreover 
$\logb \val f((\logb \val f)\inv(P))=P$. This reduction essentially uses
(3) (\ref{theo-negligible}) through its consequence Remark  \ref{rem-limit-moregeneral}, 
together with some elementary definability arguments.

Set $h=(f,g)$, $P=\std(\Pi)$, $W=\left(\logb \val g\right)\inv(P)\subset V\an$, and 
$Q=\logb \val h(W)\subset \R^{m+\ell}$.
Then $W=(\logb \val h)\inv(Q)$. We are now going to explain why it is sufficient to prove assertion (3) for
$(Q_R,h)$ instead of $(\Pi, g)$. So we assume (3) (a) and (b) hold for $(Q_R,h)$.

If $\epsilon$ is a positive real number we clearly have $\left(\logb \val h\right)\inv(Q_\epsilon)
\subset\left(\logb \val g\right) \inv(P_\epsilon)$. 
On the other hand for every $\epsilon>0$ the set $\left(\logb \val h\right)\inv(Q_\epsilon)$
is a neighborhood of $W$, hence contains
$\left(\logb \val g\right)\inv(P_\eta)$ for some
$\eta$ which can be taken in $(0,\epsilon]$
(here we use topological properness -- recall that $\Pi\in \Theta(g)$).
Let $\delta(\epsilon)$ denote the
least upper bound of
\[\left\{\eta\in (0,\epsilon), \left(\logb \val g\right)\inv(P_\eta)\subset
\left(\logb \val h\right)\inv(Q_\epsilon)\right\}\;;\]note that
by compactness we have $\left(\logb \val h\right)\inv(P_{\delta(\epsilon)})
\subset \left(\logb \val h\right)\inv(Q_\epsilon)$. 
Then $\delta$ is a {\sc doag}-definable function; in view of the fact that $\delta(\epsilon)\leq\epsilon$
by definition, this implies that there exists a positive rational number $r$ and a positive real number $M$
such that $\delta(\epsilon)=M\epsilon^r$ for $\epsilon$ small enough. 

This implies that 
\[\left(\loga \abs h\right)\inv(Q_{R,\frac \epsilon2})
\subset \left(\loga \abs g\right)\inv(\Pi_{\epsilon})
\subset 
\left(\loga \abs h\right)\inv(Q_{R,\frac 2M\epsilon^{1/r}})\]
for $\epsilon$ a small enough standard positive real number. 
Since we assume that (3) (a) and (3) (b) hold for $(Q_R,h)$ (and since (3) (c) has already been proved)
it follows from Remark \ref{rem-limit-moregeneral} 
that 
\begin{align*}
\std\left(\int_{\left(\loga \abs g\right)\inv(\Pi_\epsilon)}
\omega\right)
&\longrightarrow
\int_{\left(\logb \val g\right)\inv(P)}\omega_\flat\\
&\text{and}\\
\std\left(\int_{\left(\loga \abs g\right)\inv(\Pi_\epsilon)}
\abs{\omega}\right)
&\longrightarrow
\int_{\left(\logb \val g\right)\inv(P)}\val{\omega_\flat}
\end{align*}
when the positive \emph{standard}
number $\epsilon$
belongs to $\Lambda(g,\Pi)$ and tends to $0$. 

Therefore if the result holds for $(Q_R,h)$ it holds for $(\Pi,g)$; we thus can replace $\Pi$ by $Q_R$ and $g$ by $h$, 
and then enlarge $f$ (which is harmless) so that $g=f$.  We keep the notation $P=\std(\Pi)$ and
\[W=\left(\logb \val g\right)\inv(P)
=\left(\logb \val f\right)\inv(P)\;\]
note that we have $\Pi=P_R$ and $\left(\logb \val f\right)(W)=P$.
%(which amounts to saying that $(\log \val g)(W)=\lambda_\flat P$). 
%The form $\omega_\flat$ is equal to $f^*\xi$ 
%where $\xi$ is the Lagerberg form 
%\[\frac 1{\lambda_\flat^n}\sum_{I,J}\phi_{I,J}\left(\frac {x_1}{\lambda_\flat}, \ldots, \frac{x_m}{\lambda_\flat}\right)
%\di x_I\wedge \dc x_J.\] 

\subsubsection{Arguing piecewise on $P$}\label{sss-arg-piece}
To allow for more flexibility in the proof, we shall
need to argue piecewise on $P$. We explain here why it is possible;
the key point are once again (3) (\ref{theo-negligible}), and the additivity of integrals in 
both frames.

Assume that we are given a finite covering $(P_i)_{i\in I}$ of
$P$
by pseudo-polyhedra,
and that for every non-empty
subset $J$ of $I$, statements (3) (a) and (3) (b) hold for 
$(P_J,f)$ with $P_J:=\bigcap_{i\in J}P_i$. Then these statements hold for $(P,f)$. 

Indeed, for every $i$  set $\Pi_i=P_{i,R}$, and every non-empty subset $J$ of $I$, 
set $\Pi_J=P_{J,R}$.
For every positive standard $\epsilon$
we have $\Pi_\epsilon=\bigcup_i \Pi_{i,\epsilon}$. Now let $J$ be a non-empty subset of $I$. 

If $P_J=\emptyset$
then for $\epsilon$ small enough we have $\bigcap_{i\in J}\Pi_{i,\epsilon}=\emptyset$. 
If $P_J\neq \emptyset$ then 
by definability and compactness there exists two positive real number $A$ and $\eta$ such that 
\[P_{J,\epsilon}\subset \bigcap_{i\in J}P_{i,\epsilon}\subset P_{J, A\epsilon}\]
for all positive real number $\epsilon<\eta$ which implies (by model-completeness
of {\sc doag}) that
\[\Pi_{J,\epsilon}\subset \bigcap_{i\in J}\Pi_{i,\epsilon}\subset \Pi_{J, A\epsilon}\]
for every positive $\epsilon<\eta$ in $R$

The difference
\[\int_{\left(\loga \abs f\right)\inv(\Pi_\epsilon)}\omega
-\sum_{J
\neq \emptyset}(-1)^{\abs J+1}\int_{\left(\loga \abs f\right)\inv(\Pi_{J,\epsilon})}\omega\]
can be rewritten 
\[\sum_{J
\neq \emptyset}(-1)^{\abs J+1}\left(\int_{\bigcap_{i\in J}\left(\loga \abs f\right)\inv(\Pi_{i,\epsilon})}
\omega-\int_{\left(\loga \abs f\right)\inv(\Pi_{J,\epsilon})}\omega\right).\]
It now follows from (3) (c) (which has already be proven) and from the inclusions 
$\Pi_{J,\epsilon}\subset \bigcap_{i\in J}\Pi_{i,\epsilon}\subset \Pi_{J, A\epsilon}$ (which hold for 
$\epsilon$ small enough) that 
\[\std\left(\int_{\bigcap_{i\in J}\left(\loga \abs f\right)\inv(\Pi_{i,\epsilon})}
\omega-\int_{\left(\loga \abs f\right)\inv(\Pi_{J,\epsilon})}\omega\right)
\longrightarrow 0\]
when $\epsilon \longrightarrow 0$ (and remains standard). 
Thus
\[\std\left(\int_{\left(\loga \abs f\right)\inv(\Pi_\epsilon)}\omega
-\sum_{J
\neq \emptyset}(-1)^{\abs J+1}\int_{\left(\loga \abs f\right)\inv(\Pi_{J,\epsilon})}\omega\right)
\longrightarrow 0\]
when $\epsilon\longrightarrow 0$
As statements (3) (a) and (3) (b) hold for every $P_J$,
this implies that 
\[\std\left(\int_{\left(\loga \abs f\right)\inv(\Pi_\epsilon)}\omega\right)\longrightarrow \sum_J
(-1)^{\abs J+1}
\int_{\left(\logb \val f\right)\inv(P_J)}
\omega_\flat=\int_{\left(\logb \val f\right)\inv(P)}
\omega_\flat\]
when $\epsilon\longrightarrow 0$.

We prove in the same way that 
\[\std\left(\int_{\left(\loga \abs f\right)\inv(\Pi_\epsilon)}\abs \omega\right)\longrightarrow
\int_{\left(\logb \val f\right)\inv(P)}
\val{\omega_\flat}\]
when $\epsilon\longrightarrow 0$.

\subsubsection{}\label{sss-good-piece}
Being allowed to argue piecewise on $P$, we now would like
to cut it into finitely many pieces
as nice as possible. This will be achieved by exhibiting a finite covering $(P_i)$
of $P$ by pseudo-polyhedra such that for every $i$ the following hold: 

\begin{itemize}[label=$\diamond$]
\item for  every pair
$(I,J)$ of subsets of $\{1,\ldots, m\}$ of cardinality $n$, either 
$\phi_{I,J}$ is identically zero on $P_i$, either for every $(x_1,\ldots, x_m)\in P_i$ and every 
$j\in I\cup J$ we have $x_j\neq -\infty$; 
\item there exists a subset $E$ of $\{1,\ldots, m\}$ such that:
\begin{itemize}[label=$\bullet$]
\item for every $(x_1,\ldots, x_m)\in P_i$ and every 
$j\in E$ we have $x_j\neq -\infty$; 
\item for every pair $(I,J)$ of subsets of $\{1,\ldots, m\}$ of cardinality $n$, 
there exists a compactly supported smooth function 
$\psi_{I,J}$ on $\R^E$ such that  for every $(x_1,\ldots, x_m)\in P_i$ 
one has 
$\phi_{I,J}(x_1,\ldots, x_m)=\psi_{I,J}(x_j)_{j\in E}$. 
\end{itemize}
\end{itemize}

Let us explain how this can be done. 
Let $\xi$ be a point of $P$
and let $I$ and $J$ be two subsets of $\{1,\ldots, m\}$ 
of cardinality $n$. By the very definition of $(I\cup J)$-vanishing
reasonably smooth functions, 
there exists a pseudo-polyhedral neighborhhod $Q$ of $x$ in $P$
such that 
\begin{itemize}[label=$\diamond$]
\item either 
$\phi_{I,J}$ is identically zero on $Q$, either for every $(x_1,\ldots, x_m)\in Q$ and every 
$j\in I\cup J$ we have $x_j\neq -\infty$; 
\item there exists a subset $E$ of $\{1,\ldots, m\}$ such that:
\begin{itemize}[label=$\bullet$]
\item we have $x_j\neq -\infty$ for every $(x_1,\ldots, x_m)\in Q$ and every 
$j\in E$ ; 
\item 
there exists a compactly supported smooth function 
$\psi$ on $\R^E$ such that  for every $(x_1,\ldots, x_m)\in Q$ 
one has 
$\phi_{I,J}(x_1,\ldots, x_m)=\psi(x_j)_{j\in E}$
\end{itemize}
\end{itemize}
(note that \emph{a priori} $\psi$ is a smooth function defined on an open
neighborhood of the projection of $Q$
to $\R^E$, but since the latter is compact
we can assume that $\psi$ is defined on the whole of $\R^E$ and compactly supported). 
We now conclude by compactness of $P$. 

\subsubsection{}
In view of \ref{sss-arg-piece}
and of \ref{sss-good-piece}, we can assume that there exists a subset $E$ of $\{1,\ldots, m\}$ 
satisfying the following: 
\begin{itemize}[label=$\diamond$]
\item for all $(x_1,\ldots, x_m)$ in $P$ and all $j\in E$, we have
$x_j\neq -\infty$; 
\item one can in fact write 
\begin{align*}
\omega&=\sum_{I,J}
\phi_{I,J}\left(\loga \abs{f_j}\right)_{j\in E}\
\dl\abs{f_I}\wedge \dArg f_J\\
&\textnormal{and}\\
\omega_\flat&=\sum_{I,J}
\phi_{I,J}\left(\logb \val{f_j}\right)_{j\in E}\
\di \logb\val{f_I}\wedge \dc \log \val{f_J}
\end{align*}
where $I$ and $J$
run through the set of subsets of $E$
of cardinality $n$, and where the $\phi_{I,J}$ are smooth, compactly supported
functions on $\R^E$. 
\end{itemize}
We note that the functions $f_j$ with $j\in E$ are invertible on the analytic
domain $W$; we set 
$Q=(\logb \val{f_E})(W)$; this is a compact polyhedron of $\R^E$ which can also be described as the
image of $P$ under the projection to $(\R\cup\{-\infty\})^E$.

We denote by $\xi$ the Lagerberg form 
\[\sum_{I,J}
\left(\frac 1{\lambda_\flat}\right)^n\phi_{I,J}(x_j/\lambda_\flat)_{j\in E}\
\di x_E\wedge \dc x_E\] on $\lambda_\flat Q$; by construction, $\omega_\flat=f_E^*\xi$.

\subsubsection{}\label{case-dimp-small}
We first consider the case where $\dim Q<n$. In this case the $(n,n)$-form 
$\xi$ on $\lambda_\flat Q$ is zero, and 
it suffices to prove that 
\[\std\left(\int_{\left(\loga \abs g\right)\inv(\Pi_\epsilon)}\abs
\omega\right)
\longrightarrow 0\]
when $\epsilon\longrightarrow 0$. 
This will follow quite easily from the rough estimates of \ref{proof-bounded}.

Fix $I$ be any subset of $E$ of cardinality $n$. For every positive standard real number $\epsilon$,
let $Q^I_\epsilon$ denote the image of $Q_\epsilon$ under the projection map $\R^E\to \R^I$. The inequality
$\dim Q<n$ implies that $\mathrm{Vol}(Q^I_\epsilon)\longrightarrow 0$ when $\epsilon\longrightarrow 0$. 

Now for every standard positive $\epsilon$ we have the inclusion 
\[\left(\loga \abs {f_I}\right)\left(\left(\loga \abs f\right)\inv(\Pi_\epsilon)\right)
\subset Q^I_{2\epsilon, R}.\]
It follows that 
\[\frac 1{\lambda^n}\int_{\abs{f_I}\left(\left(\loga \abs f\right)\inv(\Pi_\epsilon)\right)}
\frac{\d \rho_1}{\rho_1}\wedge\ldots
\wedge \frac{\d \rho_n}{\rho_n}\leq \mathrm{Vol}(Q^I_{2\epsilon}).\]

Since this holds for all $I$, this implies in view of inequality  (\ref{equation-reference}) of paragraph \ref{proof-bounded}
that \[\std\left(\int_{\left(\loga \abs g\right)\inv(\Pi_\epsilon)}\abs
\omega\right)
\longrightarrow 0\]
when $\epsilon\longrightarrow 0$.

\subsubsection{}We are now going to describe two general methods which we shall use several times
to make the situation simpler.
The first one essentially combines the fact that the statements we want to prove can 
be checked piecewise on $P$ (\ref{sss-arg-piece})
and the fact that they hold as soon as $\dim Q<n$ 
(\ref{case-dimp-small}); the second one follows easily from Remark 
\ref{rem-limit-moregeneral}.

\paragraph{Arguing cellwise
on $Q$}\label{arguing-cellwise}
Let $(Q_i)$ be a finite covering of $Q$ by compact
polyhedra whose pairwise intersections are of dimension 
$<n$ ; for every $i$, let $P_i$ be the pre-image of $Q_i$
in $P$. Assume that statements (3) (a) and (3) (b) hold for every $P_i$;
then they hold for $P$. Indeed, let $I$ be any finite set of indices of cardinality at least 2. Then
the projection of $\bigcap_{i\in I}P_i$ to $(\R\cup\{-\infty\})^E$ is equal to 
 $\bigcap_{i\in I}Q_i$, so it is of dimension $<n$. Therefore the theorem holds
 for $\bigcap_{i\in I}P_i$ in view of \ref{case-dimp-small}; it now follows from
 \ref{sss-arg-piece}
 that it holds for $P$.

\paragraph{Affine change of coordinates}\label{affine-change}
Let $M=(m_{ij})$ be a matrix belonging to $\mathrm M_E(\Z)$
with non-zero determinant, and let $v=(v_j)_j\in \R^E$. 
For every point
$x=(x_1,\ldots, x_m)$ in $P$ we set 
$Mx=(y_1,\ldots, y_m)$ with $y_i=x_i$ if $i\notin E$, 
and $y_i=\sum_{j\in I}m_{ij}x_j$ otherwise. 
For $i\notin E$ we set $h_i=f_i$; for $i\in E$ we
set $h_i=\abs t ^{v_i}\prod_{j\in I}f_j^{m_{ij}}$. 

Set 
$P'=MP+v$; this is a pseudo-polyhedron.
By expressing $\loga \abs h$, $\dl \abs h$ and $\da h$ in terms of
$\loga \abs f$, $\dl \abs f$ and $\da f$, and the same with
$\logb$ instead of $\loga$ and $\val \cdot$ instead of $\abs \cdot$,
we get equalities
\[\omega=\sum_{I,J}\psi_{I,J}
\left(\loga \abs {h_1},\ldots, \loga \abs {h_m}\right)\dl \abs{h_I}\wedge \dArg h_J\]
and
\[\omega_\flat=\sum_{I,J}\psi_{I,J}\left(\logb\val {h_1},\ldots, \logb \val {h_m} \right)\di \logb \val{h_I}\wedge \dc
\log \val{h_J}.\]
Assume that statements (3)(a) and (3)(b) hold for
$(P'_R,h)$. We are going to prove that they hold for $(\Pi,f)$. 

There exist two standard positive real numbers $A$ and $B$ with $A<B$ 
such that
\[\left(\loga \abs h\right)\inv(P'_{R,A\epsilon})\subset
\left(\loga \abs f\right)\inv(\Pi_\epsilon)
\subset \left(\loga \abs h\right)\inv(P'_{R, B\epsilon})\]
for $\epsilon$ small enough. Then
\[\std\left(\int_{\left( \loga \abs f\right)\inv(\Pi_\epsilon)}\omega\right)\longrightarrow
\int_{\left(
\logb \val h\right)\inv(P')}
\omega_\flat= \int_{\left(
\logb \val f\right)\inv(P)}
\omega_\flat\]
and
\[\std\left(\int_{\left(\loga \abs f\right)\inv(\Pi_\epsilon)}\abs \omega\right)\longrightarrow\
\int_{\left(
\logb \val h\right)\inv(P')}
\val{\omega_\flat}=\int_{\left(
\logb \val f\right)\inv(P)}
\val{\omega_\flat}\]
by Remark \ref{rem-limit-moregeneral}. 

\subsubsection{}\label{beth-omega-zero}
We assume now that $(\omega_\flat)|_W=0$,
which means that the form $\xi$ on $\lambda_\flat Q$ is zero, 
and we are going to prove  \textnormal{(3) (a)} and
\textnormal{(3) (b)} under this assumption. We will use the fact that these statements
hold whenever $\dim Q<n$ (\ref{case-dimp-small}), that they can be checked cellwise on $Q$ %FL
(\ref{arguing-cellwise}), that they can be proved after an affine change of coordinates 
(\ref{affine-change}), and that $\mathrm J$ acts trivially on $\mathsf A^{n,n}$; and then
we will ultimately rely on the estimates in 
\ref{proof-bounded}.

We want to prove
that 
\[\std\left(\int_{\left(\loga \abs f\right)\inv(\Pi_\epsilon)}\abs \omega\right)
\longrightarrow 0\]
when $\epsilon\longrightarrow 0$. By considering a cell
decomposition of $Q$
and using  \ref{arguing-cellwise}, 
we reduce to the case where $Q$ is a cell. 
If $\dim Q<n$
we already know that the the required statement holds (\ref{case-dimp-small}); we can thus assume
that $\dim Q=n$. And in view of \ref{affine-change}
we are allowed to perform an affine change of the coordinates indexed by $E$
with integral linear part; hence we can assume
that there exists a subset $E_0$
of $E$ of cardinality $n$ such that 
$Q$ is contained in the subspace defined by the equations
$x_i=0$ for $i$ running through $E\setminus E_0$. 
The assumption that $\xi=0$ now simply
means
that $\phi_{E_0,E_0}|_Q=0$.

We fix two subsets $I$ and $J$ of $E$, both
of cardinality $n$. Let 
$\omega_{I,J}$ be the form $\phi_{I,J}
\left(\loga \abs{f_1},\ldots,\loga \abs{f_m}\right)
\dl\abs{f_I}\wedge \dArg f_J$. 
It suffices to prove that 
\[\std\left(\int_{\left(\loga \abs f\right)\inv(\Pi_\epsilon)}
\abs{\omega_{I,J}}\right)\longrightarrow 0\]
when $\epsilon\longrightarrow 0$.

\paragraph{The case where $I=J=E_0$}
We then
have $\phi_{IJ}|_Q=0$. Let $P'$ be the pre-image of $\partial Q$ on $P$.
Since $\phi_{I,J}|_Q=0$ we have 
\[\std\left(\int_{\left(\loga \abs f\right)\inv(\Pi_\epsilon)}\abs{\omega_{I,J}}\right)
=
\std\left(\int_{\left(\loga \abs f\right)\inv(( P'_R)_\epsilon)}\abs{\omega_{I,J}}\right)\]
for all $\epsilon$, and since $\dim \partial Q<n$ the result  follows from \ref{case-dimp-small}.

\paragraph{The case where $I\neq E_0$}\label{I-not-1n}
Choose $i\in I\setminus E_0$. Then since $x_i$ vanishes
identically on $P$ we have for every $\epsilon$
\[\abs {f_i}\left(\left(\loga \abs f\right)\inv(\Pi_\epsilon)\right)
\subset \left[\abs t^{2\epsilon},\abs t^{-2\epsilon}\right].\]
Therefore
there exists some positive standard real number $A$ such that 
\[\abs{f_I}\left(\left(\loga \abs f\right)\inv(\Pi_\epsilon)\right)
\subset \left[\abs t^{2\epsilon},\abs t^{-2\epsilon}\right]^{\{i\}}
\times  \left[\abs t^A,\abs t^{-A}\right]^{I\setminus \{i\}}\]
for $\epsilon$ small enough (see \ref{proof-bounded}). In view of inequality
(\ref{equation-reference})
of \loccit, it follows that \[\std\left(\int_{\left(\loga \abs f\right)\inv(\Pi_\epsilon)}
\abs{\omega_{I,J}}\right)\longrightarrow 0\]
when $\epsilon\longrightarrow 0$.

\paragraph{The case where $J\neq E_0$}
Since the operator $\mathrm J$ acts trivially on $\mathsf A^{n,n}$
we have 
\begin{align*}
\omega_{I,J}&=\mathrm J(\omega_{I,J})\\
&=(-1)^n\phi_{I,J}
\left(\loga \abs{f_1},\ldots, \loga \abs f_m\right)
\dArg f_I\wedge \dl \abs{f_J}\\
&=(-1)^{n^2+n}\phi_{I,J}
\left(\loga \abs{f_1},\ldots, \loga \abs f_m\right)
\dl\abs{f_J}\wedge \dArg f_I\\
&=\phi_{I,J}
\left(\loga \abs{f_1},\ldots, \loga \abs f_m\right)
\dl \abs{f_J}\wedge \dArg f_I.
\end{align*}
Hence we reduce to the case considered in \ref{I-not-1n}. 

\subsubsection{Proof of
\textnormal{(3) (a)} and 
\textnormal{(3) (b)}
in the general case}\label{9.1.11}
Now comes the core of our proof; this is the only step in which one uses the actual definition of the non-archimedean integrals (the former ones used only basic properties like additivity or obvious norm estimates). Using once again the flexibility allowed by the former steps (which enables us to argue cellwise, see  \ref{arguing-cellwise}; or to modify the explicit writing of $\omega$, provided $(\omega_\flat)|_W$ remains unchanged, see
\ref{beth-omega-zero}), we will simplify slightly our assumptions, and then reduce to the case in which the integral $\int_W \omega_\flat$ can be computed by an explicit formula. The latter involves a classical real integral and the degree $d$ of an étale map between Berkovich spaces over some skeleton $\Sigma$, and the main point of our reasoning consists in interpreting this degree
$d$  in the non-standard archimedean world; this is achieved by showing that our étale map also has degree $d$ above ``sufficiently many" $C$-points (over which the degree is now simply the naive one, namely the cardinality of the fibers, which makes sense in our non-standard archimedean world as well).

By considering a cell
decomposition of $Q$
and using  \ref{arguing-cellwise}, 
we reduce to the case where $Q$ is a cell. 
If $\dim Q<n$
we already know that the required statement holds (\ref{case-dimp-small}); we can thus assume
that $\dim Q=n$. And in view of \ref{affine-change}
we are allowed to perform an affine change of the coordinates indexed by $E$
with integral linear part, we can assume
that there exists a subset $E_0$
of $E$ of cardinality $n$ such that 
$Q$ is contained in the subspace defined by the equations
$x_i=0$ for $i$ running through $E\setminus E_0$. 
Otherwise said, $Q=Q_0\times\{0\}^{E\setminus E_0}$ for
some convex polyhedron $Q_0$ of $\R^{E_0}$.
Since $\dim Q=n$
by our assumption, $\dim Q_0=n$.
Now $\xi|_{\lambda_\flat Q}$ can be written
\[\frac1{\lambda_\flat^n}\phi\left(\frac{x_j}{\lambda_\flat}\right)_{j\in E_0}\di x_{E_0}\wedge\dc x_{E_0}\]
(with $\phi$ smooth).
Set
\[\omega'=\phi
\left(\loga \abs{f_j}\right)_{j\in E_0}
\dl \abs{f_{E_0}} \wedge \dArg{f_{E_0}}
\]
and
\[\omega'_\flat=\phi
\left(\logb \val{f_j}\right)_{j\in E_0}
\di \logb \val{f_{E_0}} \wedge \dc \logb \val {f_{E_0}}.\]
Then $(\omega_\flat- \omega'_\flat)|_W=0$, and 
 in view of \ref{beth-omega-zero} 
 this implies that 
 \[\std\left(\int_{\left(\loga \abs f\right)\inv(\Pi_\epsilon)}
\abs{\omega-\omega'}\right)\longrightarrow 0\]
when $\epsilon\longrightarrow 0$.
We can thus replace $\omega$ with $\omega'$, hence reduce to the case where $\omega$ is of the form
\[\omega=\phi
\left(\loga \abs{f_j}\right)_{j\in E_0}
\dl \abs{f_{E_0}} \wedge \dArg f_{E_0}.\]

Let $\mu \colon V\to \gm^{E_0}$ be the map induced by the functions $f_j$
for $j\in E_0$.
Since $\dim Q_0=n$
the tropical dimension of $f_{E_0}$ is $n$, which forces $\mu$ to be dominant, hence generically
étale, because both schemes involved are integral of the same
dimension and the ground field is of characteristic zero. 
Let $\mathscr Z$ be a proper Zariski-closed subset of $\gm^{E_0}$
such that $\mu$ is finite étale over the open complement of $\mathscr Z$. 

Let $D$ be the affinoid domain of $\gm^{E_0,\mathrm{an}}$
defined by the condition $\logb \val T\in Q_0$ and let $D'$ be the open subset
of $D$ defined by the condition 
$\logb \val T\in\mathring Q_0$. Let also $\mathrm{sk}$
denote the canonical homeomorphism between $\R^{E_0}$ and the skeleton of $\gm^{E_0,\mathrm{an}}$. 
The images under $\logb \val {f_{E_0}}$ of the boundary of $D$ and the image of $\mathscr Z\an$
under $\logb \val T$ are of dimension $<n$. Since we can argue cellwise 
on $Q$ (\ref{arguing-cellwise})
we may thus assume the following: 
\begin{itemize}[label=$\diamond$]
\item $(\logb \val {f_{E_0}})(\partial W)\subset \partial Q_0$; 
\item the morphism $W\times_{\gm^{E_0,\mathrm{an}}}D'\to D'$ is finite étale. 
\end{itemize}
These two conditions imply that $\mu|_W$ is finite étale above $D'$; since the latter is connected
(it admits a deformation retraction to $\mathrm{sk}(\lambda_\flat\mathring Q_0)$),
the degree of $\mu|_W$ over $D'$
is constant; let us denote it by $d$. The map $\mu|_V$ is in particular
finite and flat of degree $d$
above every point of $\mathrm{sk}(\lambda_\flat \mathring Q_0)$, whence the equalities
\begin{align*}
\int_W \omega_\flat&= (-1)^{n(n-1)/2}\frac d{\lambda_\flat^n}\int_{\lambda_\flat Q_0}
\phi\left(\frac{x_j}{\lambda_\flat}\right)_{j\in E_0}
\d {x_{E_0}}\\
&=(-1)^{n(n-1)/2}d\int_{Q_0}\phi(x_j)_{j\in E_0}\d {x_{E_0}}
\end{align*}
and
\begin{align*}
\int_W \val{\omega_\flat} &=\frac d{\lambda_\flat^n}\int_{\lambda_\flat Q_0}\left|\phi\left(\frac{x_j}{\lambda_\flat}\right)_{j\in E_0}
\right|\d x_{E_0}\\
&=d\int_{Q_0}\abs{\phi(x_j)}_{j\in E_0}\d x_{E_0}
\end{align*}

\paragraph{}
By construction, every point of $D'(C)$ has $d$ pre-images under $\mu$
in $W(C)$. We would like to exploit this fact
in the non-standard archimedean setting; the point is that $D'(C)$ and $W(C)$ are
{\sc acvf}-definable, but not {\sc rcf}-definable;
so we will first have to ``approximate" them by {\sc rcf}-definable subsets for which this statement remains true.

Let $\mathfrak n$ be the set of negligible
elements of $R$. Let $\eta$ be a positive standard real number
and set $Q_\eta=\mathring Q_0\setminus (\partial Q_0)_\eta$.
Let $x\in \left(\loga \abs T\right)\inv(Q_{\eta, R})$. The point
$x$ belongs to $(\logb \val T)\inv(\mathring Q_0)$, hence
the intersection
\[\mu\inv(x)\cap \left(\loga \abs f\right)\inv(\Pi+\mathfrak n^\ell)=\mu\inv(x)\cap \left(\logb \val f\right)\inv(P)
=\mu\inv(x)\cap W\]
has exactly $d$ elements.
Let $m(x)$ and $M(x)$ be respectively the
greatest lower bound and the least upper bound 
of the set $\Theta$ of those $u\in [1,\abs t\inv]$ such that
\[\mu\inv(x)\cap  \abs f\inv(\abs t^{-\Pi}\cdot [u\inv, u])\]
has exactly $d$ elements. Since $\Theta$ is definable, if follows from the above
that
\[\std\left(\loga m(x)\right)=0\;\;\text{and}\;\;\std\left(\loga M(x)\right)>0.\]
Now $m$ and $M$ are definable functions;  as a consequence, 
the
greatest lower bound of $M$ on $\left(\loga \abs T\right)\inv(Q_{\eta, R})$
is equal to $\abs t^{B(\eta)}$ for some $B(\eta)$ with negative standard part, and the least
upper bound
of $m$ on $\left(\loga \abs T\right)\inv(Q_{\eta, R})$ is equal to
$\abs t^{b(\eta)}$ for some negative
negligible $b(\eta)$. 

\paragraph{}
Let $\delta$ be a positive real number. Choose $\eta$
such that the volume of $(\partial Q_0)_{2\eta}$ is smaller than $\delta$. Let $\epsilon$ be a positive
real number such that $\epsilon<\min(B(\eta),\eta)$.
Let $\Pi'$ be the subset of $\Pi_\epsilon$ consisting of points whose
projection to the variables in $E_0$ belongs to $Q_\eta$, and let $\Pi''$
be the complement of $\Pi'$ in $\Pi_\epsilon$. 
One has
\[\int_{\left(\loga \abs f\right)\inv(\Pi_\epsilon)}\omega
=\int_{\left(\loga \abs f\right)\inv(\Pi')}\omega+
\int_{\left(\loga \abs f\right)\inv(\Pi'')}\omega.\]
It follows
from inequality (\ref{equation-reference})
of \ref{proof-bounded}
that there exists a positive standard real number $M$ (independent of $\delta, \eta, \epsilon....$)
such that $\int_{\left(\loga \abs f\right)\inv(\Pi'')}\abs \omega\leq M\mathrm{Vol}((\partial Q_0)_{2\eta})
\leq M\delta$. 

Now since $b(\eta)<\epsilon<B(\eta)$ the map $\mu$ induces a $d$-fold covering  
\[\left(\loga \abs f\right)\inv(\Pi')\longrightarrow \left(\loga \abs T
\right)\inv(Q_{\eta,R}),\]
so
\begin{align*}
\int_{\left(\loga \abs f\right)\inv(\Pi')}\omega
&=
\frac d{\lambda^n}
\int_{\left(\loga \abs T\right)\inv(Q_{\eta,R})}
\phi\left(\loga \abs{T} \right)\dl \abs T\wedge
\dArg{T}\\
&=(-1)^{n(n-1)/2}d\int_{Q_{\eta,R}}\phi(x_j)_{j\in E_0}\d x_{E_0}\\
&=(-1)^{n(n-1)/2}d\int_{Q_\eta}\phi(x_j)_{j\in E_0}\d x_{E_0}.
%&=&(-1)^{n(n-1)/2}(2\pi)^nd\int_{P_0}\phi(x_1,\ldots, x_n)\d x_1\wedge\ldots \wedge \d x_n.
\end{align*}
Therefore
\begin{align*}
\left| \int_{\left(\loga \abs f\right)\inv(\Pi')}\omega
-\int_W\omega_\flat\right|&\leq \sup_{P_0}\abs \phi d\mathrm{Vol}(Q_0\setminus Q_\eta)\\
&\leq d\mathrm{Vol}((\partial Q_0)_{2\eta}) \sup_{Q_0}\abs \phi\\
&\leq \delta d\sup_{Q_0}\abs \phi.\end{align*}
Hence
\[\left|\int_{\left(\loga \abs f\right)\inv(\Pi_\epsilon)}\omega
-\int_W\omega_\flat\right|\leq \delta(M+d\sup_{Q_0}\abs \phi).\]
One shows exactly in the same way that
\[\left|\int_{\left(\loga \abs f\right)\inv(\Pi_\epsilon)}\abs \omega
-\int_W\abs{ \omega_\flat}\right|\leq \delta(M+d\sup_{Q_0}\abs \phi).\]
We thus have proved that 
\begin{align*}
\std\left(\int_{\left(\loga \abs f\right)\inv(\Pi_\epsilon)}\omega\right)
&\longrightarrow \int_{\left(\logb \val f)\right)\inv(\std(\Pi))}\omega_\flat\\
&\text{and}\\
\std\left(\int_{\left(\loga \abs f\right)\inv(\Pi_\epsilon)}\abs\omega\right)
&\longrightarrow \int_{\left(\logb \val f)\right)\inv(\std(\Pi))}\val{\omega_\flat}.
\end{align*}
when the standard positive real number $\epsilon$ tends to zero. 

\subsection{Construction of the map $\omega\mapsto \omega_\flat$}
\label{construction-map}
It
is clear that there is at most one such morphism of sheaves.
We are going to prove that there is actually one
by using our comparison theorem
for integrals and the fact that forms are naturally embedded into currents on
Berkovich spaces.
Let $p$ and $q$ be two integers. 
Let 
$U$ be a Zariski-open subset of
$X$.
Let 
$\omega$ be a 
section of $\mathsf A^{p,q}$
on $U$
that can be written
\[\omega=\sum_{\abs I=p,\abs J=q}
\phi_{I,J}\left(\loga \abs{f_1},\ldots,\loga \abs{f_m}
\right)\dl\abs{f_I}\wedge \dArg f_J
\]
with $f_i$ regular functions on $U$
and $\phi_{I,J}$ an $(I\cup J)$-vanishing
reasonably smooth function in
$\mathscr S^{I,J,(f_i)}$
for each $(I,J)$
(we shall say for short that $\omega$ is \emph{tropical}
on $U$).

Let $\omega_\flat$
be the section 
\[\sum_{I,J}\phi_{I,J}\left(\logb \val{f_1},
\ldots, \logb \val{f_m}\right)\di \logb \val{f_I}\wedge\dc \log \val{f_J}\]
of $\mathsf B^{p,q}$ on $U$.
It suffices to show that $\omega_\flat$
only depends on $\omega$,
and not on the particular way we have written it.
One immediately reduces to
proving that $\omega_\flat=0$ if $\omega=0$;
for that purpose we
suppose that $\omega_\flat \neq 0$,
and we are going to prove that $\omega\neq 0$.
Since $\omega_\flat\neq 0$ and since $U\an$ is boundaryless, there exists a smooth compactly supported
$(n-p,n-q)$ form $\eta$ on $U\an$ such that $\int_{U\an}\omega_\flat\wedge \eta\neq 0$ (\cite{chambertloir-d2012}, Cor. 4.3.7). 
Every point of $U\an$ has a basis of affinoid neighbourhoods $V$ having the following properties:

\begin{itemize}[label=$\diamond$]
\item the restriction $\eta|_V$
can be written 
\[\sum_{\abs I=n-p,\abs J=n-q}
\psi_{I,J}\left(
\logb \val{g_1}, \ldots, \logb \val{g_\ell}\right)
\di \log \val {g_I}\wedge \dc \log \val{ g_J}\]
with $g_i$ regular functions on $V$
and  $\psi_{I,J}$ compactly supported
smooth functions on $\R^\ell$. 
\item The domain $V$ is a Weierstra\ss~domain
of $\Omega\an$ for some open subscheme $\Omega$ of $U$
(see \ref{aff-neighb-weierstr}). 
\end{itemize}

Then we can find such a $V$ with $\int_V\omega_\flat \wedge \eta\neq 0$. Since $V$ is a Weierstra\ss~domain in $\Omega\an$, 
and since $\eta|_V$ does not change if we replace each $g_i$ by a function having the same norm on $V$
(\cite{chambertloir-d2012}, Lemme 3.1.10),
we can assume by approximation that each of the functions $g_i$ comes from a function belonging to $\mathscr O(\Omega)$,
which we still denote
by $g_i$. Then by replacing $\Omega$ by the intersection of the sets $D(g_i)$, we can assume that $g_i\in \mathscr O(\Omega)\gpm$ for all $i$. 

Now set \[\eta^\sharp=
\sum_{\abs I=n-p,\abs J=n-q}
\psi_{I,J}\left(
\loga \abs{g_1},\ldots, \loga \abs{g_\ell}\right)
\dl \abs {g_I}\wedge \dArg {g_J}.\]
This is a section of $\mathsf A^{n-p,n-q}$
on $\Omega$. By \ref{proof-integration-compatible}
the integral $\int_V\omega_\flat \wedge \eta$ can be expressed as a limit of standard parts of integrals of $\omega|_\Omega \wedge \eta^\sharp$
on suitable definably compact
semi-algebraic subsets of $\Omega(C)$. Then these integrals cannot be all equal to zero, which implies that 
$\omega|_\Omega \wedge \eta^\sharp\neq 0$, and a fortiori that $\omega \neq 0$. 
\textbf{We thus are done with the proof in the particular
setting of \ref{setting-omega-explicit}}.

\subsection{Proof of (3)}
We are now going to prove (3) in the general case. 
The reasoning is tedious but rather formal; it uses as a crucial input the particular case handled above in
\ref{proof-integration-compatible},
together with the additivity of the integrals in both settings. 

For all $i$ we set $P_i=\std(\Pi_i)$
and $W_i=\left(\logb \val {g_i}\right)\inv(P_i)\subset V_i\an$; 
we also set $W=\bigcup_i W_i$.

\subsubsection{Reduction to the case where $\Pi_i=P_{i,R}$
for all $i$}
Assume that (3)
holds for $(P_{i,R})_i$. 
Since $\std(\Pi_i)=\std(P_{i,R})$,
there exists a positive negligible element
$a$
such that $\Pi_i\subset P_{i,R,a}$ and $P_{i,R}\subset \Pi_{i,a}$
for every $i$. 
Let $\epsilon$ be a standard positive real number. 
By the above 
\[P_{i,R,\epsilon/2}\subset \Pi_{i,\epsilon}\subset P_{i,R,2\epsilon}\]
for all $i$. 
Then it follows from Remark \ref{rem-limit-moregeneral}
that statements (3)(a) and (3)(b) hold for $(\Pi_i)$. 

We then have for all standard $\epsilon>0$ and all $i$ 
\[\bigcup_i \Pi_{i,\epsilon}\setminus \bigcup_i \Pi_{i,\alpha+a}\subset \bigcup_i P_{i,R,2\epsilon}\setminus
\bigcup_i P_{i,R,\alpha},\]
so (3)(c) holds for $(\Pi_i)_i$ with the negligible element $\alpha+a$ instead of $\alpha$. 
We henceforth assume from now on that $\Pi_i=P_{i,R}$
for all $i$.

\subsubsection{}
Fix an index $i$. Let $x$ be a point of $W_i$. There exists a Zariski-open subset $\Omega$ of $V_i$
on which $\omega$ is tropical, and such that $x\in \Omega\an$. The point $x$ has a Weierstra\ss ~neighborhood
$\Omega'$ in $\Omega\an$; by construction, $\Omega'\cap W_i$ is of the form 
$(\logb \val h)\inv(Q)$ for some family $h=(h_1,\ldots, h_N)$ of regular functions on $\Omega'$
and some pseudo-polyhedron $Q$ of
$(\R\cup\{-\infty\})^N$. 

By compactness, it follows that there exists a finite family $(V_{ij})$ of Zariski-open subsets
of $V_i$ and, for each $(i,j)$, a finite family $h_{ij}=(h_{ijk})_{1\leq k\leq \ell_{ij}}$ of regular functions on $V_{ij}$ and a pseudo-polyhedron
$P_{ij}$ of $(\R\cup \{-\infty \})^{\ell_{ij}}$ such that the following hold: 
\begin{itemize}[label=$\diamond$]
\item for each $(i,j)$, the form $\omega$ is tropical on $V_{ij}$; 
\item $W_i=\bigcup_j W_{ij}$ with $W_{ij}=(\logb \val{h_{ij}})\inv(P_{ij})$. 
\end{itemize}

We set $\Pi_{ij}=P_{ij,R}$; for every non-empty set $I$ of pairs $(i,j)$
we set 
\begin{itemize}[label=$\diamond$]
\item $\ell_I=\sum_{(i,j)\in I}\ell_{ij}$;
\item $\Pi_I=\prod_{(i,j)\in I} \Pi_{ij}\subset (R\cup\{-\infty\})^{\ell_I}$; 
\item $P_I=\prod_{(i,j)\in I} P_{ij}\subset (\R\cup\{-\infty\})^{\ell_I}$; 
\item $V_I=\bigcap_{(i,j)\in I}V_{ij}\;\;\text{and}\;\;W_I=\bigcap_{(i,j)\in I}W_{ij}$.
\end{itemize}
We also denote by $h_I$ the concatenation of the functions $h_{ij}$ for $(i,j)\in I$; this is a family of
$\ell_I$
invertible functions on $V_I$
and $W_I=\left(\logb \val {h_I}\right)\inv(P_I)\subset V_I\an$. 

For every $I$, the form $\omega|_{V_I}$ is tropical. It follows therefore from
\ref{proof-integration-compatible}
that 
\begin{align}
\label{pi-omega}\std\left(\int_{\left(\loga \abs {h_I}\right)\inv(\Pi_{I,\epsilon})}\omega\right)
&\longrightarrow \int_{\left(\logb \val {h_I}\right)\inv(P_I)}\omega_\flat\\
\label{pi-absomega}\std\left(\int_{\left(\loga \abs {h_I}\right)\inv(\Pi_{I,\epsilon})}\abs
\omega\right)
&\longrightarrow \int_{\left(\logb \val {h_I})\right)(P_I)}\val{\omega_\flat}\end{align}
when the positive \emph{standard}
number $\epsilon$ tends to $0$, and
that there exists a positive negligible $\alpha\in R$
such that 
\begin{align}
\label{pi-zero}\std\left(\int_{\left(\loga \abs g\right)\inv(\Pi_{I,\epsilon}\setminus \Pi_{I,\alpha})}\abs
\omega\right)
&\longrightarrow 0\end{align}
when the positive \emph{standard}
number $\epsilon$ tends to $0$.

The equality $W=\bigcup_{(i,j)\in I} V_{ij}$ can be rewritten
\[\bigcup_i 
\underbrace{\left(\logb \val {g_i}\right)\inv(P_i)}_
{\text{understood as contained in }V_i\an}=
\bigcup_{(i,j)}\underbrace{\left(\logb \val {h_{ij}}\right)\inv(P_{ij})}_{\text{Understood
as contained in}\;V_{ij}\an}.\]

If $a$ is a small enough positive real number then 
for every $i,j$ the sets $\left(\logb \val {g_i}\right)\inv(P_{i,a})$
and $\left(\logb \val {h_{ij}}\right)\inv(P_{ij,a})$ are compact in view of assertion (1).
Hence
for $a$ small enough, the infimum $m(a)$ of all positive real numbers $b$ such that 
\[\bigcup_i \left(\logb \val {g_i}\right)\inv(P_{i,a})\subset
\bigcup_{(i,j)}\left(\logb \val {h_{ij}}\right)\inv(P_{{ij},b})\]
is well-defined. This is a {\sc doag}-definable function of $a$ that tends to zero when $a$ tends to zero.
It follows that there exists a positive rational number $\rho$ such that 

\[\bigcup_i \left(\logb \val {g_i}\right)\inv(P_{i,a})\subset
\bigcup_{(i,j)}\left(\logb \val {h_{ij}}\right)\inv(P_{{ij},\rho a})\]
for $a$ small enough. We can perform the same kind of reasoning for the converse inclusion, 
and by taking $\rho$ big enough we can thus assume that we also have 

\[\bigcup_i \left(\logb \val {g_i}\right)\inv(P_{i,\rho a})\supset
\bigcup_{(i,j)}\left(\logb \val {h_{ij}}\right)\inv(P_{ij,a})\]
for $a$ small enough.

We then have for all positive standard real number $a$ the inclusions
\[\bigcup_i
\left(\loga \abs {g_i}\right)\inv(\Pi_{i,a})\subset
\bigcup_{(i,j)}\left(\loga \abs {h_{ij}}\right)\inv(\Pi_{ij,2\rho a})\]
and
\[\bigcup_i\left(\loga \abs {g_i}\right)\inv(\Pi_{i,2\rho a})\supset
\bigcup_{(i,j)}\left(\loga \abs {h_{ij}}\right)\inv(\Pi_{ij,a}).\]
But then by a definability
argument (using \ref{sss-tp-definable}),
there exist a positive negligible 
element $\beta\in R$ and an element $\gamma\in R$ with positive
standard part such that the
above inclusions hold for all elements $a\in R$ with
$\beta\leq a\leq \gamma$. 

By the same kind of arguments,
we can increase $\beta$ and $\rho$
and decrease $\gamma$ 
so that we have for all $I$ and
all $a\in [\beta,\gamma]$ the inclusions
\[\bigcap_{{(i,j)}\in I}
\left(\loga \abs {h_{ij}}\right)
\inv(\Pi_{ij,a})\subset
\left(\loga \abs {h_I}\right)\inv(\Pi_{I,\rho a})\]
and
\[\bigcap_{(i,j)\in I}\left(\loga \abs {h_{ij}}\right)\inv
(\Pi_{ij,\rho a})\supset
\left(\loga \abs {h_I}\right)\inv(\Pi_{I,a}).\]

Together with (\ref{pi-omega}), (\ref{pi-absomega}) and (\ref{pi-zero}) above
and with the additivity of both the archimedean and the Berkovich integrals, this ends the proof of (2). 

\subsection{End of the proof}
It remains to show \textnormal{(1)} and \textnormal{(4)}. 
The proofs essentially consist in standard computations, once granted the existence of our map of complexes and the comparison theorems (3) (a), (b) and (c) for integrals.

We use the assumptions of (1).
Choose a finite open affine cover $(U_i)$ of $U$. For every $i$, let 
$(f_{ij})_j$ be a finite generating family of the $C$-algebra $\mathscr O_X(U_i)$. 
By our assumption on the support of $\omega$ and by Lemma
\ref{lem-compact-open-cover}, there exists $A\in \R$ such that 
the $\omega$ is zero outside the set
\[E_A:=\bigcup_i \left\{x\in U_i(C), \loga\abs {f_{ij}(x)}\leq A\;\text{for all}\;j\right\}.\]
We also set 
\[E_{A,\flat}=\bigcup_i \left\{x\in U_i\an, \logb \val{f_{ij}(x)}\leq A\;\text{for all}\;j\right\}.\]

\subsubsection{Proof of \textnormal{(1)}}\label{sss-compactsupp-compatible}
We are going to prove that $\omega_\flat$ is zero outside $E_{A,\flat}$, 
which will show that it is compactly supported. 

Let $y$ be a point of $U\an\setminus E_{A,\flat}$. The point $y$ belongs to $U_i$ for some $i$. Let 
us choose a neighborhood $V$ of $y$ in $U_i\an \setminus E_{A,\flat}$ of the form
$\logb \val g\inv(P)$ where $g=(g_1,\ldots, g_m)$ is a finite family of regular functions on $U_i$ and where 
$P\subset (\R\cup\{-\infty\})^m$ is a product of intervals, each of which is either of the form $(\lambda, \mu)$ 
or of the form $(-\infty, \mu)$. Up to shrinking $P$ we can assume
that for some $\epsilon>0$ the pre-image 
$\logb \val g\inv(P+[0,\epsilon)^m)$ still avoids $E_{A,\flat}$. Let $\phi$ be a 
reasonably smooth function on $(\R\cup\{-\infty\})^m$ whose support is contained in $P$, which does not vanish 
at $g(y)$, and which takes only non-negative values. We shall prove that the form 
$\phi(\logb \val g)\omega_\flat\in \mathscr A^{p,q}(U_i\an)$ is zero; this will ensure that $\omega_\flat$ vanishes around $y$ 
and thus imply our claim. 

Since $\logb \val g\inv(P+[0,\epsilon)^m)$ is contained in $U\an\setminus E_{A,\flat}$, 
the pre-image $\loga \abs g\inv(P_R)$ avoids $E$. As the support of $\phi$ is contained in $P$ and as
$\omega$ vanishes outside $E$, the form $\phi(\loga \abs g)\omega$
vanishes. But this form belongs to $\mathsf A^{p,q}(U_i)$ and its image in 
$\mathsf B^{p,q}(U_i)$ is precisely $\phi(\logb \val g)\omega_\flat$. The latter is thus zero, as announced. 

\subsubsection{Proof of \textnormal{(4)}}
Assume moreover that $p=q=n$ and let us prove (f) and (g).
It follows from (2), (3) (a) and (3) (b)
that if the standard positive $\epsilon$ is small enough then 
$\int_{E_{A+\epsilon}}\abs \omega$ is bounded, and that 
\[\std\left(\int_{E_{A+\epsilon}}\omega\right) \to \int_{E_{A,\flat}} \omega_\flat\]
and
\[\std \left(\int_{E_{A+\epsilon}}\abs \omega\right) \to \int_{E_{A,\flat}} \val{\omega_\flat}\]
when $\epsilon$ tends to zero (while remaining standard and positive). 

But since $\omega$ is zero outside $E_A$ we have

\[\int_{E_{A+\epsilon}}\omega =\int_{U(C)}\omega \;\;\text{and}\;\;\int_{E_{A+\epsilon}}
\abs \omega =\int_{U(C)}\abs \omega\]
for any $\epsilon$ as above. 
And since $\omega_\flat$ is zero outside $E_{A,\flat}$ by \ref{sss-compactsupp-compatible}, 
we have 

\[\int_{E_{A,\flat}}\omega_\flat =\int_{U\an}\omega_\flat\;\;\text{and}\;\;
\int_{E_{A,\flat}}\val{\omega_\flat} =\int_{U\an}\val{\omega_\flat}.\]
Assertion (4) follows immediately.~$\qed$

\bibliographystyle{smfalpha}
\bibliography{aducros}

\end{document}